\documentclass[11pt]{article}
\usepackage[table]{xcolor}
\usepackage[left=1in,top=1in,right=1in,bottom=1in,letterpaper]{geometry}
\usepackage{amsthm}
\usepackage{amsmath}
\usepackage{graphicx}
\usepackage{float}
\usepackage{subfig}
\usepackage{tikz}
\usepackage{amssymb}
\usepackage{enumerate}
\usepackage{algorithm}
\usepackage{algpseudocode}
\usepackage{multirow}
\usepackage[colorlinks,
          linkcolor = teal,		
          anchorcolor = blue,	
          urlcolor = teal,		
          citecolor = teal,
           ]{hyperref}
\usepackage{appendix}
\usepackage{MnSymbol}

\usepackage{enumerate}
\usepackage{enumitem}
\usepackage{url}
\usepackage[backend=biber,style=alphabetic,doi=false,url=false,maxbibnames=99,isbn=false]{biblatex}
\renewbibmacro{in:}{%
  \ifboolexpr{%
     test {\ifentrytype{article}}%
     or
     test {\ifentrytype{inproceedings}}%
  }{}{\printtext{\bibstring{in}\intitlepunct}}%
}
\DeclareFieldFormat[article]{volume}{\bibstring{jourvol}\addnbspace #1}
\DeclareFieldFormat[article]{number}{\bibstring{number}\addnbspace #1}

\renewbibmacro*{volume+number+eid}{%
  \printfield{volume}%
  \setunit{\addcomma\space}%
  \printfield{number}%
  \setunit{\addcomma\space}%
  \printfield{eid}}

\usepackage{pifont}
\newcommand{\xmark}{\ding{55}}%

\newtheorem{definition}{Definition}[section]
\newtheorem{lemma}{Lemma}[section]
\newtheorem{proposition}{Proposition}[section]
\newtheorem{remark}{Remark}[section]
\newtheorem{corollary}{Corollary}[section]

\newtheorem{thom}{Theorem}[section]

\newcommand\keywords[1]{\textbf{Keywords}: #1}
\newcommand\msc[1]{\textbf{Mathematical Subject Classifications}: #1}
\makeatletter

\newcommand{\Rmnum}[1]{\expandafter\@slowromancap\romannumeral #1@}
\makeatother

\newcommand{\etal}{et al. }

\begin{document}

\title{AdaBB: Adaptive Barzilai-Borwein Method \\ for Convex Optimization}

\author{Danqing Zhou\footnotemark[1]   \and Shiqian Ma\footnotemark[2]\and Junfeng Yang\footnotemark[1]}

\footnotetext[1]{Department of Mathematics, Nanjing University, \#22 Hankou Road, Nanjing, P. R. China.
Research supported by the National Natural Science Foundation of China (NSFC-12371301) and the Natural Science Foundation for Distinguished Young Scholars of Gansu Province (22JR5RA223). Email: zhoudanqing@smail.nju.edu.cn, jfyang@nju.edu.cn.}

\footnotetext[2]{Department of Computational Applied Mathematics and Operations Research, Rice University, Houston, USA. Research supported in part by NSF grants DMS-2243650, CCF-2308597, CCF-2311275 and ECCS-2326591, and a startup fund from Rice University. Email: sqma@rice.edu.}

\maketitle

\begin{abstract}
In this paper, we propose AdaBB, an adaptive gradient method based on the Barzilai-Borwein stepsize. The algorithm is line-search-free and parameter-free, and essentially provides a convergent variant of the Barzilai-Borwein method for general unconstrained convex optimization. We analyze the ergodic convergence of the objective function value and the convergence of the iterates for solving general unconstrained convex optimization. Compared with existing works along this line of research, our algorithm gives the best lower bounds on the stepsize and the average of the stepsizes.
Moreover, we present an extension of the proposed algorithm for solving composite optimization where the objective function is the summation of a smooth function and a nonsmooth function.
Our numerical results also demonstrate very promising potential of the proposed algorithms on some representative examples.
\end{abstract}

\noindent \keywords Adaptive Gradient Descent; Parameter-Free; Line-Search-Free; Automated Gradient Descent; Barzilai-Borwein Stepsize; Locally Lipschitz Gradient.

\noindent \msc 90C25

\section{Introduction} \label{intro}
In this paper, we propose an adaptive Barzilai-Borwein method (AdaBB) for solving unconstrained convex and smooth optimization problem
\begin{equation}
    \min_{x\in \mathbb{R}^n} f(x), \label{p}
\end{equation}
where $f:\mathbb{R}^{n}\rightarrow \mathbb{R}$ is convex with a locally Lipschitz gradient $\nabla f$. Remarkably, AdaBB is an automated gradient descent method -- it is line-search-free and parameter-free.

Our AdaBB method is closely related to the pioneering work by Malitsky and Mishchenko \cite{MM20} along the line of automated gradient descent method, as well as its follow-up works \cite{MM23,LTSP23,lTP23}. Recently, it draws great attention on how to design line-search-free and parameter-free gradient descent methods (GD) for solving \eqref{p}:
\begin{equation}\label{gd}
 x^{k+1} = x^{k} - \alpha_{k}\nabla f(x^{k}), \; k=0,1,2,\ldots,
\end{equation}
where $\{\alpha_k\}$ is a sequence of stepsizes and $x^0\in\mathbb{R}^n$ is a starting point. 
See \cite{GSW23,G23,LOZ23,AP23,AP232,LL23,DVR23} for a partial list of other works concerning this problem.
When $f$ is convex and has globally Lipschitz gradient with Lipschitz constant $L>0$, \eqref{gd} can take a fixed stepsize $\alpha_k=\alpha>0$. It is well known that setting $\alpha<2/L$ guarantees the convergence and the $O(1/K)$ convergence rate of GD \eqref{gd}. Drori and Teboulle \cite{DT14} first proposed the performance estimation problem (PEP) approach to study the complexity of GD when $\alpha\in(0,1/L]$. The resulting complexity bound improved the classical one by a factor of two and was shown to be tight. In a recent work, Teboulle and Vaisbourd \cite{TV23} provided an elementary proof to reach this tight bound for $\alpha \in (0,1/L]$ and designed a dynamic stepsize $\alpha_k\in [1/L,2/L)$ to improve the tight convergence bound by a constant that converges to two as $k\rightarrow\infty$.

However, these results are not parameter-free, because they require the Lipschitz constant $L$. Moreover, the Lipschitz constant $L$ can be much larger than the one given by the local curvature information, which will lead to conservative stepsizes and slow convergence in practice.

Since the emergence of applications from large-scale machine learning, it has been an active research area on how to adaptively choose $\alpha_k$ in gradient descent and stochastic gradient descent (SGD) methods. Moreover, the computation of $\alpha_k$ should only be conducted using existing gradient information and should not involve expensive computation. In this sense, GD/SGD with line search are not considered adaptive methods. As explained in a recent paper by Malitsky and Mishchenko \cite{MM23}, ``(a method is called adaptive), if it automatically adopts a stepsize to (its) local smoothness without additional expensive computation and the method does not deteriorate the rate of the original method in the worst case'', and it was pointed out in \cite{MM23} that AdaGrad \cite{DHS11, MS10} is not adaptive. 

In this paper, we focus on GD methods that adaptively compute the stepsizes $\alpha_k$ and are parameter-free. Therefore, this kind of methods are truly automated. The earliest work in this direction is the Barzilai-Borwein (BB) method \cite{BB88}, which adaptively computes stepsize using two consecutive gradients. There are two BB stepsizes in the literature:
\begin{align}
    \textrm{(Long BB)} \qquad & \beta_k = \frac{\|x^{k}-x^{k-1}\|^2}{\langle \nabla f(x^{k}) - \nabla f(x^{k-1}), x^{k}-x^{k-1} \rangle},  \label{betak} \\
    \textrm{(Short BB)} \qquad & \lambda_{k} = \frac{\langle \nabla f(x^{k}) - \nabla f(x^{k-1}), x^{k}-x^{k-1} \rangle}{\|\nabla f(x^{k})- \nabla f(x^{k-1})\|^2}. \label{lamk}
\end{align}
The BB method, albeit guaranteed to converge only for some special classes of problems, has been very influential in nonlinear optimization. The first time that BB was adopted to SGD was due to Tan \etal \cite{Tan-SVRGBB}, where the authors proposed SGD-BB and SVRG-BB and proved the convergence of the latter one under the assumption that the objective function is strongly convex. We will give a more detailed survey of the BB method in the next section. Recently, Malitsky and Mishchenko have made a significant breakthrough in parameter-free adaptive GD \cite{MM20, MM23}.
Specifically, starting with initial point $x^0\in\mathbb{R}^n$ and initial stepsize $\alpha_0>0$, the AdGD algorithm proposed in \cite{MM20, MM23} first computes $x^1 = x^0-\alpha_0\nabla f(x^0)$, and then updates the iterates in the $k$-th iteration as follows for $k=1,2,\ldots$,
\begin{align}\label{AdGD}
    \alpha_k & = \min\left\{\alpha_{k-1}\sqrt{1+\theta_{k-1}}, 1/(\sqrt{2} L_k) \right\}, \mbox{ where } L_k = \frac{\|\nabla f(x^k) - \nabla f(x^{k-1})\|}{\|x^k-x^{k-1}\|}, \nonumber\\
    x^{k+1} & = x^k - \alpha_k\nabla f(x^k), \text{~~and update $\theta_k$ via~~}     \theta_k   = \alpha_k/\alpha_{k-1}. 
\end{align}
We point out that \eqref{AdGD} is proposed in \cite{MM23}, and the algorithm proposed in \cite{MM20} replaces the factor $\sqrt{2}$ with $2$ in the updating formula of $\alpha_k$. Here we refer \eqref{AdGD} as AdGD because the second term in updating $\alpha_k$ allows larger stepsize.
This algorithm is guaranteed to converge for solving \eqref{p} for any $\alpha_0>0$ and $\theta_0=0$. Therefore, it is parameter-free. Note that $L_k$ estimates the local curvature information. Furthermore, in \cite{MM23} a variant of AdGD (we call it AdGD2) is proposed, which allows larger stepsizes.
Starting with initial point $x^0$ and initial stepsize $\alpha_0$, the AdGD2 algorithm first computes $x^1 = x^0-\alpha_0\nabla f(x^0)$, and then updates the iterates in the $k$-th iteration as follows for $k=1,2,\ldots$,
\begin{align}  \label{AdGD2}
    \alpha_k & = \min\left\{\sqrt{\tfrac{2}{3}+\theta_{k-1}}\alpha_{k-1},\frac{\alpha_{k-1}}{\scriptstyle{\sqrt{[2\alpha_{k-1}^2L_k^2-1]_{+}}}} \right\}, 
    \mbox{ where } L_k = \frac{\|\nabla f(x^k) - \nabla f(x^{k-1})\|}{\|x^k-x^{k-1}\|}, \nonumber \\
    x^{k+1} & = x^k - \alpha_k\nabla f(x^k), \text{~~and update $\theta_k$ via~~}     \theta_k   = \alpha_k/\alpha_{k-1}.  
\end{align}
where $[a]_+ = \max\{a,0\}$. The AdGDs \eqref{AdGD}-\eqref{AdGD2} are extended to adaptive proximal gradient methods by \cite{LTSP23, MM23,lTP23}. Moreover, Latafat \etal \cite{LTSP23,lTP23} used the BB stepsizes \eqref{betak} and \eqref{lamk} to estimate the local curvature information. Starting with initial point $x^0$, initial stepsize $\alpha_0$ and $\theta_0\geq 1$, the basic AdaPGM algorithm \cite[Alg 2.1]{LTSP23} first computes $x^1 = x^0 -\alpha_0\nabla f(x^0)$, and then updates the iterates as follows for $k=1,2,\ldots$,
\begin{align}\label{AdaPGM}
    \alpha_k & = \min\left\{\alpha_{k-1}\sqrt{1+\theta_{k-1}}, \; \frac{\alpha_{k-1}}{2\sqrt{[(\alpha_{k-1}/\beta_{k})(\alpha_{k-1}/\lambda_{k}-1)]_+}} \right\}, \nonumber\\
    x^{k+1} & = x^k - \alpha_k\nabla f(x^k),  \text{~~and update $\theta_k$ via~~}     \theta_k   = \alpha_k/\alpha_{k-1}.  
\end{align}
It is noted that both long BB \eqref{betak} and short BB \eqref{lamk} are used to estimate the local curvature information in \eqref{AdaPGM}. Notably, in a more recent work, Latafat \etal \cite{lTP23} introduced a unified framework, $\text{AdaPGM}^{\pi,r}$, that updates $\alpha_{k}$ through the formula
\[
\alpha_{k} = \min \left\{ \sqrt{\tfrac{1}{\pi}+\theta_{k-1}}, \; \sqrt{\tfrac{1-(r/\pi)}{[(1-2r)+\alpha_{k-1}^2L_k^2+2\alpha_{k-1}(r-1)/\beta_k]_{+}}} \right\} \alpha_{k-1}
\]
for any $\pi>r\geq \frac{1}{2}$. This modification gives more flexibility to balance the effects of two terms. Specifically, when opting for $r=\frac{1}{2}$ and $\pi=1$, it aligns with updates of AdaPGM \eqref{AdaPGM} but improves the second term by a factor of $\sqrt{2}$. This variant also guarantees a larger lower bound for the stepsize sequence, as we will discuss later.

{\bf Our contributions.} The main contributions of our paper lie in several folds. 
\begin{enumerate}[leftmargin=*]
    \item We propose a new adaptive algorithm, AdaBB, for solving \eqref{p}. There are two prominent features of our AdaBB algorithm: (a) We only use the short BB stepsize \eqref{lamk}, and we do not use the long BB stepsize \eqref{betak} and $L_k$ in \eqref{AdGD}. Note that the long BB stepsize $\beta_k$ can be removed from adaPGM \eqref{AdaPGM} but then $L_k$ needs to be used. (b) Our AdaBB has a very simple and intuitive connection with the BB method, and essentially provides a convergent variant of the BB method for general convex optimization.
    \item We answer an open question posed by Malitsky and Mishchenko in \cite{MM23} affirmatively. Specifically, Malitsky and Mishchenko posed an open question asking whether there exists an adaptive method in which the sum of the stepsizes $\sum_{i=1}^k\alpha_i$ is close to $k/L$ with readable proof, where $L$ denotes the local Lipschitz constant of $\nabla f$. Note that this indicates the average of the stepsizes is close to $1/L$. We prove that the sum of the stepsizes $\sum_{i=1}^k\alpha_i$ in our AdaBB algorithm is lower bounded by $(k-2+\sqrt{2})/L$, $\forall k\geq 1$, which can be further improved to $k/L$ with a suitably chosen initial stepsize.
    \item We prove that the stepsize $\alpha_k$ in our AdaBB is always lower bounded by $1/(\sqrt{2}L)$, i.e., $\alpha_k\geq 1/(\sqrt{2}L)$ for any $k\geq 1$. This also improves the existing results in \cite{MM20,MM23,LTSP23,lTP23}. See the detailed comparison in Table \ref{table:stepsize-comparison}. 
\end{enumerate}

{\bf Notation.} Throughout the paper, we assume the optimal solution set $\mathcal{X}^*$ of \eqref{p} is nonempty and denote its optimal value by $f_*$. We use $x^*$ to denote one element in $\mathcal{X}^*$. Let $\mathbb{R}_{+}$ be positive real line, and $|\mathcal{W}|$ be the cardinality of the set $\mathcal{W}$. We use the symbol $\times$ to represent the Cartesian product.
We can prove that the sequence $\{x^k\}$ generated by AdaBB is bounded and lies in a ball $B(x^*, R)$ whose center is $x^*$ and radius is $R$, which will be specified later. We assume that $f$ is $L$-locally smooth (or $\nabla f$ is $L$-locally Lipschitz) in $B(x^*,R)$, which is defined as
\begin{equation}\label{local-smooth-1}
    f(x)-f(y)-\langle \nabla f(y), x-y\rangle \geq \tfrac{1}{2L}\|\nabla f(x)-\nabla f(y)\|^2, \forall x,y \in B(x^*,R).
\end{equation}
According to \cite[Theorem 2.1.5]{N87}, this further implies $\|\nabla f(x)-\nabla f(y)\|\leq L\|x-y\|, \forall x,y \in B(x^*,R)$, and
\begin{equation}\label{local-smooth-2}
    \langle \nabla f(x)-\nabla f(y), x-y \rangle \geq {\tfrac{1}{L}}\|\nabla f(x)-\nabla f(y)\|^2, \forall x,y \in B(x^*,R).
\end{equation}
{We point out that these inequalities are not equivalent in when the points are restricted to a bounded convex set. A counterexample can be found in \cite{D18}. However, when $f$ is globally $L$-smooth in $\mathbb{R}^n$, they become equivalent (see \cite[Theorem 2.1.5]{N87}).}

{\bf Organization.} The rest of this paper is organized as follows.
In Section \ref{s2}, we present a simple version of our AdaBB algorithm and discuss its connection with the BB method. In Section \ref{main}, we present the full version of our AdaBB algorithm which allows more flexible choices of stepsizes. We also analyze the ergodic convergence of the function value error and the convergence of the iterates.
In Section \ref{explan}, we conduct an in-depth analysis of the stepsizes generated by our AdaBB algorithm. In particular, we provide lower bounds for both $\alpha_i$ and $\sum_{i=1}^k \alpha_i$.
In Section \ref{exten}, we extend our AdaBB algorithm to the case where the objective function is locally strongly convex, and to the composite case where the objective function is the summation of a smooth function and a nonsmooth function. Numerical experimental results are reported in Section \ref{ne} to illustrate the effectiveness of the proposed algorithms. Finally, we draw some concluding remarks in Section \ref{conclusions}.

\section{Our AdaBB Algorithm} \label{s2}

The main motivation of the BB method \cite{BB88} is to use a diagonal matrix ($\frac{1}{\eta_k} I$) to approximate the quasi-Newton matrix in the $k$-th iteration of the quasi-Newton method, where scalar $\eta_k >0$. In order to satisfy the secant equation, essentially we require $\frac{1}{\eta_k} s^k$ = $y^k$, where $s^k = x^k - x^{k-1}$ and $y^k = \nabla f(x^k) - \nabla f(x^{k-1})$. Since $s^k$ and $y^k$ are both $n$-dimensional vectors, it is impossible to find a scalar $\eta_k$ such that this linear equation holds. Therefore, one has to find $\eta_k$ that minimizes the residual, i.e.,
\[
\min_{\eta_k} \|s^k /\eta_k - y^k\|, \mbox{ or } \min_{\eta_k} \|s^k - \eta_k y^k \|,
\]
which leads to the two formulas given in \eqref{betak} and \eqref{lamk} (with $\eta_k$ replaced by $\beta_k$ and $\lambda_k$, respectively).

However, the naive BB method
\begin{equation}\label{BB-method}
x^{k+1} = x^k - \beta_k \nabla f(x^k), \quad \mbox{ or } \quad x^{k+1} = x^k - \lambda_k \nabla f(x^k)
\end{equation}
does not always converge. In fact, existing convergence results of \eqref{BB-method} have been mainly restricted to the special case where $f$ is a strongly convex quadratic function. In the original paper by Barzilai and Borwein \cite{BB88}, it is proved that the BB method \eqref{BB-method} converges $R$-superlinearly, if $f$ is strongly convex quadratic and $n=2$. When $f$ is strongly convex quadratic with a general dimensionality $n$, the BB method \eqref{BB-method} is proved to converge globally \cite{R93} and at an $R$-linear rate \cite{DL02}. However, when $f$ is not strongly convex quadratic function, i.e., when it is a general convex function, there exist counterexamples showing that the BB method \eqref{BB-method} can diverge \cite{BDH19}. To address this limitation, Rayden \cite{R97} incorporated the non-monotone line search technique from \cite{GLL86} to the BB method \eqref{BB-method} and proved its global convergence when $f$ is a general convex function.

Hence, it has been an open question whether there exists a simple variant of the BB method without line search that globally converges for general convex function. We answer this question affirmatively by proposing our AdaBB algorithm, which is described in Algorithm \ref{alg:AdaBB}.

\begin{algorithm}[!htb]
\caption{\textbf{Ada}ptive  \textbf{B}arzilai-\textbf{B}orwein Algorithm (AdaBB)} \label{alg:AdaBB}
\textbf{Input:} $x^0\in\mathbb{R}^n$, $\alpha_0 > 0$, $\theta_0 \geq 0 $
\begin{algorithmic}[1]
\State $x^1 = x^{0} - \alpha_0 \nabla f(x^0)$
\For{$k=1,2,\ldots,$}
\State $\lambda_{k} = \frac{\langle \nabla f(x^{k}) - \nabla f(x^{k-1}), x^{k}-x^{k-1} \rangle}{\|\nabla f(x^{k})- \nabla f(x^{k-1})\|^2}$
\If{$\lambda_k \geq \alpha_{k-1}$}\dotfill (Case i)
\State $\alpha_k = \sqrt{1+\theta_{k-1}}\alpha_{k-1}$, and $\theta_{k} = \frac{\alpha_{k}}{\alpha_{k-1}}$
\ElsIf{$\alpha_{k-1}/2 <\lambda_k < \alpha_{k-1}$} \dotfill (Case ii)
\State $\alpha_k = \lambda_{k}$, and
$\theta_{k} = \frac{2\alpha_{k}}{\alpha_{k-1}}-\frac{\alpha_{k}}{\lambda_{k}}$
\Else\dotfill (Case iii)
\State $\alpha_k = \frac{\lambda_{k}}{\sqrt{2}}$, and $\theta_{k} = \frac{\alpha_{k}}{\alpha_{k-1}}$
\EndIf
\State $x^{k+1} = x^{k} - \alpha_{k} \nabla f(x^{k})$
\EndFor
\end{algorithmic}
\end{algorithm}

This is our basic AdaBB algorithm. We will prove its convergence in the next section when we discuss a more general version of AdaBB. We also note that the per-iteration computational cost is almost the same as the BB method \eqref{BB-method}. We now give some intuitive explanation why AdaBB (Algorithm \ref{alg:AdaBB}) can overcome the drawbacks of the BB method. Note that BB method may diverge for general convex function because the stepsize is sometimes too aggressive. In AdaBB, we carefully design the stepsize so that if we find that the stepsize is too large in some iteration, then we use a smaller stepsize in the next iteration. This ensures that the risk of taking very large stepsizes is hedged so that the algorithm will not diverge. At the same time, when the stepsize is too small in some iteration, then we use a larger stepsize in the next iteration. This ensures that the stepsize is not always small to ensure a fast convergence. More specifically, the three ``if-else'' conditions in Algorithm \ref{alg:AdaBB} can be interpreted as follows.
\begin{itemize}
    \item (Case i), when $\lambda_k\geq\alpha_{k-1}$, it means that the stepsize $\alpha_{k-1}$ in the previous iteration is too small. So we set $\alpha_k$ to be larger than $\alpha_{k-1}$. This ensures that the stepsize is not always too small. The choice of $\theta_k$ will be clear from the convergence proof.
    \item (Case ii), when $\alpha_{k-1}/2<\lambda_k<\alpha_{k-1}$, it means that the BB stepsize is not too large and not too small. So we just use the BB stepsize $\lambda_k$ as the stepsize.
    \item (Case iii), when $\lambda_k\leq \alpha_{k-1}/2$, it means that the stepsize $\alpha_{k-1}$ in the previous iteration is too large. So we shrink the BB stepsize by $\sqrt{2}$ and use it as the stepsize for the current iteration. This ensures that the stepsize in the current iteration is not too large when the previous stepsize is large, to hedge the risk of divergence.
\end{itemize}

\section{The General Version of AdaBB and Convergence Analysis} \label{main}

In this section, we first present the general version of AdaBB in Algorithm \ref{agbb}, which offers more flexibility when choosing stepsize $\alpha_k$. Our basic AdaBB (Algorithm \ref{alg:AdaBB}) is a special case of the general version of AdaBB. We then analyze the convergence properties of this algorithm for solving \eqref{p}. 
\begin{algorithm}[!htb]
\caption{The General Version of AdaBB} \label{agbb}
\textbf{Input:} $x^0\in\mathbb{R}^n$, $\alpha_0 > 0$, $\theta_0 \geq 0 $
\begin{algorithmic}[1]
\State $x^1 = x^{0} - \alpha_0 \nabla f(x^0)$
\For{$k=1,2,\ldots,$}
\State $\lambda_{k} = \frac{\langle \nabla f(x^{k}) - \nabla f(x^{k-1}), x^{k}-x^{k-1} \rangle}{\|\nabla f(x^{k})- \nabla f(x^{k-1})\|^2}$
\If{$\lambda_k \geq \alpha_{k-1}$}\dotfill (Case i)
\State $\alpha_k = \sqrt{1+\theta_{k-1}}\alpha_{k-1}$, and $\theta_{k} = \frac{\alpha_{k}}{\alpha_{k-1}}$
\ElsIf{$\alpha_{k-1}/2 <\lambda_k < \alpha_{k-1}$}\dotfill (Case ii)
\State
\begin{eqnarray}\label{gen-AdaBB-alpha-1}
\alpha_k = \left\{\begin{array}{ll}
    \min \left\{ \sqrt{\frac{\lambda_{k}}{2(\alpha_{k-1}-\lambda_{k})}}, \, \sqrt{\frac{(1+\theta_{k-1})\lambda_{k}}{2\lambda_{k}-\alpha_{k-1}}}\right\}\alpha_{k-1}, & \textrm{(Option I),} \\
    \lambda_{k}, & \textrm{(Option II),}
\end{array}\right. \mbox{ and } \theta_{k} = \frac{2\alpha_{k}}{\alpha_{k-1}}-\frac{\alpha_{k}}{\lambda_{k}}
\end{eqnarray}
\Else\dotfill (Case iii)
\State
\begin{eqnarray}\label{gen-AdaBB-alpha-2}
\alpha_k = \left\{\begin{array}{ll}
    \sqrt{\frac{\alpha_{k-1}}{2(\alpha_{k-1}-\lambda_{k})}}\lambda_{k}, & \textrm{(Option I),} \\
    \frac{\lambda_{k}}{\sqrt{2}}, & \textrm{(Option II),}
\end{array}\right. \mbox{ and } \theta_{k} = \frac{\alpha_{k}}{\alpha_{k-1}}
\end{eqnarray}
\EndIf
\State $x^{k+1} = x^{k} - \alpha_{k} \nabla f(x^{k})$
\EndFor
\end{algorithmic}
\end{algorithm}
We note that Algorithm \ref{agbb} offers more choices for $\alpha_k$ in \eqref{gen-AdaBB-alpha-1} and \eqref{gen-AdaBB-alpha-2}. Choosing (Option II) in both \eqref{gen-AdaBB-alpha-1} and \eqref{gen-AdaBB-alpha-2} recovers the basic AdaBB (Algorithm \ref{alg:AdaBB}). Therefore, we only need to analyze the convergence of Algorithm \ref{agbb} and we will devote the rest of this section to it. The following lemma gives some immediate property of $\alpha_k$.

\begin{lemma}\label{alpha-property}
In both \eqref{gen-AdaBB-alpha-1} and \eqref{gen-AdaBB-alpha-2}, the stepsize $\alpha_k$ provided by (Option II) is less than or equal to that provided by (Option I).
\end{lemma}

\begin{proof}
    When $\alpha_{k}/2 < \lambda_{k+1} < \alpha_{k}$, we immediately have
    \[\sqrt{\frac{(1+\theta_{k})\lambda_{k+1}}{2\lambda_{k+1}-\alpha_{k}}} \alpha_{k} \geq \alpha_{k} \geq \lambda_{k+1}.
    \]
    Moreover, since $v(x):=\sqrt{\frac{\lambda_{k+1}x^2}{2(x-\lambda_{k+1})}}$ is decreasing for $x \in (\lambda_{k+1},2\lambda_{k+1}]$, we have
\[
v(\alpha_k)=\sqrt{\frac{\lambda_{k+1}}{2(\alpha_{k}-\lambda_{k+1})}}\alpha_{k}  \geq v(2\lambda_{k+1}) > \lambda_{k+1}.
\]
This proves that (Option II) is not larger than (Option I) in \eqref{gen-AdaBB-alpha-1}. The proof when $ \lambda_{k+1} \leq \alpha_{k}/2$ is trivial and thus omitted.
\end{proof}

We now define some important notation. In particular, we define
$M_{k}$ and $P_{k}$ for $k\geq 1$ as follows
\begin{equation}\label{def-Mk-Pk}
\left\{
\begin{array}{lll}
M_{k}=0 ,    &P_{k}= \frac{\alpha_{k}^{2}}{\alpha_{k-1}},     & {\textrm{if }\lambda_{k}\geq \alpha_{k-1}}, \\
M_{k}=\frac{\alpha_{k}^{2}}{\lambda_{k}\alpha_{k-1}}-\frac{\alpha_{k}^{2}}{\alpha_{k-1}^2},   &
P_{k}= \frac{2\alpha_{k}^{2}}{\alpha_{k-1}}-\frac{\alpha_{k}^{2}}{\lambda_{k}},   &
{\textrm{if }\frac{\alpha_{k-1}}{2}<\lambda_{k} <  \alpha_{k-1}},\\
M_{k}=\frac{\alpha_{k}^{2}}{\lambda_{k}^2}-\frac{\alpha_{k}^{2}}{\alpha_{k-1}\lambda_{k}},   &
P_{k}= \frac{\alpha_{k}^{2}}{\alpha_{k-1}},
& {\textrm{if } 0 <\lambda_{k}\leq \frac{\alpha_{k-1}}{2}}.
\end{array} \right. 
\end{equation}
For convenience, we also define
\begin{equation}\label{def-P0}
    P_0 = P_1-\alpha_0.
\end{equation}

The following lemma provides some useful inequalities for $M_{k}$ and $P_{k}$ defined in \eqref{def-Mk-Pk} and \eqref{def-P0}.

\begin{lemma} \label{mpi}
For $\alpha_k$ generated by Algorithm \ref{agbb}, it holds that $P_{k} = \alpha_{k}\theta_{k}$ for all $k \geq 1$, and 
$2M_{k+1}\leq 1$ and $P_{k+1}\leq P_{k}+\alpha_k$ for all $k \geq 0$.
\end{lemma}
\begin{proof}
First, for any $k \geq 1$, we have $P_{k} = \alpha_{k}\theta_{k}$ by the definitions of $\theta_k$ in Algorithm \ref{agbb} and $P_k$ in \eqref{def-Mk-Pk}. 
To show the remaining results, we consider three cases.
\begin{itemize}[leftmargin=*]
    \item Case (i): when $\lambda_{k+1}\geq \alpha_{k}$, by recalling $M_{k+1}=0$, $P_{k+1}=\alpha_{k+1}^2/\alpha_{k}$, and $\alpha_{k+1} = \sqrt{1+\theta_{k}}\alpha_{k}$, the desired results follow immediately.
    \item Case (ii): when $\alpha_{k}/2 < \lambda_{k+1} < \alpha_{k}$, from Lemma \ref{alpha-property}, we only need to prove the desired results for (Option I) of $\alpha_{k+1}$ in \eqref{gen-AdaBB-alpha-1}. In this case, we have
    \begin{align*}
    M_{k+1} & = \frac{\alpha_{k+1}^{2}}{\lambda_{k+1}\alpha_{k}}-\frac{\alpha_{k+1}^{2}}{\alpha_{k}^2} \leq \frac{\lambda_{k+1}\alpha_{k}^2}{2(\alpha_{k}-\lambda_{k+1})}\Big(\frac{1}{\lambda_{k+1}\alpha_{k}}-\frac{1}{\alpha_{k}^2}\Big)=\frac{1}{2},\\
    P_{k+1} & =  \frac{2\alpha_{k+1}^{2}}{\alpha_{k}}-\frac{\alpha_{k+1}^{2}}{\lambda_{k+1}} \leq  \frac{(1+\theta_{k})\lambda_{k+1}\alpha_{k}^2}{2\lambda_{k+1}-\alpha_{k}}\Big(\frac{2}{\alpha_{k}}-\frac{1}{\lambda_{k+1}}\Big)=\alpha_{k}+P_{k}.
    \end{align*}
    \item Case (iii): when $\lambda_{k+1}\leq \alpha_{k}/2$, again from Lemma \ref{alpha-property}, we only need to prove the desired results for (Option I) of $\alpha_{k+1}$ in \eqref{gen-AdaBB-alpha-2}. In this case, we have
    \begin{align*}
        M_{k+1} & =\frac{\alpha_{k+1}^{2}}{\lambda_{k+1}^2}-\frac{\alpha_{k+1}^{2}}{\alpha_{k}\lambda_{k+1}} \leq \frac{\lambda_{k+1}^2\alpha_{k}}{2(\alpha_{k}-\lambda_{k+1})}\Big(\frac{1}{\lambda_{k+1}^2}-\frac{1}{\alpha_{k}\lambda_{k+1}}\Big)=\frac{1}{2}, \\
        P_{k+1} & =\alpha_{k+1}^2/\alpha_{k} = \frac{\lambda_{k+1}^2}{2(\alpha_k-\lambda_{k+1})}< \alpha_{k} < P_{k}+\alpha_{k}.
    \end{align*}
\end{itemize}
Combining these three cases completes the proof.
\end{proof}

\begin{remark}\label{remark-M1-P1}
Note that $M_1$ and $P_1$ defined in \eqref{def-Mk-Pk} are determined only by $\lambda_1$, $\alpha_1$ and $\alpha_0$. Also, $\alpha_1$ is determined only by $\lambda_1$, $\alpha_0$ and $\theta_0$. Moreover, $\lambda_1$ is determined only by $x^0$ and $x^1$, while $x^1$ is determined only by $x^0$ and $\alpha_0$. Therefore, $M_1$ and $P_1$ are both absolute constants determined only by $x^0$, $\theta_0$ and $\alpha_0$. That is, by slightly abusing the notation, we can denote $M_1 = M_1(x^0,\theta_0,\alpha_0)$ and $P_1 = P_1(x^0,\theta_0,\alpha_0)$.
\end{remark}

The following lemma provides some useful properties about $\|x^{k+1}-x^{k}\|^2$.

\begin{lemma} \label{eq}
Let  $\{x^{k}\}$ be the  sequence generated by Algorithm \ref{agbb}, and $M_k$ and $P_k$ be defined in \eqref{def-Mk-Pk}.
Then, for any $k\geq 1$, we have
    \begin{align}
        \|x^{k+1}-x^{k}\|^2=\mathbf{I}_1 & := \Big(\frac{\alpha_{k}^{2}}{\lambda_{k}}-\frac{\alpha_{k}^{2}}{\alpha_{k-1}}\Big)\langle \nabla f(x^{k}) \!-\!\nabla f(x^{k-1}), x^{k}-x^{k-1} \rangle + \frac{\alpha_{k}^{2}}{\alpha_{k-1}}\langle\nabla f(x^{k}),x^{k-1}\!-\!x^{k}\rangle, \label{i1}\\
        \|x^{k+1}-x^{k}\|^2=\mathbf{I}_2 & :=  \alpha_k  \langle \nabla f(x^{k+1}) -\nabla f(x^{k}),x^{k+1}-x^{k}\rangle
+ \alpha_{k} \langle \nabla f(x^{k+1}),x^{k}-x^{k+1}\rangle, \label{i2} \\
        \|x^{k+1}-x^{k}\|^2\leq \mathbf{E} & := M_{k}\|x^{k}-x^{k-1}\|^2+P_{k}\big(f(x^{k-1})-f(x^{k})\big). \label{e1}
    \end{align}
\end{lemma}
\begin{proof}
Let $k\geq 1$ be arbitrarily fixed. 
Equation \eqref{i1} can be proved as follows by using the update $x^{k+1}=x^k -\alpha_k\nabla f(x^k)$ and the identity $\|a\|^2=\|a-b\|^2-\|b\|^2+2\langle a,b\rangle$:
\begin{align*}
    & \|x^{k+1}-x^{k}\|^2 \\
    = & \alpha_k^2 \|\nabla f(x^{k})\|^2
    =\alpha_{k}^{2}\left\|\nabla f(x^{k})-\nabla f(x^{k-1})\right\|^{2}-\alpha_{k}^{2}\left\|\nabla f(x^{k-1})\right\|^{2}+2 \alpha_{k}^{2}\langle\nabla f(x^{k}), \nabla f(x^{k-1})\rangle \nonumber\\
    = & \frac{\alpha_{k}^{2}}{\lambda_{k}} \langle \nabla f(x^{k}) - \nabla f(x^{k\!-\!1}), x^{k}-x^{k\!-\!1} \rangle - \alpha_{k}^{2} \langle \nabla f(x^{k\!-\!1}) - \nabla f(x^{k}), \nabla f(x^{k\!-\!1})\rangle + \alpha_{k}^{2}\langle\nabla f(x^{k}), \nabla f(x^{k\!-\!1})\rangle \nonumber \\
    = & \frac{\alpha_{k}^{2}}{\lambda_{k}} \langle \nabla f(x^{k}) \!-\!\nabla f(x^{k\!-\!1}), x^{k}-x^{k\!-\!1} \rangle - \frac{\alpha_{k}^{2}}{\alpha_{k-1}} \langle \nabla f(x^{k\!-\!1})\!-\! \nabla f(x^{k}), x^{k\!-\!1}-x^{k} \rangle +  \frac{\alpha_{k}^{2}}{\alpha_{k-1}}\langle\nabla f(x^{k}),x^{k\!-\!1}\!-\!x^{k}\rangle, \nonumber
\end{align*}
which proves \eqref{i1} by noting the definition of $\lambda_k$ in \eqref{lamk}.
Equation \eqref{i2} simply follows from
\[
\langle \nabla f(x^{k+1}) -\nabla f(x^{k}),x^{k+1}-x^{k}\rangle = \langle \nabla f(x^{k+1}),x^{k+1}-x^{k}\rangle + \frac{1}{\alpha_k}\|x^{k+1}-x^{k}\|^2.
\]
We now prove \eqref{e1} by analyzing three cases.
\begin{itemize}[leftmargin=*]
    \item Case (i): when $\lambda_{k}\geq \alpha_{k-1}$, we know $1/\lambda_k\leq 1/\alpha_{k-1}$. By the convexity of $f$ and  monotonicity of $\nabla f$, we have
\begin{equation}
\|x^{k+1}-x^{k}\|^2 = \mathbf{I}_1 \leq  \frac{\alpha_{k}^{2}}{\alpha_{k-1}}\langle\nabla f(x^{k}),x^{k-1}\!-\!x^{k}\rangle \leq \frac{\alpha_{k}^{2}}{\alpha_{k-1}}\big(f(x^{k-1})-f(x^{k})\big).   \label{c1}
\end{equation}
    \item Case (ii): when $\alpha_{k-1}/2 < \lambda_{k} < \alpha_{k-1}$, we have ${2}/{\alpha_{k-1}}-{1}/{\lambda_{k}}>0$, and therefore,
\begin{align}
\|x^{k+1}-x^{k}\|^2 =\; & \mathbf{I}_1 = \Big(\frac{2\alpha_{k}^{2}}{\alpha_{k-1}}-\frac{\alpha_{k}^{2}}{\lambda_{k}}\Big)\langle\nabla f(x^{k}),x^{k-1}\!-\!x^{k}\rangle \nonumber \\
& +
 \Big(\frac{\alpha_{k}^{2}}{\lambda_{k}}-\frac{\alpha_{k}^{2}}{\alpha_{k-1}}\Big)\Big(\langle \nabla f(x^{k}) \!-\!\nabla f(x^{k-1}), x^{k}-x^{k-1} \rangle + \langle\nabla f(x^{k}),x^{k-1}\!-\!x^{k}\rangle \Big) \nonumber \\
\stackrel{\eqref{i2}} \leq & \Big(\frac{\alpha_{k}^{2}}{\lambda_{k}\alpha_{k-1}}-\frac{\alpha_{k}^{2}}{\alpha_{k-1}^2}\Big)\|x^{k}-x^{k-1}\|^2+\Big(\frac{2\alpha_{k}^{2}}{\alpha_{k-1}}-\frac{\alpha_{k}^{2}}{\lambda_{k}}\Big)\big(f(x^{k-1})-f(x^{k})\big), \label{c2}
\end{align}
where the inequality follows from the convexity of $f$ and replacing $k$ by $k-1$ in \eqref{i2}.
    \item Case (iii): when $\lambda_{k}\leq \alpha_{k-1}/2$, by using the Young's inequality, we have
\[
\begin{aligned}
\Big(\frac{\alpha_{k}^{2}}{\lambda_{k}}-\frac{\alpha_{k}^{2}}{\alpha_{k-1}}\Big)\langle \nabla f(x^{k}) \!-\!\nabla f(x^{k-1}), x^{k}-x^{k-1} \rangle & = \Big(\frac{\alpha_{k}^{2}}{\lambda_{k}^2}-\frac{\alpha_{k}^{2}}{\alpha_{k-1}\lambda_{k}}\Big)\frac{\langle \nabla f(x^{k}) \!-\!\nabla f(x^{k-1}), x^{k}-x^{k-1} \rangle^2}{\|\nabla f(x^{k})- \nabla f(x^{k-1})\|^2} \\
& \leq \Big(\frac{\alpha_{k}^{2}}{\lambda_{k}^2}-\frac{\alpha_{k}^{2}}{\alpha_{k-1}\lambda_{k}}\Big)\|x^{k}-x^{k-1}\|^2,
\end{aligned}
\]
which directly leads to
\begin{align}
\|x^{k+1}-x^{k}\|^2 = \mathbf{I}_1 \leq   \Big(\frac{\alpha_{k}^{2}}{\lambda_{k}^2}-\frac{\alpha_{k}^{2}}{\alpha_{k-1}\lambda_{k}}\Big)\|x^{k}-x^{k-1}\|^2 + \frac{\alpha_{k}^{2}}{\alpha_{k-1}}\big(f(x^{k-1})-f(x^{k})\big).\label{c3}
\end{align}
\end{itemize}
Indeed, the inequality \eqref{c3} holds for all $\lambda_{k}< \alpha_{k-1}$. Combining these three cases proves \eqref{e1}.
\end{proof}

Now, we are ready to derive a non-increasing Lyapunov energy. For this purpose, we define
\[
w_k := \alpha_k+P_k-P_{k+1}, \quad \forall\, k\geq 0.
\]

\begin{lemma} \label{lemma-ener}
Define Lyapunov function
\begin{equation}\label{def-varUpsilon}
\varUpsilon_{k} := \|x^{k}-x^*\|^2  +  2M_{k}\|x^{k}-x^{k-1}\|^2 + (2\alpha_{k-1}+2P_{k-1}) \big(f(x^{k-1})-f_*\big).
\end{equation}
    Then for $\{x^{k}\}$ generated by Algorithm \ref{agbb}, we have
    \begin{align}
        \varUpsilon_{k+1}\leq \varPhi_k:= \varUpsilon_{k}-2 w_{k-1}\big(f(x^{k-1})-f_*\big)\leq \varUpsilon_{k}, \quad \forall\, k\geq 1. \label{energy}
    \end{align}
\end{lemma}

\begin{proof}
First, from the convexity of $f$, we have
\begin{equation}
\begin{aligned} \label{gds}
\|x^{k+1}-x^{*}\|^{2} & =\left\|x^{k}-\alpha_{k} \nabla f(x^{k})-x^{*}\right\|^{2} \\
& =\|x^{k}-x^{*}\|^{2}-2 \alpha_{k}\langle \nabla f(x^{k}), x^{k}-x^{*}\rangle+\alpha_{k}^{2}\|\nabla f(x^{k})\|^{2} \\
& \leq \|x^{k}-x^{*}\|^{2}-2 \alpha_{k}\left(f(x^{k})-f_*\right)+\alpha_{k}^{2}\left\|\nabla f(x^{k})\right\|^{2}.\end{aligned}
\end{equation}
From \eqref{e1}, we have
    \begin{equation}\label{gds-follow-1}
        \|x^{k+1}-x^{k}\|^2 = \alpha_k^2 \|\nabla f(x^{k})\|^2 \leq 2 \mathbf{E}-\|x^{k+1}-x^{k}\|^2 = 2 \mathbf{E}-\alpha_{k}^{2}\left\|\nabla f(x^{k})\right\|^{2}.
    \end{equation}
Summing \eqref{gds} and \eqref{gds-follow-1} yields
\begin{equation}\label{gds-follow-2}
    \begin{aligned}
    & \|x^{k+1}-x^{*}\|^{2}+\|x^{k+1}-x^{k}\|^2+2 \alpha_{k}\left(f(x^{k})-f_*\right) \\ 
    \leq \, & \|x^{k}-x^{*}\|^{2}+2 \mathbf{E} \\ 
    \leq \, & \varUpsilon_k-(2\alpha_{k-1}+2P_{k-1})(f(x^{k-1})-f_*)+2P_k(f(x^{k-1})-f(x^k)),
    \end{aligned}
\end{equation}
which further implies
    \begin{equation}
     \varUpsilon_{k+1}\leq \|x^{k+1}-x^*\|^2 + \|x^{k+1}-x^{k}\|^2 + (2\alpha_{k}+2P_{k}) \big(f(x^{k})-f_*\big) \leq \varPhi_k,    \label{gdmp}
    \end{equation}
    where the first inequality is due to $2M_{k+1}\leq 1$ from Lemma \ref{mpi}.
    This proves the first inequality in \eqref{energy}. The second inequality in \eqref{energy} is trivial because $w_k\geq 0$ from Lemma \ref{mpi}.
\end{proof}

Lemma \ref{lemma-ener} immediately leads to the boundedness of $\{x^k\}$.

\begin{corollary} \label{co:bound}
The sequence $\{x^{k}\}$ generated by Algorithm \ref{agbb} is bounded. In particular, for all $k\geq 0$, we have $x^k \in B(x^*,R)$, where $R$ is defined as:
\begin{equation}\label{def-R}
R^2 := \|x^0-x^*\|^2+ \alpha_0^2(1+2M_1)\|\nabla f(x^0)\|^2 + \max\{2P_1-2\alpha_0,0\} \big(f(x^0)-f^*\big).
\end{equation}
Note that from Remark \ref{remark-M1-P1}, $M_1 = M_1(x^0,\theta_0,\alpha_0)$ and $P_1 = P_1(x^0,\theta_0,\alpha_0)$ are both absolute constants.
\end{corollary}
\begin{proof}
    From \eqref{energy}, we can obtain
    $\|x^{k}-x^*\|^2 \leq \varUpsilon_{k} \leq \varPhi_{k-1} \leq \varPhi_1$ for all $k \geq 1$.
    From \eqref{def-P0}, we have $w_0=0$. Therefore,
    \begin{align}
    \varPhi_1=\varUpsilon_1 & = \|x^1-x^*\|^2  +  2 M_{1}\|x^{1}-x^{0}\|^2 + (2\alpha_{0}+2P_{0}) \big(f(x^{0})-f_*\big)\nonumber \\ & \stackrel{\eqref{gds}} \leq \|x^{0}-x^*\|^2  +  \alpha_0^2(1+2M_{1})\|\nabla f(x^{0})\|^2 + 2(P_{1}-\alpha_0) \big(f(x^{0})-f_*\big) \leq R^2.\label{phi1smallerR2}
    \end{align}
    Consequently, we have $x^k\in B(x^*,R)$ for all $k\geq 1$. It is trivial to see $x^0\in B(x^*,R)$. This completes the proof.
\end{proof}

\begin{remark}
    From now on, we assume that both \eqref{local-smooth-1} and \eqref{local-smooth-2} hold with $R$ defined in \eqref{def-R}. From \eqref{local-smooth-2}, we immediately have the following useful result:
    \begin{equation}\label{lambdak-lower-bound}
        \lambda_k \geq 1/L, \quad \forall k\geq 1.
    \end{equation}
\end{remark}

The following proposition gives a lower bound on $\alpha_k$ and estimates the order of $\sum_{i=1}^k \alpha_i$.
\begin{proposition}\label{lemma:sum-alpha}
For $\{\alpha_k\}$ generated by Algorithm \ref{agbb}, we have
\begin{itemize}
    \item[(i)] If $\alpha_j\geq\frac{1}{\sqrt{2}L}$ for some $j$, then $\alpha_k\geq\frac{1}{\sqrt{2}L}$ for any $k\geq j$;
    \item[(ii)] $\alpha_k\geq c:= \min\{\alpha_0,\frac{1}{\sqrt{2}L}\} > 0$ for all $k\geq 0$;
    \item[(iii)] $\sum_{i=1}^k \alpha_i = O(k)$.
\end{itemize}
\end{proposition}

\begin{proof}
We first prove part (i) by considering the three cases in the ($j+1$)-th iteration of Algorithm \ref{agbb}. If (Case i) happens, then we have $\alpha_{j+1}\geq \alpha_j\geq \frac{1}{\sqrt{2}L}$. If (Case ii) or (Case iii) happens, then we have $\alpha_{j+1}\geq \frac{\lambda_{j+1}}{\sqrt{2}}\geq \frac{1}{\sqrt{2}L}$, where we used \eqref{lambdak-lower-bound} for the second inequality. By induction, this completes the proof of part (i).

We now prove part (ii). Let $r\geq 1$ be the smallest integer that satisfies $\lambda_r < \alpha_{r-1}$.  
When $r=1$, we obtain $\alpha_1\geq \frac{\lambda_1}{\sqrt{2}}\geq \frac{1}{\sqrt{2}L}$. Consequently, applying the result from part (i) yields $\alpha_k\geq \frac{1}{\sqrt{2}L}$ for all $k\geq 1$.
When $r>1$,  this implies
$\lambda_k\geq \alpha_{k-1}$ for $k = 1,2,\ldots,r-1$, i.e., (Case i) in Algorithm \ref{agbb} happens for the first $r-1$ iterations.  This leads to $\alpha_1\geq\alpha_0$, and $\alpha_k\geq \sqrt{2}\alpha_{k-1}$ for $k=2,\ldots,r-1$. Therefore, we have $\alpha_k \geq \sqrt{2}^{k-1}\alpha_0\geq\alpha_0$ for $k=1,2,\ldots,r-1$.
Moreover, $\lambda_r < \alpha_{r-1}$ also implies that either (Case ii) or (Case iii) in Algorithm \ref{agbb} happens for the $r$-th iteration. In both cases, we have
$\alpha_r\geq\frac{\lambda_r}{\sqrt{2}}\geq \frac{1}{\sqrt{2}L}$, where we used \eqref{lambdak-lower-bound} for the second inequality. Now from part (i), we know that $\alpha_k\geq\frac{1}{\sqrt{2}L}$ for any $k\geq r$. This completes the proof of part (ii).

Part (iii) follows from part (ii) immediately.
\end{proof}

Now we are ready to present the main convergence result of Algorithm \ref{agbb}.

\begin{thom}[Ergodic convergence] \label{thom1}
For sequence $\{x^{k}\}$ generated by Algorithm \ref{agbb}, define
\begin{align*}
    \Bar{x}^{k}  = \frac{(\alpha_k+P_k)x^{k}+\sum_{i=1}^{k-1}w_ix^{i}}{S_{k}},  \text{~~with~~} S_k = P_1+ \sum_{i=1}^k \alpha_i.
\end{align*}
Then we have
\begin{equation}\label{thom1-conclusion}
f(\Bar{x}^{k})-f_* \leq \frac{\varPhi_1}{2S_{k}}=O\left(\frac{1}{k}\right).
\end{equation}
\end{thom}

\begin{proof}
We only need to prove the inequality in \eqref{thom1-conclusion}, because $\frac{\varPhi_1}{2S_{k}}=O\left(\frac{1}{k}\right)$ follows directly from \eqref{phi1smallerR2}, Remark \ref{remark-M1-P1} and Proposition \ref{lemma:sum-alpha} part (iii).
    From \eqref{energy} we have $\varUpsilon_{i+1}\leq \varUpsilon_{i}-2 w_{i-1}\big(f(x^{i-1})-f_*\big)$. Summing this inequality over $i=1,\ldots,k$ yields (note $w_0=0$)
    \begin{align*}
        \varUpsilon_{k+1} + 2\sum_{i=1}^{k-1}w_i\big(f(x^i)-f_*\big) \leq \varUpsilon_{1} = \varPhi_1.
    \end{align*}
    Using \eqref{def-varUpsilon}, we know that
    \[
    \|x^{k+1}-x^*\|^2  +  2M_{k+1}\|x^{k+1}-x^{k}\|^2 + (2\alpha_{k}+2P_{k}) \big(f(x^{k})-f_*\big) + 2\sum_{i=1}^{k-1}w_i\big(f(x^i)-f_*\big) \leq \varPhi_1,
    \]
    which further leads to
    \begin{equation}\label{proof-thom1-eq1}
    (\alpha_{k}+P_{k}) \big(f(x^{k})-f_*\big) + \sum_{i=1}^{k-1}w_i\big(f(x^i)-f_*\big) \leq \frac{\varPhi_1}{2}.
    \end{equation}
    Since $w_i=\alpha_i+P_i-P_{i+1}$, we have $\sum_{i=1}^{k-1}w_i = \sum_{i=1}^{k-1}\alpha_i+P_1-P_k$. We thus have
    \[
        (\alpha_k + P_{k}) + \sum_{i=1}^{k-1}w_i = (\alpha_k + P_{k}) + \sum_{i=1}^{k-1}\alpha_i+P_1-P_k = P_1+\sum_{i=1}^k\alpha_i = S_k.
    \]
    Utilizing the convexity of $f$, we obtain
    \[
        f(\bar{x}^{k}) = f\left(\frac{(\alpha_k+P_k)x^{k}+\sum_{i=1}^{k-1}w_ix^{i}}{S_{k}}\right) \leq \frac{\alpha_k+P_k}{S_k} f(x^k) + \sum_{i=1}^{k-1}\frac{w_i}{S_k}f(x^i),
    \]
    which leads to
    \[
        f(\bar{x}^{k})-f_* \leq \frac{\alpha_k+P_k}{S_k} (f(x^k)-f_*) + \sum_{i=1}^{k-1}\frac{w_i}{S_k}(f(x^i)-f_*)\leq \frac{\varPhi_1}{2S_k},
    \]
    where the last inequality follows from \eqref{proof-thom1-eq1}.
\end{proof}

Next, we present a variant of the Opial lemma, which is useful in our convergence analysis.
\begin{lemma}[{\cite[Lemma 2]{MM20}}]\label{opial}
Let $\{x^{k}\}$ and $\{a_{k}\}$ be two sequences in $\mathbb{R}^{n}$ and $\mathbb{R}_{+}$, respectively. Suppose that $\{x^{k}\}$ is bounded, its cluster points belong to $\mathcal{X} \subset \mathbb{R}^{n} $ and it also holds that
\[
    \|x^{k+1}-x\|^2 + a_{k+1} \leq \|x^{k}-x\|^2 + a_k, \quad \forall x\in \mathcal{X},
\]
then $\{x^{k}\}$ converges to some element in $\mathcal{X}$.
\end{lemma}

\begin{thom}[Pointwise convergence] \label{thom:ic}
The sequence $\{x^{k}\}$ generated by Algorithm \ref{agbb} globally converges to an optimal solution of \eqref{p}. Moreover, we have the following sublinear convergence rate of $\min_{1\leq i\leq k} \|\nabla f(x^i)\|^2$:
\begin{equation}\label{grad-min-rate}
\min_{1\leq i\leq k} \|\nabla f(x^i)\|^2 = O\left(\frac{1}{k}\right).
\end{equation}
\end{thom}

\begin{proof}
    By using \eqref{local-smooth-1}, we have the following improved analysis for \eqref{gds}:
    \begin{equation}
\begin{aligned} \label{gds1}
\|x^{k+1}-x^{*}\|^{2} & =\|x^{k}-x^{*}\|^{2}-2 \alpha_{k}\langle \nabla f(x^{k}), x^{k}-x^{*}\rangle+\alpha_{k}^{2}\|\nabla f(x^{k})\|^{2} \\
& \leq \|x^{k}-x^{*}\|^{2}-2 \alpha_{k}\left(f(x^{k})-f_*+\frac{1}{2L}\|\nabla f(x^k)\|^2\right)+\alpha_{k}^{2}\left\|\nabla f(x^{k})\right\|^{2},\end{aligned}
\end{equation}
and then \eqref{gdmp} can be changed to:
    \[
        \varUpsilon_{k+1}+\frac{\alpha_k}{L}\|\nabla f(x^k)\|^2\leq\|x^{k+1}-x^*\|^2 + \|x^{k+1}-x^{k}\|^2 + (2\alpha_{k}+2P_{k}) \big(f(x^{k})-f_*\big) + \frac{\alpha_k}{L}\|\nabla f(x^k)\|^2\leq \varPhi_k\leq\varUpsilon_k,
    \]
    which yields to $\varPhi_{k+1}+\frac{\alpha_k}{L}\|\nabla f(x^k)\|^2\leq \varPhi_k$ for all $k\geq 1$.
    Change the index $k$ to $i$, and sum this inequality over
    $i=1,\ldots,k$, we obtain
    $\sum_{i=1}^k \alpha_k\|\nabla f(x^k)\|^2 \leq \varPhi_1 L$. Using Proposition \ref{lemma:sum-alpha} part (ii), we have
    \begin{equation}\label{sum-grad-bounded}
        \sum_{i=1}^k \|\nabla f(x^k)\|^2 \leq \frac{\varPhi_1 L}{c}.
    \end{equation}
    Therefore, $\nabla f(x^k)\xrightarrow{k \rightarrow \infty} 0$. This proves that all cluster points of $\{x^{k}\}$ belong to the solution sets of \eqref{p}. Now using $\varPhi_{k+1}\leq\varPhi_k$ and applying Lemma \ref{opial} by letting  $\mathcal{X}=\mathcal{X^*}$ and $a_{k} =2M_{k}\|x^{k}-x^{k-1}\|^2+2P_{k}\big(f(x^{k-1})-f_*\big)$, we obtain that $\{x^k\}$ converges to an optimal solution of \eqref{p}.

    Moreover, \eqref{sum-grad-bounded} immediately leads to the convergence rate \eqref{grad-min-rate}.
\end{proof}

\section{Improved Lower Bound for \texorpdfstring{$\sum_{i=1}^k\alpha_i$}{}}\label{explan}

Existing analysis of gradient method \eqref{gd} for convex minimization \eqref{p} requires stepsize $\alpha_k\leq 1/L$ to achieve the $O(1/k)$ convergence rate $f(x^k)-f_* = O(1/k)$. In this kind of analysis, choosing $\alpha_k\leq 1/L$ guarantees that the function value has a sufficient decrease in each iteration. In practice, $\alpha_k=1/L$ is usually chosen because it is the largest stepsize in this setting. In the case where adaptive stepsize is used, a natural question to ask is whether we can guarantee that on average the stepsize is approximately equal to $1/L$. This has been posed as an open question recently by Malitsky and Mishchenko \cite{MM23}. More precisely, the open question posed in \cite{MM23} asks whether one can design an adaptive algorithm such that the lower bound for $\sum_{i=1}^k\alpha_i$ is close to $k/L$. In this section, we answer this question affirmatively: the stepsizes generated by our AdaBB algorithm satisfy
\begin{equation}\label{open-question-answer}
    \sum_{i=i_0+1}^k\alpha_i\geq (k-i_0)/L, \mbox{ where } i_0 \in \{ 0,1,2\} \mbox{ depending on the choice of } \alpha_0.
\end{equation}
This further implies $\sum_{i=1}^k \alpha_i \geq \frac{k-2+\sqrt{2}}{L}$ for all $k\geq 1$. With a suitably chosen $\alpha_0$, this can be further improved to $\sum_{i=1}^k\alpha_i\geq k/L$.

Our result requires the following choice of $\theta_0$: \begin{equation}\label{int:theta}
    \theta_0 = \left\{\begin{array}{ll}
        {\lambda_1^2}/{(2\alpha_0^2)}-1, & \mbox{ if } \lambda_1 \geq \sqrt{2}\alpha_0, \\
        0, & \mbox{ otherwise.}
    \end{array}\right.
\end{equation}
Note that $\lambda_1$ is fully determined by $\alpha_0$ and $x^0$. So $\theta_0$ can be pre-given to the algorithm. In the rest of this paper, we assume that $\theta_0$ is chosen as in \eqref{int:theta}. This choice of $\theta_0$ immediately leads to the following lower bound for $\alpha_i$ and $\sum_{i=1}^k\alpha_i$.

\begin{lemma} \label{lower-bound-alphai}
For any given $x^0\in\mathbb{R}^n$ and $\alpha_0 > 0$, the stepsizes generated by our AdaBB (Algorithm \ref{agbb}) satisfy
$\alpha_i\geq \frac{1}{\sqrt{2}L}$ for all $i\geq 1$, and hence
$\sum_{i=1}^{k} \alpha_{i} \geq \frac{k}{\sqrt{2}L}$ for any $k\geq 1$.
\end{lemma}

\begin{proof}
We first prove $\alpha_1\geq\lambda_1/\sqrt{2}$. There are three cases to consider.
\begin{itemize}
    \item[(a).] If $\lambda_1\geq\sqrt{2}\alpha_0$, then (Case i) in Algorithm \ref{agbb} happens. From the definition of $\theta_0$ in \eqref{int:theta}, we have $\alpha_1=\lambda_1/\sqrt{2}$.
    \item[(b).] If $\alpha_0\leq\lambda_1 <\sqrt{2}\alpha_0$, then (Case i) happens. In this case, we have $\theta_0=0$, and therefore, $\alpha_1=\alpha_0 > \lambda_1/\sqrt{2}$.
    \item[(c).] If $\lambda_1<\alpha_0$, then (Case ii) or (Case iii) happens. In this case, recall Lemma \ref{alpha-property}, we have $\alpha_1\geq\lambda_1/\sqrt{2}$.
\end{itemize}
Combining these three cases proves $\alpha_1\geq\lambda_1/\sqrt{2}$, which further implies $\alpha_1\geq 1/(\sqrt{2}L)$ due to \eqref{lambdak-lower-bound}. It then follows from Proposition \ref{lemma:sum-alpha} (i) that $\alpha_i\geq \frac{1}{\sqrt{2}L}$ for all $i\geq 1$.
As a consequence, we obtain $\sum_{i=1}^{k} \alpha_{i} \geq \frac{k}{\sqrt{2}L}$ for any $k\geq 1$.
\end{proof}

For the ease of presentation, we partition the index set $\{1,2,3, \ldots\}$ into three categories which correspond to the three cases in our AdaBB (Algorithm \ref{agbb}):
\begin{equation}\label{def-category}
I_1:=\{k \geq 1 \mid \lambda_k\geq\alpha_{k-1}\}, \;
I_2:=\{k \geq 1 \mid \alpha_{k-1}/2<\lambda_{k}<\alpha_{k-1}\} \text{~and~}
I_3:=\{k \geq 1 \mid 0<\lambda_{k} \leq \alpha_{k-1}/2\}.
\end{equation}

We first establish a useful lemma.

\begin{lemma} \label{est}
For any given $x^0\in\mathbb{R}^n$ and $\alpha_0 > 0$, the stepsizes generated by our AdaBB (Algorithm \ref{agbb}) have the following properties.
\begin{enumerate}
   \item[(a)] If $i \in I_2$, then $\alpha_{i}\geq \frac{1}{L}$;
\item[(b)] If $(i+1) \in I_3$, then $\alpha_{i}\geq \frac{2}{L}$ and $\alpha_{i}+\alpha_{i+1}\geq\frac{4+\sqrt{2}}{2L}$;
\item[(c)]  If  $i \in I_1\cup I_2$ and $(i+1) \in I_1$, then $\alpha_{i+1}\geq \frac{1}{L}$;
\item[(d)] If $i \in I_3$, 
then $\alpha_{i-1}+\alpha_{i}+\alpha_{i+1}\geq \frac{2+\sqrt{2}}{L}>\frac{3}{L}$.
\end{enumerate}
\end{lemma}

\begin{proof}
(a)  Since $i\in I_2$, by Lemma \ref{alpha-property} and \eqref{lambdak-lower-bound} we have $\alpha_{i}\geq \lambda_{i}\geq \frac{1}{L}$.

(b)  By definition,  $(i+1) \in I_3$ implies that $\alpha_{i}\geq 2\lambda_{i+1}\geq \frac{2}{L}$. This combining with Lemma \ref{lower-bound-alphai} yields $\alpha_{i}+\alpha_{i+1}\geq \frac{1}{\sqrt{2}L}+\frac{2}{L} = \frac{4+\sqrt{2}}{2L}$.

(c) If $i \in I_1$, we have $\alpha_{i}=\sqrt{1+\theta_{i-1}}\alpha_{i-1}\geq \alpha_{i-1}$, which gives $\theta_{i}=\frac{\alpha_i}{\alpha_{i-1}}\geq 1$. Since $(i+1) \in I_1$, we obtain $\alpha_{i+1}=\sqrt{1+\theta_{i}}\alpha_{i}\geq \sqrt{2}\alpha_{i}\geq \frac{1}{L}$.
If $i \in I_2$, we have $\alpha_{i}\geq \lambda_i \geq \frac{1}{L}$. Then, $(i+1)\in I_1$ implies $\alpha_{i+1}=\sqrt{1+\theta_{i}}\alpha_{i}\geq \alpha_{i}\geq \frac{1}{L}$.

(d) Since $i\in I_3$, from part (b) we have $\alpha_{i-1}+\alpha_i \geq  \frac{4+\sqrt{2}}{2L}$. The result follows  by noting
$\alpha_{i+1}\geq \frac{1}{\sqrt{2}L}$.
\end{proof}

We now define some useful notation.
Let $1\leq p\leq q$ be integers. We define $(p,\ldots,q)$ as the ordered sequence of indices from $p$ to $q$, and $\{p,\ldots,q\}$ as the set of indices from $p$ to $q$ without regard to order. 

\begin{definition}
  [Break index]
   An index $i\geq 1$ is called a break index if $i\in I_1$ and $(i+1)\notin I_3$.
   For $j\geq 1$, we let $i_j$ be the $j$th smallest break index within $\{1,2,3,\ldots\}$.
\end{definition}

To carry out a more elaborate analysis, we define $i_0 \in \{0, 1, 2\}$ as follows:
\begin{equation}\label{def:i0}
i_0 =
\left\{
  \begin{array}{ll}
    0, & \hbox{if $1\in I_2$, or $(1,2)\in I_1\times I_3$, or $(1,2) \in I_3\times I_3$,} \\
    1, & \hbox{if $(1,2)\in I_1\times I_1$, or $(1,2)\in I_1\times I_2$, or $(1,2)\in I_3\times I_2$, or $(1,2,3)\in I_3\times I_1\times I_3$,} \\
    2, & \hbox{if $(1,2,3) \in I_3\times I_1 \times I_1$, or $(1,2,3)\in I_3\times I_1 \times I_2$.}
  \end{array}
\right.
\end{equation}
Note that every index belongs to one of the three categories \eqref{def-category}. Moreover, the nine conditions in \eqref{def:i0} cover all possibilities for the first three indices 1, 2, and 3. 
Our idea to prove the improved bound \eqref{open-question-answer} is to divide the ordered sequence of indices $(i_0+1,i_0+2,\ldots,k)$ into many shorter pieces, and for each piece, say, $(p,\ldots,q)$, we shall show that $\sum_{i=p}^q \alpha_i\geq (q-p+1)/L$.
In the rest of this section, we assume $k\geq 3$ is an arbitrarily fixed integer.
For fixed $k\geq 3$, we assume that there are $(m-1)$ break indices within $\{1,2,\ldots, k\}$, which satisfy
$1\leq i_1<i_2<\ldots<i_{m-1}\leq k$.
For convenience, we define
\begin{equation}\label{def:Tj}
T_j := ({i_{j-1}+1},{i_{j-1}+2},\ldots,{i_j}) \text{~~for~}j=1,2,\ldots,m-1, \text{~and~~} T_m := (i_{m-1}+1,i_{m-1}+2,\ldots, k).
\end{equation}
That is,
\[
\overbrace{i_0+1, \ldots, i_1}^{T_1}, \overbrace{i_1+1, \ldots, i_2}^{T_2},
 \ldots, \ldots,
 \overbrace{i_{m-2}+1, \ldots, i_{m-1}}^{T_{m-1}},
 \overbrace{i_{m-1}+1, \ldots, k}^{T_{m}}.
\]
Note that $(p,\ldots,q)$ is an empty set if $p>q$. Therefore, $T_1=\varnothing$ if $i_0\geq i_1$, and $T_m=\emptyset$ if $i_{m-1}=k$.
Due to the definition of break index, if there is an index $p\in I_1$ such that $i_{j-1}+1 \leq p \leq i_j-1$ for some $j$, then $(p+1)\in I_3$ must hold.
For $i\geq 1$ and $j\geq i+1$, we define the following sets, which contain ordered and continuous indices.
\begin{align}
    \mathcal{A} & :=
    \Big \{ (i,i+1,\ldots,j)   \mid \{i+1,\ldots,j\} \subseteq I_3\Big\}, \label{a}\\
    \mathcal{Q} & := \Big \{ (i,{i+1}, \ldots, j)  \mid  i \in I_2, \, \{i+1,\ldots,j\}\subseteq I_2 \cup I_3\Big\},\label{q}\\
    \mathcal{M} &:= \Big \{ (i,{i+1}, \ldots, j)  \mid  i \notin I_3, \, j \notin I_1, \text{ and, for any } i\leq \ell \leq j-1, \;
\ell \in I_1 \Rightarrow (\ell+1) \in I_3\Big\}.\label{m}
\end{align}
To establish our improved bound \eqref{open-question-answer}, it is sufficient to show that $\sum_{i \in T_j} \alpha_i\geq \frac{|\,T_j\,|}{L}$
for $j = 1,2,\ldots,m$. To show this, we first prove the following key lemma.
\begin{lemma}\label{lem:AQM}
Let $i\geq 1$ and $j\geq i+1$.
If $(i,i+1,\ldots,j) \in \mathcal{A} \cup \mathcal{Q} \cup  \mathcal{M}$,
where $\mathcal{A}$, $\mathcal{Q}$ and $\mathcal{M}$ are defined in \eqref{a}-\eqref{m}, then $\sum_{\ell=i}^j \alpha_\ell \geq ({j-i+1})/{L}$.
\end{lemma}

\begin{proof}
There are three cases to consider.
\begin{itemize}
    \item[(a).] $(i,i+1,\ldots,j) \in \mathcal{A}$. In this case, we have $\{i+1,\ldots,j\} \subseteq I_3$. Hence, it follows from Lemma \ref{est} (b) that $\alpha_{\ell} \geq \frac{2}{L}$ for $i \leq \ell \leq j-1$. Thus, we have $\sum_{\ell=i}^j \alpha_\ell \geq \frac{2(j-i)}{L} +\frac{1}{\sqrt{2}L} \geq \frac{j-i+1}{L}$ as $j-i\geq 1$.
    \item[(b).] $(i,i+1,\ldots,j)\in \mathcal{Q}$. In this case, we can always divide it into shorter pieces as $(i,i+1,\ldots,j)=(p_1, \ldots, q_1 \mid p_2, \ldots, q_2 \mid \ldots\mid p_s,\ldots, q_s)$ for some $s\geq 1$, with $p_1=i$, $q_s=j$, $p_{t+1}=q_t +1$ for $t=1,2,\ldots,s-1$, such that for each $(p_t,\ldots, q_t)$, $t=1,2,\ldots, s$, we have either $p_t=q_t\in I_2$, or $q_t\geq p_t+1$, $p_t\in I_2$ and $\{p_t+1,\ldots, q_t\}\subseteq I_3$. In the former case, $p_t=q_t\in I_2$, the corresponding piece has length one, and by Lemma \ref{est} (a), we have $\alpha_{p_t} \geq 1/L$. In the latter case, we have  $(p_t, \ldots, q_t)$ belongs to $\mathcal{A}$, and therefore $\sum_{\ell=p_t}^{q_t}\alpha_{\ell} \geq (q_t-p_t+1)/L$.
    \item[(c).] $(i,i+1,\ldots,j)\in\mathcal{M}$. Similarly, in this case it can always be divided into shorter pieces as $(p_1,\ldots,q_1 \mid p_2,\ldots, q_2 \mid\ldots\mid p_{s'},\ldots, q_{s'})$ for some $s'\geq 1$, with $p_1=i$, $q_{s'}=j$, $p_{t+1}=q_t+1$ for $t=1,2,\ldots,s'-1$, such that for each $(p_t,\ldots, q_t)$, $t=1,2,\ldots, s'$, we have either $p_t=q_t\in I_2$, or $q_t\geq p_t+1$,
    $\{p_t+1,\ldots, q_t\}\subseteq I_3$.
    In the former case, $p_t=q_t\in I_2$, the corresponding piece has length one, and by Lemma \ref{est} (i), we have $\alpha_{p_t} \geq 1/L$.
    In the latter case, again we have $(p_t, \ldots, q_t)$ belongs to $\mathcal{A}$, and the result $\sum_{\ell=p_t}^{q_t}\alpha_{\ell} \geq (q_t-p_t+1)/L$ follows.
\end{itemize}
For Cases (b) and (c), the result $\sum_{\ell=i}^j \alpha_\ell \geq \frac{j-i+1}{L}$ follows immediately since each piece $(p_t, \ldots, q_t)$ satisfies $\sum_{\ell=p_t}^{q_t}\alpha_{\ell} \geq (q_t-p_t+1)/L$.
\end{proof}

Equipped with Lemma \ref{lem:AQM}, we next show that $\sum_{i \in T_j} \alpha_i\geq |\,T_j\,|/L$ for $1\leq j\leq m$.
We first prove the case $2\leq j\leq m-1$, and then prove the cases $j=m$ and $j=1$.

\begin{lemma} \label{lem:in}
For $2\leq j\leq m-1$, there holds $\sum_{i \in T_j} \alpha_i\geq |\,T_j\,|/L$.
\end{lemma}
\begin{proof}
For $2\leq j\leq m-1$, we have $T_j=({i_{j-1}+1},\ldots, {i_{j}})$. By the definition of $i_j$, we have $i_{j-1}+1\in I_1\cup I_2$ and $i_j\in I_1$. This also implies $i_j-1\notin I_1$.  Moreover, if there exists $\ell \in I_1$ satisfying
${i_{j-1}+1} \leq \ell\leq i_j-2$, then there must hold $(\ell+1) \in I_3$.
We need to prove the result for two cases: (a) $|T_j \bigcap I_1|= 1$, and (b) $|T_j\bigcap I_1| \geq 2$.

\begin{itemize}
    \item[(a).] In this case, we have $|T_j\bigcap I_1|=1$. There are three cases to consider: (a1) ${i_{j-1}+1}\in I_1$; (a2) ${i_{j-1}+1}\in I_2$ and ${i_{j}-1} \in I_2$; (a3) ${i_{j-1}+1}\in I_2$ and ${i_{j}-1}\in I_3$.

    \begin{itemize}
        \item[(a1).] If ${i_{j-1}+1}\in I_1$, then ${i_{j-1}+1}={i_{j}}$ and $ T_j = \{i_j\}$.
        Since ${i_{j-1}}\in I_1$ and $i_j\in I_1$, from Lemma \ref{est} (c) we have ${\alpha_{i_{j}}} \geq \frac{1}{L}$.
        \item[(a2).] If ${i_{j-1}+1}\in I_2$ and ${i_{j}-1} \in I_2$, we partition $T_j$ into two parts: $T_j = ({i_{j-1}+1},\ldots, {i_{j}-1} \,\mid\, {i_{j}})$. Since $i_j\in I_1$, there is no index in $({i_{j-1}+1},\ldots, {i_{j}-1})$ belonging to $I_1$. Therefore, $({i_{j-1}+1},\ldots, {i_{j}-1}) \in \mathcal{Q}$. Moreover, $i_{j} \in I_1$ and ${i_{j}-1} \in I_2$ imply that  ${\alpha_{i_{j}}} \geq \frac{1}{L}$ (see Lemma \ref{est} (c)). Using Lemma \ref{lem:AQM}, we obtain $\sum_{i=i_{j-1}+1}^{i_j} \alpha_i \geq |\,T_j\,|/L$.
        \item[(a3).] If ${i_{j-1}+1}\in I_2$ and ${i_{j}-1}\in I_3$, we partition $T_j$ into two parts: $T_j = (i_{j-1}+1,\ldots, {i_{j}-3} \,\mid\, {i_{j}-2},{i_{j}-1},{i_{j}})$. Again, since $i_j\in I_1$, there is no index in $({i_{j-1}+1},\ldots, {i_{j}-3})$ belonging to $I_1$. Therefore, $({i_{j-1}+1},\ldots, {i_{j}-3})$, if nonempty, must belong to $\mathcal{Q}$. Moreover, it follows from Lemma \ref{est} (d) that $\sum_{\ell=i_{j}-2}^{i_j}\alpha_{\ell} \geq \frac{3}{L}$. Using Lemma \ref{lem:AQM}, we obtain $\sum_{i=i_{j-1}+1}^{i_j} \alpha_i \geq |\,T_j\,|/L$.
    \end{itemize}

    \item[(b).] In this case, we have $|T_j\bigcap I_1| \geq 2$. 
    There are again three cases to consider: (b1) ${i_{j}-1}\in I_2$; (b2) ${i_{j}-1}\in I_3$ and ${i_{j}-3}\in I_1$; (b3) ${i_{j}-1}\in I_3$ and ${i_{j}-3}\notin I_1$.
    \begin{itemize}
        \item[(b1).] If ${i_{j}-1}\in I_2$, then we can partition $T_j$ as $T_j = ({i_{j-1}+1},\ldots, {i_{j}-1}  \mid  {i_{j}})$. It is easy to see that $({i_{j-1}+1},\ldots, {i_{j}-1}) \in \mathcal{M}$. Moreover, according to Lemma \ref{est} (c), we have $\alpha_{i_{j}} \geq \frac{1}{L}$. Using Lemma \ref{lem:AQM}, we have $\sum_{i=i_{j-1}+1}^{i_j} \alpha_i \geq |\,T_j\,|/L$.
        \item[(b2).] If ${i_{j}-1} \in I_3$ and ${i_{j}-3}\in I_1$, then we must have ${i_{j}-2} \in I_3$, because otherwise $i_j-3$ is a break index. Moreover, $i_j-4\notin I_1$, because otherwise, $i_j-4$ is a break index. We then partition $T_j$ into $T_j = ({i_{j-1}+1},\ldots, {i_{j}-4} \mid {i_{j}-3} \mid {i_{j}-2},{i_{j}-1},{i_{j}})$. It is easy to see that $({i_{j-1}+1},\ldots,{i_{j}-4})$, if nonempty, must belong to $\mathcal{M}$. Moreover, since ${i_{j}-2} \in I_3$, we have by Lemma \ref{est} (b) that $\alpha_{i_{j}-3}\geq \frac{2}{L}$, and since ${i_{j}-1} \in I_3$,  we have $\alpha_{i_{j}-2}+\alpha_{i_{j}-1}+\alpha_{i_{j}} \geq \frac{3}{L}$ by Lemma \ref{est} (d). Using Lemma \ref{lem:AQM}, we have $\sum_{i=i_{j-1}+1}^{i_j} \alpha_i \geq |\,T_j\,|/L$.
        \item[(b3).] If ${i_{j}-1} \in I_3$ and ${i_{j}-3}\notin I_1$, then we partition $T_j$ into $T_j = ({i_{j-1}+1},\ldots, {i_{j}-3}  \mid {i_{j}-2},{i_{j}-1},{i_{j}})$. It is easy to see that $({i_{j-1}+1},\ldots, {i_{j}-3})$, if nonempty, must belong to $ \mathcal{M}$. Since ${i_{j}-1} \in I_3$, we have $\alpha_{i_{j}-2}+\alpha_{i_{j}-1}+\alpha_{i_{j}} \geq \frac{3}{L}$ by  Lemma \ref{est} (d). Using Lemma \ref{lem:AQM}, we have $\sum_{i=i_{j-1}+1}^{i_j} \alpha_i \geq |\,T_j\,|/L$.
    \end{itemize}
\end{itemize}
This completes the proof. 
\end{proof}

\begin{lemma} \label{lem:se}
For $j=m$ or $j=1$, there holds $\sum_{i \in T_j} \alpha_i\geq |\,T_j\,|/L$.
\end{lemma}
\begin{proof}
We first consider the case $j = m$ and in this case $T_m = ({i_{m-1}+1},\ldots,k)$. {If $i_{m-1}=k$, then $T_m$ equals to $\varnothing$. So we only need to consider the case when $i_{m-1}<k$.} If $k \in I_1$, then the proof is exactly the same as for the case $2\leq j\leq m-1$ in Lemma \ref{lem:in}. If $k\notin I_1$, then  $T_m\in\mathcal{M}$ as by the construction of $T_m$ we have ${i_{m-1}+1}\notin I_3$ and {if $\ell\in I_1$ within $\{{i_{m-1}+1}, \ldots, k-1\}$, we have $l+1\in I_3$.} 
In both cases, we have $\sum_{\ell=i_{m-1}+1}^{k} \alpha_{\ell} \geq |\,T_m\,|/L$.

We now consider the case $j=1$ and in this case $T_1 = ({i_{0}+1},\ldots,i_1)$. There are nine cases to consider according to the definition of $i_0$ in \eqref{def:i0}. Among them, in the following four cases, we have $T_1=\varnothing$ and no proof is needed:
   $(1,2,3) \in I_3\times I_1 \times I_1$ ($i_0=i_1=2$),
or $(1,2,3)\in I_3\times I_1 \times I_2$  ($i_0=i_1=2$),
or $(1,2)\in I_1\times I_1$ ($i_0=i_1=1$),
or $(1,2)\in I_1\times I_2$ ($i_0=i_1=1$).
It remains to consider the following five
scenarios: (a) $1\in I_2$; (b) $(1,2)\in I_1\times I_3$; (c) $(1,2)\in I_3\times I_2$; (d) $(1,2,3)\in I_3\times I_1\times I_3$; and (e) $(1,2) \in I_3\times I_3$.
\begin{itemize}
    \item[(a).] $1\in I_2$. In this case, $i_0=0$ and $T_1 = (i_0+1,\ldots,i_1)$ reduces to $(1,\ldots,i_1)$ with $i_1\geq 2$. This case is the same as the one in Lemma \ref{lem:in} because $1\notin I_3$ and $i_1\in I_1$.
    \item[(b).] $(1,2)\in I_1\times I_3$. In this case, $i_0=0$ and $T_1 = (i_0+1,\ldots,i_1)$ reduces to $(1,\ldots,i_1)$ with $i_1\geq 3$. This case is again the same as the one in Lemma \ref{lem:in} because $1\notin I_3$ and $i_1\in I_1$.
    \item[(c).] $(1,2)\in I_3\times I_2$. In this case, $i_0=1$ and $T_1 = (i_0+1,\ldots,i_1)$ reduces to $(2,\ldots,i_1)$ with $i_1\geq 3$. This case is again the same as the one in Lemma \ref{lem:in} because $2\notin I_3$ and $i_1\in I_1$.
    \item[(d).] $(1,2,3)\in I_3\times I_1\times I_3$. In this case, $i_0=1$ and $T_1 = (i_0+1,\ldots,i_1)$ reduces to $(2,\ldots,i_1)$ with $i_1\geq 4$. This case is again the same as the one in Lemma \ref{lem:in} because $2\notin I_3$ and $i_1\in I_1$.
    \item[(e).] $(1,2) \in I_3\times I_3$. In this case, $i_0=0$ and $T_1 = (i_0+1,\ldots,i_1)$ reduces to $(1,\ldots,i_1)$ with $i_1\geq 3$. In this case, we partition $T_1 = ({1},\ldots, {p}  \mid  {p+1},\ldots,{i_{1}})$, where $p\geq 2$, $\{1,\ldots,p\} \subseteq I_3$, and $(p+1) \notin I_3$. If $i_1>p+1$, then $(1,\ldots,p)$ belongs to $\mathcal{A}$ and $(p+1,\ldots,i_1)$ is the same as the one in Lemma \ref{lem:in}. If ${(p+1)}={i_1}\in I_1$, we then partition $T_1 = (1,\ldots, {i_1-3} \mid {i_1-2},{i_1-1},{i_1})$ and there are three cases to consider for the first part $(1,\ldots, {i_1-3})$:
    \begin{itemize}
        \item[(e1).] It is empty.
        \item[(e2).] It contains the index $1$ only, in which case we have $\alpha_1\geq\frac{2}{L}$  due to Lemma \ref{est} (b) and $2 \in I_3$.
        \item[(e3).] It belongs to $ \mathcal{A}$.
    \end{itemize}
Moreover, for the second part $({i_1-2},{i_1-1},{i_1})$, it follows from Lemma \ref{est} (d) that $\alpha_{i_1-2}+\alpha_{i_1-1}+\alpha_{i_1}\geq \frac{3}{L}$ because $i_1-1\in I_3$. Using Lemma \ref{lem:AQM}, we have shown $\sum_{\ell\in T_1} \alpha_{\ell} \geq |\,T_1\,|/L$ for all cases.
\end{itemize}
This completes the proof.
\end{proof}

Combining Lemmas \ref{lem:in} and \ref{lem:se}, we obtain the following theorem immediately.

\begin{thom} \label{1l}
For any given $x^0\in\mathbb{R}^n$ and $\alpha_0 > 0$, the stepsizes generated by our AdaBB (Algorithm \ref{agbb}) satisfy $\sum_{i=i_0+1}^{k} \alpha_i \geq \frac{k-i_0}{L}$ for all $k\geq 3$, where $i_0$ is defined in \eqref{def:i0}. This also implies $\sum_{i=1}^{k} \alpha_i \geq \frac{k-2+\sqrt{2}}{L}$ for all $k\geq 1$.
\end{thom}
\begin{proof}
The implication $\sum_{i=1}^{k} \alpha_i \geq \frac{k-2+\sqrt{2}}{L}$ can be verified as follows:
\begin{itemize}
    \item[(a).] If $i_0=0$, we have $\sum_{i=1}^{k} \alpha_i \geq \frac{k}{L} > \frac{k-2+\sqrt{2}}{L}$;
    \item[(b).] If $i_0=1$, we have $\sum_{i=1}^{k} \alpha_i = \alpha_1 + \sum_{i=2}^{k} \alpha_i \geq \frac{1}{\sqrt{2}L} + \frac{k-1}{L} >  \frac{k-2+\sqrt{2}}{L}$;
    \item[(c).] If $i_0=2$, we have $\sum_{i=1}^{k} \alpha_i = \alpha_1 + \alpha_2 +  \sum_{i=3}^{k} \alpha_i \geq \frac{\sqrt{2}}{L} + \frac{k-2}{L} = \frac{k-2+\sqrt{2}}{L}$,
\end{itemize}
 which completes the proof.
\end{proof}

\begin{remark}
Similar to \cite[Algorithm 2]{MM23}, if one can ensure $\alpha_0\in (\lambda_1, 2\lambda_1)$ through a line search strategy, then it is guaranteed that $1\in I_2$ and $i_0=0$. In this case, our improved bound \eqref{open-question-answer} becomes $\sum_{i=1}^k \alpha_i \geq \frac{k}{L}$ for any $k\geq 3$. For $k=1$, we have $\sum_{i=1}^k \alpha_i = \alpha_1 \geq \frac{1}{L}$ since $1\in I_2$. For $k=2$, there are three cases to consider.
\begin{itemize}
    \item[(a).] If $2\in I_1$, we have $\alpha_2\geq \frac{1}{L}$ follows from Lemma \ref{est} (c), hence we get $\alpha_1+\alpha_2\geq \frac{2}{L}$;
    \item[(b).] If $2 \in I_2$, we have $\alpha_2\geq \frac{1}{L}$ follows from Lemma \ref{est} (a), hence we get $\alpha_1+\alpha_2\geq \frac{2}{L}$;
    \item[(c).] If $2 \in I_3$, we have  $\alpha_1+\alpha_2\geq \frac{4+\sqrt{2}}{2L}\geq \frac{2}{L}$ follows from Lemma \ref{est} (b).
\end{itemize}
Hence, we prove that $\sum_{i=1}^k \alpha_i \geq \frac{k}{L}$ holds for all $k\geq 1$.
\end{remark}

We now give a detailed comparison of our results on the lower bounds of $\alpha_k$ and $\sum_{i=1}^k\alpha_i$ with the existing results in the literature. The results are summarized in Table \ref{table:stepsize-comparison}. From Table \ref{table:stepsize-comparison}, we first note that these results are all free with $\theta_0$ which is always pre-defined. The results for AdGD2 require a specially chosen $\alpha_0$, but other algorithms do not have restrictions on $\alpha_0$. Our AdaBB achieves the best lower bound for $\alpha_k$, i.e., $\alpha_k \geq \frac{1}{\sqrt{2}L}$, $\forall k\geq 1$. While AdGD \cite{MM23} and $\text{AdaPGM}^{\pi,r}$ also achieve the same lower bound, their results only hold for $k\geq r_1, r_3$, respectively. Lastly, our AdaBB clearly achieves the best lower bound for $\sum_{i=1}^k\alpha_i$. Overall, we believe that it is fair to claim that our AdaBB achieves the best results for the lower bounds of $\alpha_k$ and $\sum_{i=1}^k\alpha_i$.

\begin{table}[!htbp]
\begin{center}
\begin{tabular}{l|c|c|c|c|c|c}
\hline
& \multicolumn{2}{c|}{AdGD}      & AdGD2      & AdaPGM   & $\text{AdaPGM}^{\pi,r}$    & AdaBB \\ \cline{1-7}
$\theta_0$ & $+\infty$  & $0$   & $\frac{1}{3}$  & $\geq 1$          & $1$      & \eqref{int:theta} \\ \hline
$\alpha_0$ free?  & $\checkmark$       & $\checkmark$     & \xmark  & $\checkmark$  & $\checkmark$   & $\checkmark$  \\ \hline
$\alpha_k\geq $      & $\frac{1}{2L}$  & $\frac{1}{\sqrt{2}L}$ ($k\geq r_1$) & $\frac{1}{\sqrt{3}L}$                   & $\frac{1}{2L}$ ($k\geq r_2$) & $\frac{1}{\sqrt{2}L}$ ($k\geq r_3$) & $\frac{1}{\sqrt{2}L}$              \\ \hline
$\sum_{i=1}^k \alpha_i\geq$ & $\frac{k}{2L}$ & $\frac{k-r_1+1}{\sqrt{2}L}$                        & $\frac{k}{\sqrt{2}L}$            & $\frac{k-r_2+1}{2L}$                         & $\frac{k-r_3+1}{\sqrt{2}L}$                        & $\frac{k-2+\sqrt{2}}{L}$             \\ \hline
\end{tabular}
\end{center}
\caption{Comparision of the lower bounds of $\alpha_k$ and $\sum_{i=1}^k\alpha_i$ among AdGD \cite{MM20,MM23}, AdGD2 \cite{MM23}, AdaPGM \cite{LTSP23}, $\text{AdaPGM}^{\pi,r}$ \cite{lTP23}, and AdaBB. We listed two results for AdGD correponding to different values of $\theta_0$. The case $\theta_0=+\infty$ is analyzed in \cite{MM20} and the case $\theta_0=0$ is analyzed in \cite{MM23}. Here we note that the results of AdGD2 only hold when a special $\alpha_0$ satisfying $\alpha_0 L_1 \in [\frac{1}{\sqrt{2}},2]$ is chosen. Other algorithms do not have restrictions on $\alpha_0$. The lower bounds of $\alpha_k$ for AdGD \cite{MM23}, AdaPGM and $\text{AdaPGM}^{\pi,r}$ only hold for $k\geq r_1, r_2, r_3$, respectively. Here, $r_1$ is the smallest integer satisfying $\Pi_{i=1}^{r_1} \varpi_i\geq \frac{1}{\sqrt{2}L\alpha_0}$ with $\varpi_{n+1}=\sqrt{1+\varpi_{n}}$ and $\varpi_1=1$; $r_2=\lfloor 2\log_2{\frac{1}{\alpha_0 L}} \rfloor_{+}$; $r_3= 2 \lfloor \log_2{\frac{1}{\alpha_0 L}} \rfloor_{+}$.}\label{table:stepsize-comparison}
\end{table}

\section{Extensions}\label{exten}

In this section, we extend AdaBB (Algorithm \ref{alg:AdaBB}) to locally strongly convex problem and composite
convex optimization problems.

\subsection{When \texorpdfstring{$f$}{} is Locally Strongly Convex} \label{lconv}
In this subsection, we extend our analysis to the case where $f$ is locally strongly convex.
Specifically, in addition to the locally $L$-smoothness condition \eqref{local-smooth-1}, we also assume that $f$ is
locally $\mu$-strongly convex in $B(x^*, R)$, i.e.,
\begin{equation}\label{def:strong}
    f(x)-f(y)-\langle \nabla f(y), x-y\rangle \geq \tfrac{\mu}{2}\|x-y\|^2, \forall x,y \in B(x^*,R),
\end{equation}
where $R$ is defined in \eqref{def-R}, and we can prove that the sequence $\{x^k\}$ generated by the following algorithm lies in $B(x^*, R)$.
Additionally, it is worth noting that the parameters $\mu, L$ and $R$ are used solely for the purpose of analysis and are not involved in the algorithm. 
According to \cite[Theorem 2.1.10]{N87}, \eqref{def:strong} implies
\begin{equation}\label{jy-01}
\langle \nabla f(x)-\nabla f(y),x-y\rangle \leq \|\nabla f(x)-\nabla f(y)\|^2/\mu, \forall x,y \in B(x^*,R).
\end{equation}
We will present an extension of Algorithm \ref{alg:AdaBB} to handle this case and establish a linear convergence result.

Recall that $\lambda_k$ denotes the Short BB stepsize and is given by \eqref{lamk}.
The new algorithm follows the same iteration scheme as \eqref{gd}, with the only variation from Algorithm  \ref{alg:AdaBB}  being the update rule for the stepsize $\alpha_k$.
Specifically, for $k\geq 1$ we update $\alpha_k$ as follows
\begin{equation} \label{bbsimsc}
\alpha_k = \left\{
\begin{array}{lll}
\min\{\sqrt{1+\eta\theta_{k-1}}\alpha_{k-1},\lambda_k\}, &\text{update } \theta_k = \frac{\alpha_k}{\alpha_{k-1}}, & \text{if } \lambda_k \geq \alpha_{k-1},\\
\lambda_{k}, &\text{update } \theta_{k} = \frac{2\alpha_{k}}{\alpha_{k-1}}-1, & \text{if }\frac{\delta \alpha_{k-1}}{2} < \lambda_k < \alpha_{k-1},\\
\frac{\lambda_{k}}{\sqrt{2}}, &\text{update } \theta_k = \frac{\alpha_k}{\alpha_{k-1}}, & \text{if } 0 <\lambda_{k}\leq \frac{\delta\alpha_{k-1}}{2},
\end{array}
\right.
\end{equation}
where $\eta\in [0,1)$ and $\delta\in (1,2)$ are parameters.
In \eqref{bbsimsc}, when $\lambda_k \geq \alpha_{k-1}$, a more cautious stepsize is used compared to Algorithm \ref{alg:AdaBB}. Additionally, the region
${\delta \alpha_{k-1} / 2} < \lambda_k < \alpha_{k-1}$ is narrower than (Case ii) in Algorithm \ref{alg:AdaBB}. As a result, the region
$0 <\lambda_{k}\leq  {\delta\alpha_{k-1}/2}$ in \eqref{bbsimsc} becomes broader than (Case iii) in Algorithm \ref{alg:AdaBB}.
Similarly  to \eqref{def-Mk-Pk}, for all $k\geq 1$ we define $M_{k}$ and $P_{k}$ as follows
\begin{equation}\label{def-Mk-Pk-1}
\left\{
\begin{array}{lll}
M_{k}=0 ,    &P_{k}= \frac{\alpha_{k}^{2}}{\alpha_{k-1}},     & {\textrm{if }\lambda_{k}\geq \alpha_{k-1}}, \\
M_{k}=\frac{\alpha_{k}}{\alpha_{k-1}}-\frac{\alpha_{k}^{2}}{\alpha_{k-1}^2},   &
P_{k}= \frac{2\alpha_{k}^{2}}{\alpha_{k-1}}-\alpha_k,   &
{\textrm{if }\frac{\delta\alpha_{k-1}}{2}<\lambda_{k} <  \alpha_{k-1}},\\
M_{k}=\frac{1}{2}-\frac{\lambda_{k}}{2\alpha_{k-1}},   &
P_{k}= \frac{\alpha_{k}^{2}}{\alpha_{k-1}},
& {\textrm{if } 0 <\lambda_{k}\leq \frac{\delta\alpha_{k-1}}{2}}.
\end{array} \right. 
\end{equation}
For convenience, we define $P_0$ as $(P_1-\alpha_0)/\eta$ for $\eta \in (0,1)$. Otherwise, $P_0 = 0$.
We have the following lemma.

\begin{lemma} \label{mpis}
Let $\{\alpha_k\}$ and $\{\theta_k\}$ be generated by \eqref{bbsimsc} and  $M_{k}$ and $P_{k}$ be defined in \eqref{def-Mk-Pk-1}.
Then, we have $M_k\geq 0$, $P_k\geq 0$ and $P_{k} = \alpha_{k}\theta_{k}$ for all $k \geq 1$, and
$2M_{k+1}\leq 1$ and $P_{k+1}\leq \alpha_k +\eta P_{k}$ for $k \geq 0$.
\end{lemma}

\begin{proof}
Recall that $\delta\in(1,2)$. First, the nonnegativity of $M_k$ and $P_k$ and $P_{k} = \alpha_{k}\theta_{k}$ for $k\geq 1$ can be verified straightforwardly by their definitions together with \eqref{bbsimsc}.
Let $k\geq 0$ be fixed.
We then show the remaining claims by considering the following three cases.
\begin{enumerate}
\item[(i)] $\lambda_{k+1}\geq \alpha_{k}$. In this case, we have $M_{k+1} = 0$ and $\alpha_{k+1} \leq \sqrt{1+\eta \theta_{k}}\alpha_{k}$. Hence,
it follows that $P_{k+1} = {\alpha_{k+1}^{2}/\alpha_{k}} \leq (1+\eta \theta_{k})\alpha_{k} = \alpha_k+\eta P_k$.
\item[(ii)]  $\delta\alpha_{k}/2 < \lambda_{k+1} < \alpha_{k}$. In this case, we have $\alpha_{k+1}=\lambda_{k+1}$, and it is easy to verify that
    \begin{align*}
    M_{k+1} 
    = \frac{\lambda_{k+1}}{\alpha_k}-\frac{\lambda_{k+1}^2}{\alpha_k^2} \leq \frac{1}{4}
    \text{~~and~~}
    P_{k+1}  
    =  \frac{2\lambda_{k+1}^{2}}{\alpha_{k}}-\lambda_{k+1}\leq \alpha_k \leq \alpha_k+\eta P_k.
    \end{align*}

\item[(iii)] $0<\lambda_{k+1}\leq \delta\alpha_{k}/2$. In this case, we have $\alpha_{k+1}=\lambda_{k+1}/{\sqrt{2}}$, and it is elementary to verify that
    \begin{align*}
        M_{k+1}   = \frac{1}{2}-\frac{\lambda_{k+1}}{2\alpha_{k}} \leq  \frac{1}{2} \text{~~and~~}
        P_{k+1}    = \frac{\lambda_{k+1}^2}{2\alpha_k} \le \frac{\delta^2 \alpha_k}{8} \le \alpha_{k} \le \alpha_{k}+ \eta P_{k}.
    \end{align*}
\end{enumerate}

In all three cases, we have shown that $2M_{k+1}\leq 1$ and $P_{k+1}\leq \alpha_k +\eta P_{k}$.
\end{proof}

Again, we emphasize that $\alpha_k$, $\theta_k$, $M_k$ and $P_k$ are defined in \eqref{bbsimsc}-\eqref{def-Mk-Pk-1}.
With these newly defined parameters, we still define $w_k = \alpha_k+P_k-P_{k+1}$ for $k\geq 0$ as in Lemma \ref{lemma-ener}, and
${\bf E}$,  $\varUpsilon_{k}$ and $\varPhi_k$ as in \eqref{e1}, \eqref{def-varUpsilon} and \eqref{energy}, respectively.
Next, we present the pointwise convergence of the gradient method \eqref{gd} with $\alpha_k$ given by \eqref{bbsimsc} and
establish bounds on $\alpha_k$ and $P_k/\alpha_k$. Note that since $f$ is locally strongly convex, it has a unique optimal solution.

\begin{thom}[Pointwise convergence] \label{thom:ic-strong}
For any $x^0\in\mathbb{R}^n$ and $\alpha_0>0$, let $\{x^{k}\}$ be the sequence generated by \eqref{gd} with $\alpha_k$ given by \eqref{bbsimsc}.
Then, $x^k \in B(x^*,R)$ for all $k\geq 0$, where $R$ is defined in \eqref{def-R}, and $\{x^k\}$ converges to the unique optimal solution $x^*$ of \eqref{p}.
\end{thom}
\begin{proof}
First, by following the proof of Lemma \ref{eq}, it is elementary to verify that \eqref{c3}, and thus \eqref{e1}, holds as well for  $\lambda_{k}< \alpha_{k-1}$.
Then, by following the proof of Lemma \ref{lemma-ener}, it is also easy to observe that \eqref{gds-follow-1} holds as well, with $\bf E$ defined in  \eqref{e1}.
Combining Lemma \ref{mpis}, which confirms that $2M_{k+1}\leq 1$ and $P_{k+1}\leq \alpha_k + \eta P_{k} \leq \alpha_k + P_{k}$ for all $k \geq 0$,
with \eqref{gds-follow-1}, we obtain \eqref{energy}, with $\varUpsilon_{k}$ and $w_k$ defined in \eqref{def-varUpsilon} and Lemma \ref{lemma-ener}, respectively.
Consequently, by following the same lines of proof as in Corollary \ref{co:bound} and Theorem \ref{thom:ic}, we can show that $x^k \in B(x^*,R)$ for all $k\geq 0$, where $R$ is defined in \eqref{def-R}, and $\{x^k\}$ converges to the unique optimal solution of \eqref{p}. The details are omitted due to the high similarity.
\end{proof}

\begin{proposition}[Bounds on $\alpha_k$ and $P_k/\alpha_k$] \label{alpha-sc}
Let $\{\alpha_k\}$ be generated by \eqref{bbsimsc} with any $\alpha_0>0$. Then, for $k\geq 1$, we have
(i) $c\leq \alpha_k \leq 1/\mu$, where $c:= \min\{\alpha_0, 1/(\sqrt{2}L)\} > 0$, and (ii)
$P_k/\alpha_k \geq c_0 := \min \{{c \mu },\delta-1\} >0$.
\end{proposition}

\begin{proof}
Let $k\geq 1$ be fixed.
(i) First, $\alpha_k\geq c =\min\{\alpha_0, 1/(\sqrt{2}L)\}$ follows from the same analysis as in Proposition \ref{lemma:sum-alpha} (ii).
Second, the definition of $\alpha_k$ in \eqref{bbsimsc} shows that $\alpha_k\leq \lambda_k$.
Further considering  \eqref{jy-01}, we obtain $\alpha_k \leq \lambda_k \leq 1/\mu$.
For part (ii), we split the analysis into two cases:
(a) $\lambda_k\geq \alpha_{k-1}$ or $\lambda_k\leq \delta\alpha_{k-1}/2$, and (b) $\delta\alpha_{k-1}/2 < \lambda_{k} < \alpha_{k-1}$.
For case (a), we have $ {P_k/\alpha_k} = {\alpha_k/\alpha_{k-1}}\geq c\mu$, where the inequality follows from $c\leq \alpha_k \leq 1/\mu$. For case (b), we have $P_{k}= \frac{2\alpha_{k}^{2}}{\alpha_{k-1}}-\alpha_k$, $\alpha_k = \lambda_k$, and hence ${P_{k}/\alpha_{k}} \geq \delta-1$. Combining these two cases completes the proof.
\end{proof}

Now, we are ready to establish the linear convergence result.
\begin{thom}[Linear convergence]
For any $x^0\in\mathbb{R}^n$ and $\alpha_0>0$,  the sequence $\{x^{k}\}$ generated by \eqref{gd} with $\alpha_k$ given by \eqref{bbsimsc}
converges linearly to the unique optimal solution $x^*$ of \eqref{p}.
\end{thom}
\begin{proof}
For convenience, we define for $k\geq 0$ that
\begin{equation}\label{def:Phik}
       \Phi_{k+1}:= \|x^{k+1}-x^*\|^2  +  2(1+\tfrac{\mu}{2L})M_{k+1}\|x^{k+1}-x^{k}\|^2 + \tfrac{2\alpha_{k}+2P_k}{\alpha_{k}+\eta P_k}P_{k+1} \big(f(x^{k})-f_*\big).
\end{equation}
Recall that $f_* = f(x^*)$ and $\nabla f(x^*)=0$.
It follows from \eqref{def:strong}, \eqref{local-smooth-1} and \eqref{gd} that
    \begin{align*}
    \alpha_{k}\langle \nabla f(x^{k}), x^{*}-x^{k}\rangle & \stackrel{\eqref{def:strong}}\leq \alpha_k \big(f_*-f(x^k)\big)-\tfrac{\alpha_k\mu}{2}\|x^k-x^*\|^2,\\
    \alpha_{k}\langle \nabla f(x^{k}), x^{*}-x^{k}\rangle & \stackrel{\eqref{local-smooth-1}}\leq \alpha_k\big(f_*-f(x^k)\big)-\tfrac{\alpha_k}{2 L}\|\nabla f(x^k)-\nabla f(x^*)\|^2 \\
    & \stackrel{\eqref{gd}}= \alpha_k\big(f_*-f(x^k)\big)-\tfrac{1}{2\alpha_k L}\|x^{k+1}-x^{k}\|^2 \\
    & \leq \alpha_k\big(f_*-f(x^k)\big)-\tfrac{\mu}{2L}\|x^{k+1}-x^{k}\|^2,
    \end{align*}
    where the last ``$\leq$" is due to $\alpha_k \leq 1/\mu$.
Combining the above two inequalities to obtain
    \begin{equation}
    \alpha_{k}\langle \nabla f(x^{k}), x^{*}-x^{k}\rangle \leq \alpha_k\big(f_*-f(x^k)\big) - \tfrac{\alpha_k\mu}{4}\|x^k-x^*\|^2-\tfrac{\mu}{4L}\|x^{k+1}-x^k\|^2.   \label{modf}
    \end{equation}
Plugging \eqref{modf} into \eqref{gds}, we arrive at
    \begin{equation}
    \begin{aligned} \label{gdss}
    \|x^{k+1}& -x^{*}\|^{2} 
     \leq (1-\tfrac{\alpha_k\mu}{2}) \|x^{k}-x^{*}\|^{2}-2 \alpha_{k}\left(f(x^{k})-f_*\right)-\tfrac{\mu}{2L}\|x^{k+1}-x^k\|^2 + \alpha_{k}^{2}\left\|\nabla f(x^{k})\right\|^{2}.
     \end{aligned}
\end{equation}
By summing \eqref{gdss} and \eqref{gds-follow-1},  considering the definition of $\bf E$ in \eqref{e1}, and reorganizing terms, we can easily derive
 \begin{equation}
    \begin{aligned} \label{e11}
    (1-\tfrac{\alpha_k\mu}{2})\|x^{k}&-x^*\|^2  + 2M_{k}\|x^{k}-x^{k-1}\|^2+2P_{k}\big(f(x^{k-1})-f_*\big) \\
  \geq \, & \|x^{k+1}-x^*\|^2  + (1+\tfrac{\mu}{2L})\|x^{k+1} -x^{k}\|^2 + (2\alpha_{k}+2P_{k}) \big(f(x^{k})-f_*\big)
       \geq
    \Phi_{k+1},
   \end{aligned}
\end{equation}
where the second ``$\geq$" follows from  \eqref{def:Phik} and Lemma \ref{mpis}.
Define
\[
c_{1,k} := \max\big\{1-\tfrac{\alpha_k\mu}{2},1/(1+\tfrac{\mu}{2L}),\tfrac{\alpha_{k-1}+\eta P_{k-1}}{\alpha_{k-1}+P_{k-1}}\big\} > 0
\text{~~and~~}
c_1:= \max\big\{1-\tfrac{c\mu}{2},1/(1+\tfrac{\mu}{2L}),1-\tfrac{c_0 (1-\eta)}{c_0+1}\big\},
\]
where $c, c_0>0$ are defined in Proposition \ref{alpha-sc}.
From Proposition \ref{alpha-sc} (i), we have $\alpha_k\geq c$ and thus  $1-\tfrac{\alpha_k\mu}{2}\leq 1-\tfrac{c\mu}{2} < 1$. On the other hand,
from Proposition \ref{alpha-sc} (ii) we have $\alpha_{k-1}/P_{k-1} \leq 1/c_0$  for $k\geq 2$, and hence
     \[
\tfrac{\alpha_{k-1}+\eta P_{k-1}}{\alpha_{k-1}+P_{k-1}}=1-\tfrac{1-\eta }{(\alpha_{k-1}/P_{k-1})+1}
    \leq 1- \tfrac{c_0(1-\eta)}{c_0+1} < 1.
     \]
Therefore, we have shown that $c_{1,k} \leq c_1  < 1$ for all $k\geq 2$.
It then follows from \eqref{def:Phik}, \eqref{e11} and the definition of $c_{1,k}$ that $\Phi_{k+1} \leq c_{1,k} \, \Phi_{k} \leq c_1 \Phi_{k}$ for all $k\geq 2$.
Again, it follows from the definition of  \eqref{def:Phik} and $c_1  < 1$ that $\{x^k\}$ converges linearly to the unique optimal solution of \eqref{p}.
\end{proof}

\subsection{Composite Convex Optimization Problems}
Let $g: \mathbb{R}^n \rightarrow \mathbb{R}$ be an extended real-valued closed, proper and convex function, which may be non-smooth.
In this subsection, we extend AdaBB (Algorithm \ref{alg:AdaBB}) to solve the composite convex optimization problem
\begin{equation} \label{pcom}
    \min_{x\in\mathbb{R}^n} F(x):= f(x)+g(x),
\end{equation}
where
$f: \mathbb{R}^n \rightarrow \mathbb{R}$ is the same as in \eqref{p}. In particular, $f$ is a locally $L$-smooth function convex function satisfying \eqref{local-smooth-1} in which the radius $R := T$ and $T$ is defined in \eqref{def-T}.  We assume that the set of optimal solutions of \eqref{pcom}, also denoted by $\mathcal{X}^{*}$,  is non-empty and denote the optimal value of $F$ by $F_*$. In this section, we consider the proximal gradient method of the form
\begin{equation}\label{pg}
x^{k+1} = \text{prox}_{\alpha_k g}(x^{k} - \alpha_{k}\nabla f(x^k)), \quad k\geq 0,
\end{equation}
where $\alpha_k>0$ denotes the stepsize and will be chosen adaptively, and for given $\alpha>0$, $\text{prox}_{\alpha g}(\cdot)$ is defined by
\[
\text{prox}_{\alpha g}(x) = \arg\min\nolimits_{y\in\mathbb{R}^n} g(y) + \frac{1}{2\alpha}\|y-x\|^2, \quad x\in\mathbb{R}^n.
\]
An equivalent implicit form of \eqref{pg} is given by
\begin{equation} \label{pgd}
    x^{k+1} = x^{k} -\alpha_{k}\big(\nabla f(x^k)+\xi^{k+1}\big) \text{~~for some~~} \xi^{k+1} \in \partial g(x^{k+1}).
\end{equation}

Our adaptive proximal BB method (AdaPBB) for solving \eqref{pcom} is presented in Algorithm \ref{agbbc}.
\begin{algorithm}[!htb]
\caption{Adaptive Proximal BB Method (AdaPBB)} \label{agbbc}
\textbf{Input:} $x^0\in\mathbb{R}^n$, $\alpha_0 > 0$, $\theta_0 \geq 0$
\begin{algorithmic}[1]
\State  $x^{1} = \text{prox}_{\alpha_0 g}(x^{1} - \alpha_{0}\nabla f(x^0))$
\For{$k=1,2,\ldots,$}
\State $\lambda_{k} = \frac{\langle \nabla f(x^{k}) - \nabla f(x^{k-1}), \, x^{k}-x^{k-1} \rangle}{\|\nabla f(x^{k})- \nabla f(x^{k-1})\|^2}$
\If{$\lambda_k \geq \alpha_{k-1}$} \dotfill (Case i)
\State $\alpha_k = \sqrt{1+\theta_{k-1}}\alpha_{k-1}$, and $\theta_{k} = \frac{\alpha_{k}}{\alpha_{k-1}}$
\ElsIf{$\alpha_{k-1}/2 <\lambda_k < \alpha_{k-1}$} \dotfill (Case ii)
\State $\alpha_k = \frac{\alpha_{k-1}}{\sqrt{2}}$, and $\theta_{k} = 0 $
\Else \dotfill (Case iii)
\State $\alpha_k = \frac{\lambda_k}{\sqrt{2}}$, and $\theta_{k} = 0$
\EndIf
\State $x^{k+1} = \text{prox}_{\alpha_k g}(x^{k} - \alpha_{k}\nabla f(x^k))$
\EndFor
\end{algorithmic}
\end{algorithm}

\begin{remark} \label{remark-5.1}
It is worth noting that in Algorithm \ref{agbbc}, if $\lambda_{k} < \alpha_{k-1}$ (Cases ii and iii), then $\alpha_{k}\geq \lambda_{k}/{\sqrt{2}}\geq 1/{(\sqrt{2}L)}$.
Similar discussions in Proposition \ref{lemma:sum-alpha} will be formally presented in Proposition \ref{lemma:sum-alpha2} later.
\end{remark}

Before analyzing the convergence of Algorithm \ref{agbbc}, we define some useful notation and recall some important inequalities for the scheme \eqref{pgd}.
For $k\geq 1$, we  define $B_k$ and $E_k$ as follows
\begin{equation}
\left\{
\begin{array}{lll}
B_{k}:=0 ,  & E_{k}:= \frac{1}{\alpha_{k-1}},     & {\text{if }\lambda_{k}\geq \alpha_{k-1}}, \\
B_{k}:=1,   & E_{k}:=0,   & {\text{if }\frac{\alpha_{k-1}}{2}<\lambda_{k}\leq  \alpha_{k-1}}, \\
B_{k}:=\frac{(\alpha_{k-1}-\lambda_{k})^2}{\lambda_k^2},   & E_{k}:=0, & {\text{if $0 <\lambda_{k}\leq \frac{\alpha_{k-1}}{2}$}}.
\end{array} \right. \label{efc}
\end{equation}
For convenience, we also define $E_0:=(E_1\alpha_1^2-\alpha_0)/\alpha_0^2$.
The following lemma provides useful inequalities for $B_k$ and $E_k$ defined in \eqref{efc}.

\begin{lemma} \label{cbei}
Let $\{\alpha_k\}$ and $\{\theta_k\}$ be generated by Algorithm \ref{agbbc} and $B_k$ and $E_k$ be defined in \eqref{efc}. Then,
$\alpha_k E_{k} = \theta_{k}$ for $k\geq 1$, and  $2B_{k+1}\leq \alpha_{k}^2/\alpha_{k+1}^2$ and $E_{k+1}\alpha_{k+1}^2\leq E_{k}\alpha_k^2+\alpha_k$ for $k\geq 0$.
\end{lemma}
\begin{proof}
The fact that $\alpha_k E_{k} = \theta_{k}$ for $k \geq 1$ is obvious.
Let $k\geq 0$ be fixed. To establish the remaining results, we split the analysis into the following three cases.
Case (i):  $\lambda_{k+1}\geq \alpha_{k}$. In this case, we have $B_{k+1}=0$, $E_{k+1}=1/\alpha_{k}$, $\alpha_{k+1} = \sqrt{1+\theta_{k}}\alpha_{k}$,
and thus $E_{k+1}\alpha_{k+1}^2 = (1+\theta_k)\alpha_k = E_{k}\alpha_k^2+\alpha_k$.
Case (ii):  $\alpha_{k}/2 < \lambda_{k+1} < \alpha_{k}$.  In this case, we have $B_{k+1}=1$, $E_{k+1}=0$, $\alpha_{k+1}=\alpha_{k}/\sqrt{2}$, and thus
$2B_{k+1} = \alpha_{k}^2/\alpha_{k+1}^2$.
Case (iii):  $\lambda_{k+1}\leq \alpha_{k}/2$. In this case, we have  $B_{k+1}=(\alpha_{k}-\lambda_{k+1})^2/\lambda_{k+1}^2$, $E_{k+1}=0$, $\alpha_{k+1}=\lambda_{k+1}/\sqrt{2}$, and thus
$2B_{k+1} = {2(\alpha_{k}-\lambda_{k+1})^2 / \lambda_{k+1}^2} \leq  {2\alpha_{k}^2 / \lambda_{k+1}^2} = {\alpha_{k}^2/\alpha_{k+1}^2}$.
The proof is completed by combining the above three cases.
\end{proof}

The following results are taken from \cite{MM23}. Note that a refined inequality that improves upon \eqref{ei1} will be derived in the proof of Theorem \ref{iterative_convergence}.

\begin{lemma}[{\cite[Eq. (34) and Lemmas 11-12]{MM23}}]
Let $\{x^{k}\}$ be generated by \eqref{pgd} with arbitrarily positive stepsizes $\{\alpha_{k}\}$. Then, for $k\geq 0$, we have
    \begin{align}
    &\|x^{k+1}-x^{*}\|^2 + 2\alpha_{k}(F(x^{k})-F_*)  \leq \|x^{k}-x^*\|^2 + \alpha_{k}^2  \|\nabla f(x^k)+\xi^k\|^2, \label{ei1}\\
    &\|\nabla f(x^{k})+\xi^{k+1}\|^2 = \langle \nabla f(x^{k+1})+\xi^{k+1},\nabla f(x^{k})+\xi^{k+1}\rangle + \frac{1}{\alpha_{k}}\langle \nabla f(x^{k+1}) -\nabla f(x^{k}), x^{k+1}-x^{k} \rangle, \label{ei2}\\
    &\|\nabla f(x^{k})+\xi^{k+1}\| \leq \|\nabla f(x^k)+\xi^{k}\|.  \label{ei3}
    \end{align}
\end{lemma}
\begin{proof}
The result \eqref{ei1} follows from \cite[Eq. (34)]{MM23},
\eqref{ei2} is taken from the proof of \cite[Lemma 11]{MM23},
and \eqref{ei3} is given as Lemma 12 in \cite{MM23}.
\end{proof}

\begin{lemma}[Analogous to \eqref{e1}]\label{ceq}
Let $\{x^{k}\}$ be generated by Algorithm \ref{agbbc} and  $B_{k}$ and $E_k$ are defined in \eqref{efc}. Then, for $k\geq 1$, we have
\begin{equation}\label{jy-02}
\|\nabla f(x^k)+\xi^k\|^2\leq B_{k}\|\nabla f(x^{k-1})+\xi^{k-1}\|^2 + E_k\big(F(x^{k-1})-F(x^{k})\big).
\end{equation}
\end{lemma}
\begin{proof}
From the equality $\|a\|^2=\|a-b\|^2-\|b\|^2+2\langle a,b\rangle$, we obtain
\begin{align}
    \|\nabla f(x^k)&+\xi^k\|^2  = \|\nabla f(x^k)-\nabla f(x^{k-1})\|^2 - \|\nabla f(x^{k-1})+\xi^{k}\|^2+2\langle \nabla f(x^k)+\xi^{k}, \nabla f(x^{k-1})+\xi^{k} \rangle \nonumber \\
    &   =  \big(\frac{1}{\lambda_{k}}-\frac{1}{\alpha_{k-1}}\big)\langle \nabla f(x^{k}) -\nabla f(x^{k-1}), x^{k}-x^{k-1} \rangle
    + \langle \nabla f(x^k)+\xi^{k}, \nabla f(x^{k-1})+\xi^{k} \rangle,
    \label{eq1}
\end{align}
where the second equality follows from \eqref{lamk} and \eqref{pgd}.
It follows from \eqref{ei2} and \eqref{ei3} that
\begin{align}
\langle \nabla f(x^k)+\xi^{k}, \nabla f(x^{k-1})+\xi^{k} \rangle \stackrel{\eqref{ei2}}  = & \|\nabla f(x^{k-1})+\xi^k\|^2 - \frac{1}{\alpha_{k-1}}\langle \nabla f(x^{k}) -\nabla f(x^{k-1}), x^{k}-x^{k-1} \rangle  \nonumber \\
\stackrel{\eqref{ei3}} \leq & \|\nabla f(x^{k-1})+\xi^{k-1}\|^2- \frac{1}{\alpha_{k-1}}\langle \nabla f(x^{k}) -\nabla f(x^{k-1}), x^{k}-x^{k-1} \rangle. \label{eq2}
\end{align}
Combining \eqref{eq1} and \eqref{eq2}, we obtain
\begin{equation} \label{case23}
\|\nabla f(x^k)+\xi^k\|^2 \leq \big(\frac{1}{\lambda_{k}}-\frac{2}{\alpha_{k-1}}\big)\langle \nabla f(x^{k}) -\nabla f(x^{k-1}), x^{k}-x^{k-1} \rangle + \|\nabla f(x^{k-1})+\xi^{k-1}\|^2.
\end{equation}
We then prove the desired result \eqref{jy-02} by analyzing the following three cases.
\begin{itemize}[leftmargin=*]
\item Case (i): $\lambda_k \geq \alpha_{k-1}$. In this case, we have $1/\lambda_k -1/ \alpha_{k-1} \leq 0$. Since
 $\nabla f(x^{k-1})+\xi^{k} = (x^{k-1}-x^{k})/\alpha_{k-1}$ and $\langle \nabla f(x^{k}) -\nabla f(x^{k-1}), x^{k}-x^{k-1} \rangle\geq 0$, we
obtain from \eqref{eq1} that
\begin{equation}
\|\nabla f(x^k)+\xi^k\|^2 \leq \tfrac{1}{\alpha_{k-1}} \langle \nabla f(x^k)+\xi^{k}, x^{k-1}-x^{k} \rangle \leq \tfrac{1}{\alpha_{k-1}}\big(F(x^{k-1})-F(x^{k})\big), \label{ic1}
\end{equation}
where the second ``$\leq$" is due to the convexity of $F$ and $\nabla f(x^k)+\xi^{k}\in\partial F(x^k)$.

\item Case (ii): ${\alpha_{k-1}/2} < \lambda_k < \alpha_{k-1}$. In this case, we have $1/\lambda_k - 2/\alpha_k\leq 0$ and
\eqref{case23} implies
\begin{equation}
\|\nabla f(x^k)+\xi^k\|^2 \leq \|\nabla f(x^{k-1})+\xi^{k-1}\|^2. \label{ic2}
\end{equation}
\item Case (iii): $0<\lambda_k \leq {\alpha_{k-1}/2}$. Then, \eqref{lamk} and \eqref{case23} imply
\begin{align}
\|\nabla f(x^k)+\xi^k\|^2  & \leq \frac{1}{\lambda_{k}}\big(\frac{1}{\lambda_{k}}-\frac{2}{\alpha_{k-1}}\big)\frac{\langle \nabla f(x^{k}) -\nabla f(x^{k-1}), x^{k}-x^{k-1} \rangle ^2}{\|\nabla f(x^{k}) -\nabla f(x^{k-1})\|^2} + \|\nabla f(x^{k-1})+\xi^{k-1}\|^2 \nonumber \\
& \leq \frac{1}{\lambda_{k}}\big(\frac{1}{\lambda_{k}}-\frac{2}{\alpha_{k-1}}\big)\|x^{k}-x^{k-1}\|^2 +     \|\nabla f(x^{k-1})+\xi^{k-1}\|^2 \nonumber \\
& \stackrel{\eqref{pgd}}= \frac{\alpha_{k-1}^2}{\lambda_{k}}\big(\frac{1}{\lambda_{k}}-\frac{2}{\alpha_{k-1}}\big)\|\nabla f(x^{k-1})+\xi^{k}\|^2 + \|\nabla f(x^{k-1})+\xi^{k-1}\|^2\nonumber \\
& \stackrel{\eqref{ei3}} \leq \frac{(\alpha_{k-1}-\lambda_{k})^2}{\lambda_k^2} \|\nabla f(x^{k-1})+\xi^{k-1}\|^2.\label{ic3}
\end{align}
\end{itemize}
The desired result \eqref{jy-02} follows immediately by combining \eqref{ic1}-\eqref{ic3} with \eqref{efc}.
\end{proof}
We are now ready to establish a result that is analogous to \eqref{energy} for problem \eqref{p}.
For this purpose, in the rest of this section, we let $x^*\in{\cal X^*}$ be an arbitrarily fixed solution of \eqref{pcom} and define for $k\geq 1$ that
\begin{equation}
\left\{
\begin{array}{l}
w_k := \alpha_k+E_k\alpha_k^2-E_{k+1}\alpha_{k+1}^2,\\
V_{k} := \|x^{k}-x^*\|^2  +  2B_{k}\alpha_{k}^2 \|\nabla f(x^{k-1})+g^{k-1}\|^2 + 2\alpha_{k-1}(1+E_{k-1}\alpha_{k-1}) \big(F(x^{k-1})-F_*\big),\\
U_k := V_{k}-2w_{k-1}\big(F(x^{k-1})-F_*\big).
\end{array} \right. \label{w-v-u}
\end{equation}
It is obvious from Lemma \ref{cbei} that $w_k\geq 0$ for $k\geq 0$.
Furthermore, direct calculations show that
\begin{equation}
  \label{jy-04}
  U_k  = \|x^{k}-x^*\|^2  +  2B_{k}\alpha_{k}^2 \|\nabla f(x^{k-1})+\xi^{k-1}\|^2 + 2E_k\alpha_{k}^2  \big(F(x^{k-1})-F_*\big).
\end{equation}

\begin{lemma}[Analogous to Lemma \ref{lemma-ener}]\label{enerc}
Let $\{x^{k}\}$ be generated by Algorithm \ref{agbbc} and  $U_{k}$ and $V_k$ are defined in \eqref{w-v-u}. Then, for $k\geq 1$, we have
$V_{k+1}\leq U_k  \leq V_{k}$.
\end{lemma}
\begin{proof}
Let $k\geq 1$ be fixed.
First, \eqref{jy-02} is equivalent to
\begin{equation}\label{jy-03}
\|\nabla f(x^k)+\xi^k\|^2\leq 2B_{k}\|\nabla f(x^{k-1})+\xi^{k-1}\|^2 + 2E_k\big(F(x^{k-1})-F(x^{k})\big)-\|\nabla f(x^k)+\xi^k\|^2.
\end{equation}
By utilizing equation \eqref{jy-03} to expand the term $\|\nabla f(x^k) + \xi^k\|^2$ on the right-hand-side of equation \eqref{ei1}, and considering the definitions of $U_k$, $V_k$, and $w_k$ provided in equation \eqref{w-v-u}, we can rearrange the terms and perform elementary calculations to obtain the following inequality:
\begin{equation*}
U_k\geq    \|x^{k+1}-x^*\|^2 + \alpha_{k}^2\|\nabla f(x^{k})+\xi^{k}\|^2 + 2\alpha_{k}(1+E_{k}\alpha_{k}) \big(F(x^{k})-F_*\big) \geq V_{k+1},
\end{equation*}
where the second ``$\geq$" follows from Lemma \ref{cbei} and the definition of $V_k$ in \eqref{w-v-u}.
Finally, considering   $w_i\geq 0$ for all $i\geq 0$, we obtain $V_{k+1}\leq U_k = V_{k}-2w_{k-1}\big(F(x^{k-1})-F_*\big) \leq V_{k}$.
\end{proof}

\begin{corollary}[Analogous to Corollary \ref{co:bound}] \label{co:bound2}
Let $\{x^{k}\}$ be generated by Algorithm \ref{agbbc}. Then, $\{x^{k}\}$ is bounded. In particular, $x^k \in B(x^*,T)$ for all $k\geq 0$, where $T$ is defined as:
\begin{equation}\label{def-T}
T^2 := \|x^0-x^*\|^2+ 2\alpha_0^2\|\nabla f(x^0)+\xi^0\|^2 + \max\{2(E_1\alpha_1^2-\alpha_0),0\}  \big(F(x^0)-F^*\big).
\end{equation}
\end{corollary}
\begin{proof}
It follows from \eqref{jy-04} and Lemma \ref{enerc} that
$ \|x^{k}-x^*\|^2 \leq U_k \leq V_{k} \leq U_{k-1} \leq \cdots \leq U_1$ for all $k \geq 1$.
Setting $k=1$ in \eqref{jy-04} and \eqref{ei1}, and using $2B_{1}\alpha_{1}^2\leq \alpha_{0}^2$ from Lemma \ref{cbei}, we obtain
    \begin{align*}
    U_1 & = \|x^{1}-x^*\|^2  +  2B_{1}\alpha_{1}^2 \|\nabla f(x^{0})+\xi^{0}\|^2 + 2 E_1 \alpha_1^2 \big(F(x^{0})-F_*\big) \\
    & \stackrel{\eqref{ei1}}\leq  \|x^{0}-x^*\|^2  +  2\alpha_0^2 \|\nabla f(x^{0})+\xi^{0}\|^2 + 2 (E_1 \alpha_1^2-\alpha_0) \big(F(x^{0})-F_*\big).
    \end{align*}
Moreover, analogous to Remark \ref{remark-M1-P1}, we can claim that $E_1\alpha_1^2$ is a constant entirely determined by $x^0$ and $\alpha_0$.
Furthermore, it is trivial to observe that $\|x^0-x^*\|\leq T$.
Combining the above arguments, we conclude that $\|x^k-x^*\|\leq T$ for all $k \geq 0$.
\end{proof}

\begin{proposition}[The same as Proposition \ref{lemma:sum-alpha}] \label{lemma:sum-alpha2}
For $\{\alpha_k\}$ generated by Algorithm \ref{agbbc}, we have
(i) if $\alpha_j\geq\frac{1}{\sqrt{2}L}$ for some $j$, then $\alpha_k\geq\frac{1}{\sqrt{2}L}$ for any $k\geq j$;
(ii)  $\alpha_k\geq c:= \min\{\alpha_0,\frac{1}{\sqrt{2}L}\}$ for all $k\geq 0$; and
(iii) $\sum_{i=1}^k \alpha_i = O(k)$.
\end{proposition}
\begin{proof}
The proof is highly similar to that of Proposition \ref{lemma:sum-alpha} and is thus omitted.
\end{proof}

Now, we are ready to derive the ergodic sublinear convergence result of Algorithm \ref{agbbc}.

\begin{thom}[Analogous to Theorem \ref{thom1}] \label{thomc}
Let $\{x^{k}\}$ be generated by Algorithm \ref{agbbc}. Then, we have
\[
F(\Bar{x}^{k})-F_* \leq \frac{U_1}{2S_{k}} = O\left(\frac{1}{k}\right),
\]
where $\Bar{x}^{k} := \big(\alpha_k(1+E_k\alpha_k)x^{k}+\sum_{i=1}^{k-1}w_ix^{i}\big)/S_{k}$ with
$S_k :=  E_1\alpha_1^2+ \sum_{i=1}^k \alpha_i$.
\end{thom}

\begin{proof}
    The proof is similar to that of Theorem \ref{thom1}. By telescoping $V_{i+1}\leq V_{i}-2 w_{i-1}\big(F(x^{i-1})-F_*\big)$ given in Lemma \ref{enerc} for $i=2,\ldots,k$, we derive $V_{k+1} + 2\sum_{i=1}^{k-1}w_i\big(F(x^i)-F_*\big) \leq V_{2} \leq U_1$.
    Taking into account the definition of $V_k$ in \eqref{w-v-u}, we further derive
    \begin{equation}\label{proof-thomc1-eq1}
        \alpha_{k}(1+E_{k}\alpha_{k}) \big(F(x^{k})-F_*\big) +  \sum\nolimits_{i=1}^{k-1}w_i\big(F(x^i)-F_*\big)\leq  \tfrac{U_1}{2}.
    \end{equation}
Again, from \eqref{w-v-u} we have
$\alpha_k (1+ E_{k}\alpha_{k}) + \sum\nolimits_{i=1}^{k-1}w_i =  E_1\alpha_1^2+ \sum\nolimits_{i=1}^k \alpha_i =S_k \sim {\cal O}(k)$.
The desired result follows from \eqref{proof-thomc1-eq1}, the convexity of $F$, and Jensen's inequality.
\end{proof}

Before establishing the pointwise convergence of Algorithm \ref{agbbc}, we derive a useful inequality.
\begin{lemma}
Let $\{x^k\}$ be generated by Algorithm \ref{agbbc}. Then, for any $k\geq 1$ we have
\begin{equation}\label{icon2}
\frac{1}{\alpha_k}\|x^{k+1}-x^{k}\|^2 \leq 4\sqrt{2}\alpha_{k+1}\|\nabla f(x^{k+1})-\nabla f(x^{*})\|^2 + 4\alpha_{k}\|\nabla f(x^{k})-\nabla f(x^{*})\|^2 + 2(F(x^{k})-F(x^{k+1})).
\end{equation}
\end{lemma}
\begin{proof}
Let $k\geq 1$ and $x \in \mathbb{R}^n$ be arbitrarily fixed.
It follows from \eqref{pg} that
\begin{equation}\label{jy-08}
\alpha_k\big(g(x^{k+1})-g(x)\big)\leq \langle x^{k+1}-x^{k} + \alpha_k \nabla f(x^{k}),x-x^{k+1}\rangle.
\end{equation}
It follows from \eqref{jy-08}, the convexity of $f$  and $F = f + g$ that
\begin{equation} \label{icon}
F(x^{k+1})-F(x) \leq \frac{1}{\alpha_{k}}\langle x^{k+1}-x^{k},x-x^{k+1} \rangle +
\langle \nabla f(x^{k})-\nabla f(x^{k+1}),x-x^{k+1} \rangle.
\end{equation}
Setting $x=x^k$ in \eqref{icon} and using the definition of $\lambda_{k+1}$ in \eqref{lamk}, we derive
\begin{equation}
    \frac{1}{\alpha_k}\|x^{k+1}-x^{k}\|^2  \leq \lambda_{k+1}\|\nabla f(x^{k+1})-\nabla f(x^{k})\|^2 + F(x^{k})-F(x^{k+1}). \label{ico21}
\end{equation}
Furthermore, plugging $\langle \nabla f(x^{k})-\nabla f(x^{k+1}),x^k-x^{k+1} \rangle \leq \frac{\alpha_{k}}{2}\|\nabla f(x^{k})-\nabla f(x^{k+1})\|^2+\frac{1}{2\alpha_k}\|x^k-x^{k+1}\|^2$ into \eqref{icon} with $x=x^k$ to obtain
\begin{equation}
\frac{1}{\alpha_k}\|x^{k+1}-x^{k}\|^2  \leq
\alpha_k\|\nabla f(x^{k+1})-\nabla f(x^{k})\|^2
+ 2\big(F(x^{k})-F(x^{k+1})\big). \label{ico22}
\end{equation}
We split the discussion into two cases. (i) If $\lambda_{k+1}\geq \alpha_{k}$, then \eqref{ico22} implies
\begin{align}
\frac{1}{\alpha_k}& \|x^{k+1}-x^{k}\|^2 \leq 2\alpha_{k}\|\nabla f(x^{k+1})-\nabla f(x^{*})\|^2 + 2\alpha_k\|\nabla f(x^{k})-\nabla f(x^{*})\|^2 + 2\big(F(x^{k})-F(x^{k+1})\big)  \nonumber\\
& \leq 2\alpha_{k+1}\|\nabla f(x^{k+1})-\nabla f(x^{*})\|^2 + 2\alpha_k\|\nabla f(x^{k})-\nabla f(x^{*})\|^2 + 2\big(F(x^{k})-F(x^{k+1})\big), \label{x1}
\end{align}
where the first ``$\leq$"  uses $\|a-b\|^2\leq 2\|a-c\|^2 + 2\|b-c\|^2$, and the second  is due to  $\alpha_{k+1}=\sqrt{1+\theta_{k}}\alpha_{k}\geq \alpha_k$.
(ii) If $\lambda_{k+1}<\alpha_{k}$, then $\alpha_{k+1}\geq \frac{1}{\sqrt{2}}\lambda_{k+1}$ and \eqref{ico21} implies
\begin{align}
\frac{1}{\alpha_k}\|x^{k+1} & -x^{k}\|^2\leq 2\lambda_{k+1}\|\nabla f(x^{k+1})-\nabla f(x^{*})\|^2 + 2\lambda_{k+1}\|\nabla f(x^{k})-\nabla f(x^{*})\|^2 + F(x^{k})-F(x^{k+1}) \nonumber \\
& \leq 2\sqrt{2}\alpha_{k+1}\|\nabla f(x^{k+1})-\nabla f(x^{*})\|^2 + 2\alpha_k\|\nabla f(x^{k})-\nabla f(x^{*})\|^2 + F(x^{k})-F(x^{k+1}). \label{x2}
\end{align}
Apparently the expression on the right-hand side of \eqref{x2} is nonnegative, allowing us to expand it further by multiplying by a factor of $2$.
Moreover, considering \eqref{x1}, we can derive \eqref{icon2} in both scenarios.
\end{proof}

\begin{thom}[Analogous to Theorem \ref{thom:ic}]\label{iterative_convergence}
The sequence $\{x^{k}\}$ generated by Algorithm \ref{agbbc} converges to an optimal solution of \eqref{pcom}.
\end{thom}
\begin{proof}
First, we derive a refined inequality of \eqref{ei1}. Let $k\geq 1$ be arbitrarily fixed.
By setting $x=x^*$ in \eqref{jy-08}, we obtain
$\alpha_k(g(x^{k+1})-g(x^*)) \leq \langle x^{k+1}-x^k+\alpha_k\nabla f(x^k), x^*-x^{k+1}\rangle$,
which can be equivalently reformulated as
\begin{equation}\label{jy-07}
\|x^{k+1}-x^*\|^2 + 2 \alpha_k(g(x^{k+1})-g(x^*)) \leq \|x^{k}-x^*\|^2 + 2\alpha_k \langle \nabla f(x^k), x^*-x^{k+1}\rangle -
\|x^{k+1}-x^k\|^2.
\end{equation}
By using the inequality in \eqref{local-smooth-1} over $B(x^*,T)$ and the convexity of $g$, we obtain
\begin{equation}\label{jy-06}
\begin{aligned}
 \langle \nabla &f(x^k), x^*-x^{k+1}\rangle
= \langle \nabla f(x^k), x^*-x^k\rangle + \langle \nabla f(x^k) + \xi^k, x^k-x^{k+1}\rangle
+ \langle \xi^k, x^{k+1}-x^k\rangle \\
&\leq f(x^*) - f(x^k) - \frac{1}{2L}\|\nabla f(x^k)-\nabla f(x^*)\|^2 + \langle \nabla f(x^k) + \xi^k, x^k-x^{k+1}\rangle
+ g(x^{k+1})-g(x^k).
\end{aligned}
\end{equation}
Combining \eqref{jy-06} and \eqref{jy-07}, using $2\alpha_k\langle \nabla f(x^k) + \xi^k, x^k-x^{k+1}\rangle - \|x^{k+1}-x^k\|^2 \leq \alpha_k^2\|\nabla f(x^k) +\xi^k\|^2$, taking   into account $F = f + g$, and reorganizing terms, we obtain a refined inequality of
\eqref{ei1}:
\begin{equation}\label{jy-05}
\|x^{k+1}-x^{*}\|^2 + 2\alpha_{k}(F(x^{k})-F_*) + \frac{\alpha_k}{L}\|\nabla f(x^k)-\nabla f(x^{*})\|^2  \leq \|x^{k}-x^*\|^2 + \alpha_{k}^2  \|\nabla f(x^k)+\xi^k\|^2.
\end{equation}
Then, by using \eqref{jy-05} in place of \eqref{ei1} in
the proof of Lemma \ref{enerc}, we can derive
\begin{equation}\label{jy-11}
V_{k+1} + \frac{\alpha_k}{L}\|\nabla f(x^k)-\nabla f(x^{*})\|^2 \leq U_k\leq V_{k},
\end{equation}
which holds for all $k\geq 1$.
Telescoping this inequality leads to
\begin{equation} \label{a-f}
\sum_{k=1}^{\infty}\alpha_{k}\|\nabla f(x^{k})-\nabla f(x^{*})\|^2 \leq V_1 < +\infty.
\end{equation}
Since $\alpha_k\geq c > 0$ from Proposition \ref{lemma:sum-alpha2} (ii),    \eqref{a-f}  implies that $\lim\limits_{k\rightarrow \infty}\nabla f(x^{k}) = \nabla f(x^{*})$. We next split the proof into three cases, which cover all possibilities:
\begin{itemize}
    \item[(a).] $\liminf\nolimits_{k\rightarrow \infty} \alpha_k = \limsup\nolimits_{k\rightarrow \infty} \alpha_k = +\infty$, i.e.,  $\lim\nolimits_{k\rightarrow \infty} \alpha_k = +\infty$;
    \item[(b).] $\limsup\nolimits_{k\rightarrow \infty} \alpha_k<+\infty$, i.e., there exists a constant $C>0$, such that $\alpha_{k}\leq C$ for all $k$;
    \item[(c).] $\limsup\nolimits_{k\rightarrow \infty} \alpha_k= +\infty$, yet $\liminf\nolimits_{k\rightarrow \infty} \alpha_k < +\infty$, i.e., there exists an infinite set $K$ such that its complementary set $\overline{K} :=\{1,2,\ldots\}\backslash K$ is also infinite,
$\lim\nolimits_{K\ni k\rightarrow \infty} \alpha_k = +\infty$ and, for some $C>0$, $\alpha_{k}\leq C$ for all $k\in \overline{K}$.
\end{itemize}
First, we assume case (c) holds and show that $\lim_{K\ni k\rightarrow\infty} F(x^k) = \lim_{ \overline{K} \ni k\rightarrow\infty} F(x^k) = F_*$, and thus $\lim_{k\rightarrow\infty} F(x^k) = F_*$.
Recall that $\{x^{k}\}$ is bounded and $\|\nabla f(x^{k+1})-\nabla f(x^{k})\| \rightarrow 0$ as $k\rightarrow\infty$ since $\{\nabla f(x^k)\}$ converges.
By setting $x = x^*$ in \eqref{icon} and then taking the limit ``$K\ni k \rightarrow \infty$" on both sides of \eqref{icon}, we derive
$0\leq \lim_{K\ni k\rightarrow \infty} \big(F(x^{k+1})-F_*\big) \leq 0$ and thus $\lim_{K\ni k\rightarrow \infty} F(x^{k+1})=F_*$.
On the other hand,   by taking sum over $k=1,2,\ldots$ on both sides of \eqref{icon2}, using $\alpha_k\leq C$ for all $k\in \overline{K}$,
 $F(x^k)\geq F_*$ for all $k\geq 0$, and \eqref{a-f}, we obtain
\begin{equation}\label{jy-09}
    \frac{1}{C}\sum_{k\in\overline{K}} \|x^{k+1}-x^k\|^2 \leq \sum_{k=1}^{\infty }\frac{1}{\alpha_k}\|x^{k+1}   -x^{k}\|^2 \leq 4\sqrt{2}  V_1 + 4 V_1
     + 2\big(F(x^1)-F_*\big) < \infty.
\end{equation}
Note that $\overline{K}$ is infinite. Thus, \eqref{jy-09} implies that   $\lim\nolimits_{\overline{K}\ni k \rightarrow \infty}\|x^{k+1}-x^{k}\| = 0$.
Moreover, by applying Cauchy-Schwartz inequality to \eqref{icon} with $x=x^*$, we derive
\begin{equation}\label{jy-10}
0\leq F(x^{k+1})-F_* \leq \frac{1}{\alpha_{k}}\| x^{k+1}-x^{k}\|\|x^{*}-x^{k+1} \| +
\|\nabla f(x^{k})-\nabla f(x^{k+1})\|\|x^{*}-x^{k+1}\|.
\end{equation}
Again, $\{x^{k}\}$ is bounded and $\alpha_k\geq c >0$ from Proposition \ref{lemma:sum-alpha2} (ii).
Then, by taking the limit ``$\overline{K}\ni k \rightarrow \infty$" on both sides of \eqref{jy-10} and noting $\lim\nolimits_{\overline{K}\ni k \rightarrow \infty}\|x^{k+1}-x^{k}\| = 0$ and $\lim_{k\rightarrow\infty}\|\nabla f(x^{k+1})-\nabla f(x^{k})\|=0$, we obtain $\lim_{\overline{K}\ni k \rightarrow \infty}F(x^{k+1})=F(x^*)$. In summary, we have shown that $\lim_{k \rightarrow \infty}F(x^{k+1})=F(x^*)$ in case (c),
which confirms that all limit points of $\{x^k\}$ belong to $\mathcal{X}^*$.
It follows from \eqref{jy-11} that $U_{k+1} \leq U_k$, where $U_k$ is defined in \eqref{jy-04}.
Using Lemma \ref{opial} with $a_k:= 2B_{k}\alpha_{k}^2\|\nabla f(x^{k-1})+\xi^{k-1}\|^2+ 2E_{k}\alpha_{k}^2\big(F(x^{k-1})-F_*\big)$ and $\mathcal{X}=\mathcal{X}^*$, we derive the convergence of whole sequence $\{x^k\}$ to an element in ${\cal X}^*$. This completes the proof for case (c).
The proofs for cases (a) and (b) are much simpler and thus are omitted.
In summary, we have shown that the sequence $\{x^{k}\}$ generated by Algorithm \ref{agbbc} converges to an optimal solution of \eqref{pcom}.
\end{proof}

\section{Numerical Experiments} \label{ne}
In this section, we apply our AdaBB algorithms to two representative problems: logistic regression, where $f$ is convex and $L$-smooth, and cubic regularization, where $f$ is convex and locally smooth.
We will first compare the four algorithms implied by Algorithm \ref{agbb} by choosing different options in (Case ii) and (Case iii). These are given in Table \ref{tab:4AdaBB}. 
\begin{table}[!htbp]
\centering
\begin{tabular}{c|c|c|c|c|}
\cline{2-5}
                               & AdaBB     & AdaBB1   & AdaBB2    & AdaBB3    \\ \hline
\multicolumn{1}{|c|}{Case ii}  & Option II & Option I & Option I  & Option II \\ \hline
\multicolumn{1}{|c|}{Case iii} & Option II & Option I & Option II & Option I  \\ \hline
\end{tabular}\caption{Four AdaBB Variants.}\label{tab:4AdaBB}
\end{table}

Moreover, we will also compare the four algorithms in Table \ref{tab:4AdaBB} with the following algorithms: GD \eqref{gd} with $\alpha = 1/L$, AdGD \cite[Algorithm 1]{MM23}, and AdaPGM \cite{LTSP23}. This comparison will help demonstrate the efficiency of our proposed method. 
For initial points, we set $x^0=0$ for all situations. 
For adaptive methods: AdGD, AdaPGM and the four AdaBB variants in Table \ref{tab:4AdaBB}, we choose $\alpha_0=10^{-10}$ as recommended in \cite{MM20}. This ensures that $x^1$ will be close to $x^0$, and provides a reliable estimate of $\alpha_1$. Since $\lambda_1$ is very likely to be greater than $\alpha_0$, it is more likely that $1 \in I_1$. Hence, for the four AdaBB variants in Table \ref{tab:4AdaBB}, we set $\theta_0$ as defined in \eqref{int:theta}. In the numerical experiments, we also set $\theta_1=1\ll \frac{\alpha_1}{\alpha_0}$ to prevent excessive $\alpha_2$ values due to the small value of $\alpha_0=10^{-10}$.

Our codes were written in Python 3.11.0 and used the framework provided by Malitsky and Mishchenko \cite{MM20}. All numerical experiments were conducted on a personal computer with an AMD Ryzen 7 5800H processor, Radeon Graphics, and 16GB memory. Additionally, the experiments utilized the mushrooms, w8a, and covtype datasets from LIBSVM \cite{CL11}.

\subsection{Logistic Regression} \label{sec:lr}
In this subsection, we consider the logistic regression problem
\begin{equation}\label{log-regression}
\min_{x\in {\mathbb R}^{n}} f(x)=-\frac{1}{m}\sum_{i=1}^m \left(y_i \log (s(a_i^\top x)) + (1 - y_i) \log (1 - s(a_i^\top x))\right) + \frac{\gamma}{2}\|x\|^2,
\end{equation}
where $a_i\in {\mathbb R}^{n}$, $y_i\in \{0, 1\}$. Here,  $s(z)=1/({1+\exp(-z)})$ denotes the sigmoid function, $m$ represents the number of observations, and $\gamma$ serves as a regularization parameter. 
For this problem, the gradient of $f$ is given by $\nabla f(x) = \frac{1}{m}\sum_{i=1}^m a_i(s(a_i^\top x)-y_i) + \gamma x$. This means that $f$ is a $L$-smooth function with $L=\frac{1}{4}\lambda_{\max}(A^\top A) + \gamma$, where $A = (a_1^{\top},\ldots,a_m^{\top})^{\top}$ and $\lambda_{\max}(A^\top A)$ denotes the largest eigenvalue of matrix $A^\top A$ \cite{MM20}. 
In this experiment, we run all algorithms for a fixed number of iterations, denoted by MaxIter in Table \ref{tab:settings}. We use  $f_*$ to denote the lowest objective function value obtained among all tested algorithms.

Details of the data sets and the parameters are given in Table \ref{tab:settings}.

\begin{table}[!htbp]
\centering
\begin{tabular}{l|c|c|c|c|c|}
\cline{2-6}
                                        & $m$      & $n$   &   $L $ & $\gamma$    &                MaxIter  \\ \hline
\multicolumn{1}{|l|}{mushrooms} & $8124$   & $112$ & $2.59$ & $\frac{L}{m}=3.18\times 10^{-5}$ & $1000$      \\ \hline
\multicolumn{1}{|l|}{w8a}       & $49749$  & $300$ & $0.66$ & $\frac{L}{m}=1.32\times 10^{-6}$ & $3000$       \\ \hline
\multicolumn{1}{|l|}{covtype}   & $581012$ & $54$ & $5.04\times 10^6$ & $\frac{L}{10m}=8.68$         & $10000$        \\ \hline
\end{tabular}
\caption{Parameters settings for different datasets.}
\label{tab:settings}
\end{table}

\begin{figure}[!htbp]
\centering
\subfloat[mushrooms dataset, objective]{
\includegraphics[scale=0.25]{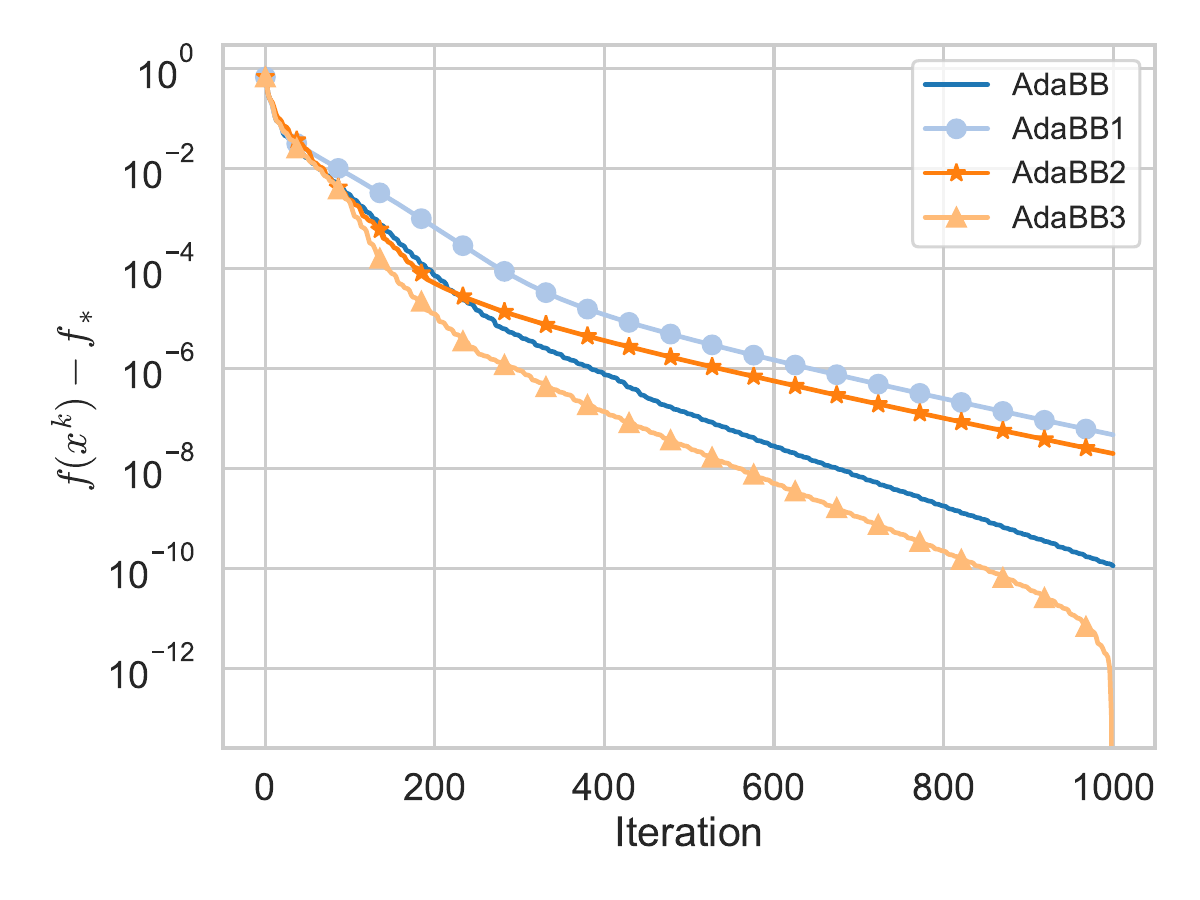}}
\subfloat[w8a dataset, objective]{
\includegraphics[scale = 0.25]{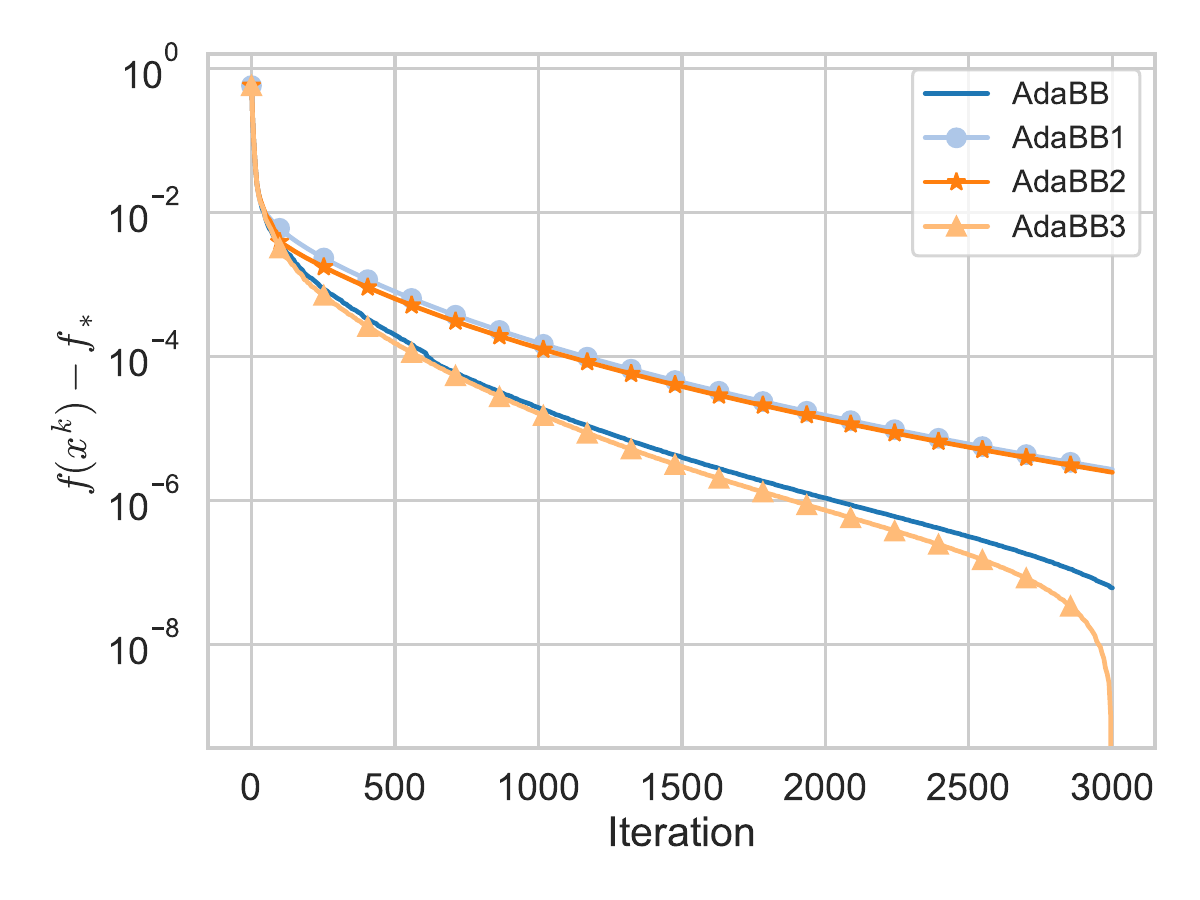}}
\subfloat[covtype dataset, objective]
{\includegraphics[scale = 0.25]{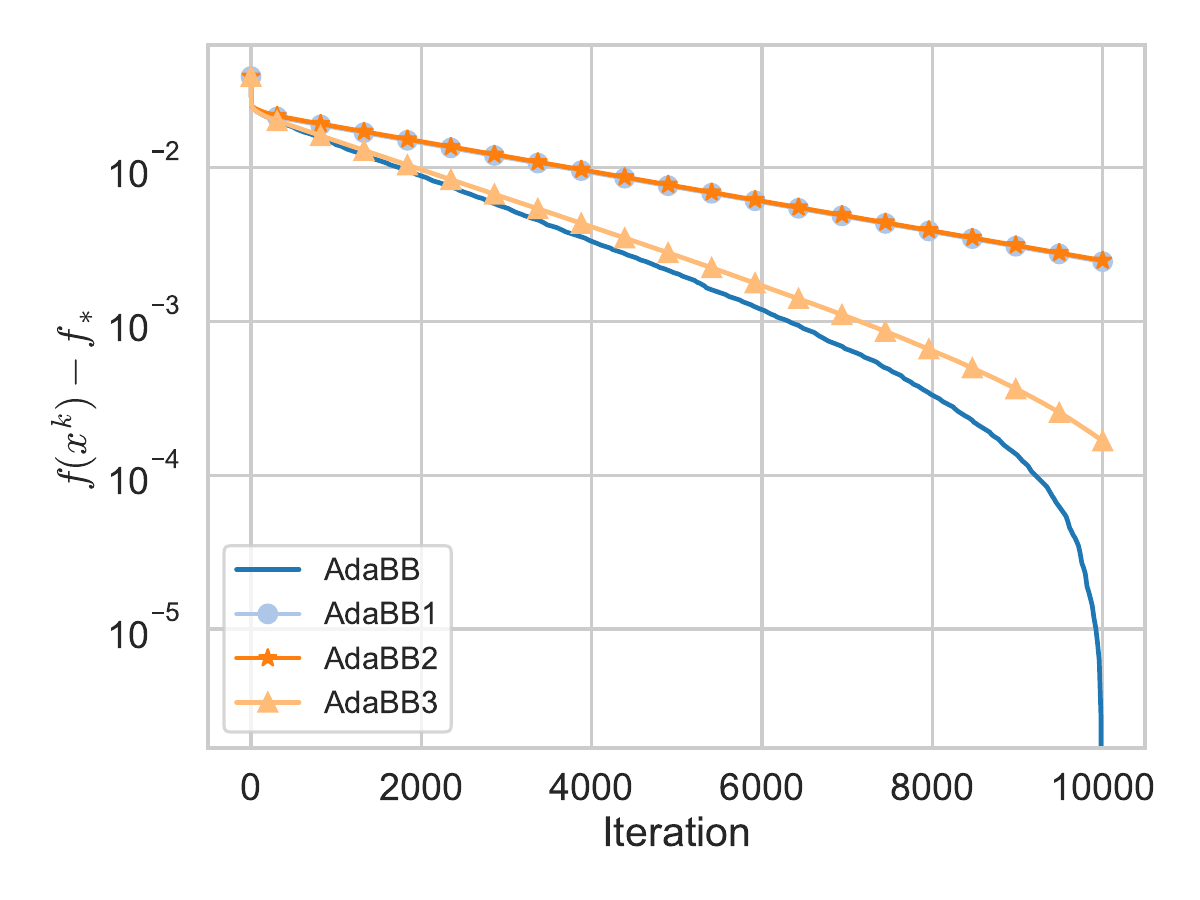}}
\caption{Results for the logistic regression problem via AdaBB, AdaBB1, AdaBB2, AdaBB3 concerning the function value residual.}
\label{fig:self-1}
\end{figure}

We first present the numerical performances for the four AdaBB variants given in Table \ref{tab:4AdaBB}.
Figure \ref{fig:self-1} shows that AdaBB and AdaBB3 are more efficient than AdaBB1 and AdaBB2. This indicates that choosing Option II in (Case ii) of Algorithm \ref{agbb} is more preferable. This further implies that when the BB stepsize $\lambda_k$ is not too large and not too small, then it gives superior performance by choosing $\lambda_k$ as the stepsize.
In the rest of this subsection, we only compare AdaBB and AdaBB3 with other popular optimization algorithms. 

In Figure \ref{fig:rg}, we compare AdaBB and AdaBB3 with GD, AdGD and AdaPGM. In subfigures (a), (b), and (c), we show the function value error, and in subfigures (d), (e) and (f), we show the norm of the gradient. From these figures we see that AdaBB and AdaBB3 both perform very well and are usually better than the other three algorithms -- AdGD is comparable sometimes. 

\begin{figure}[!htbp]
\centering
\subfloat[mushrooms, objective]
{\includegraphics[scale = 0.25]{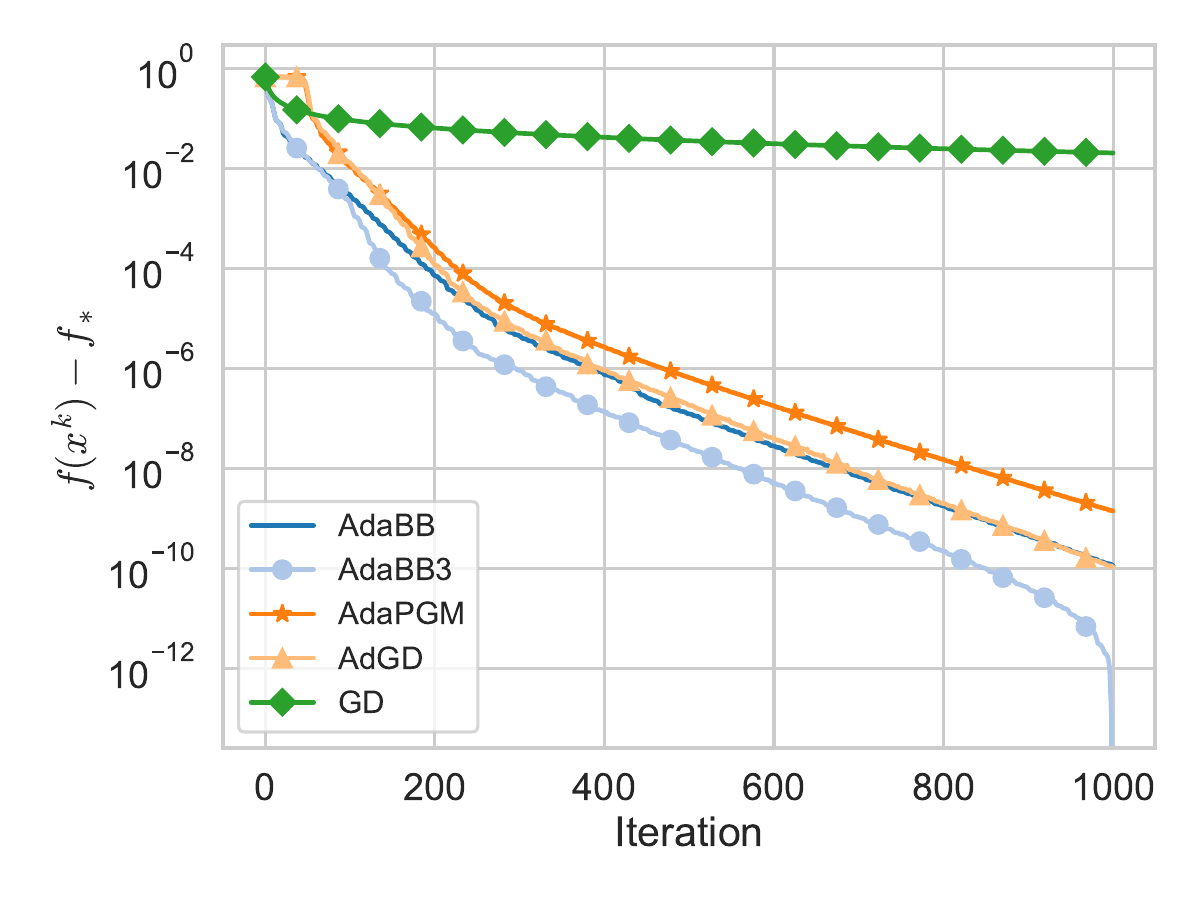}}
\subfloat[w8a, stepsize]{
\includegraphics[scale = 0.25]{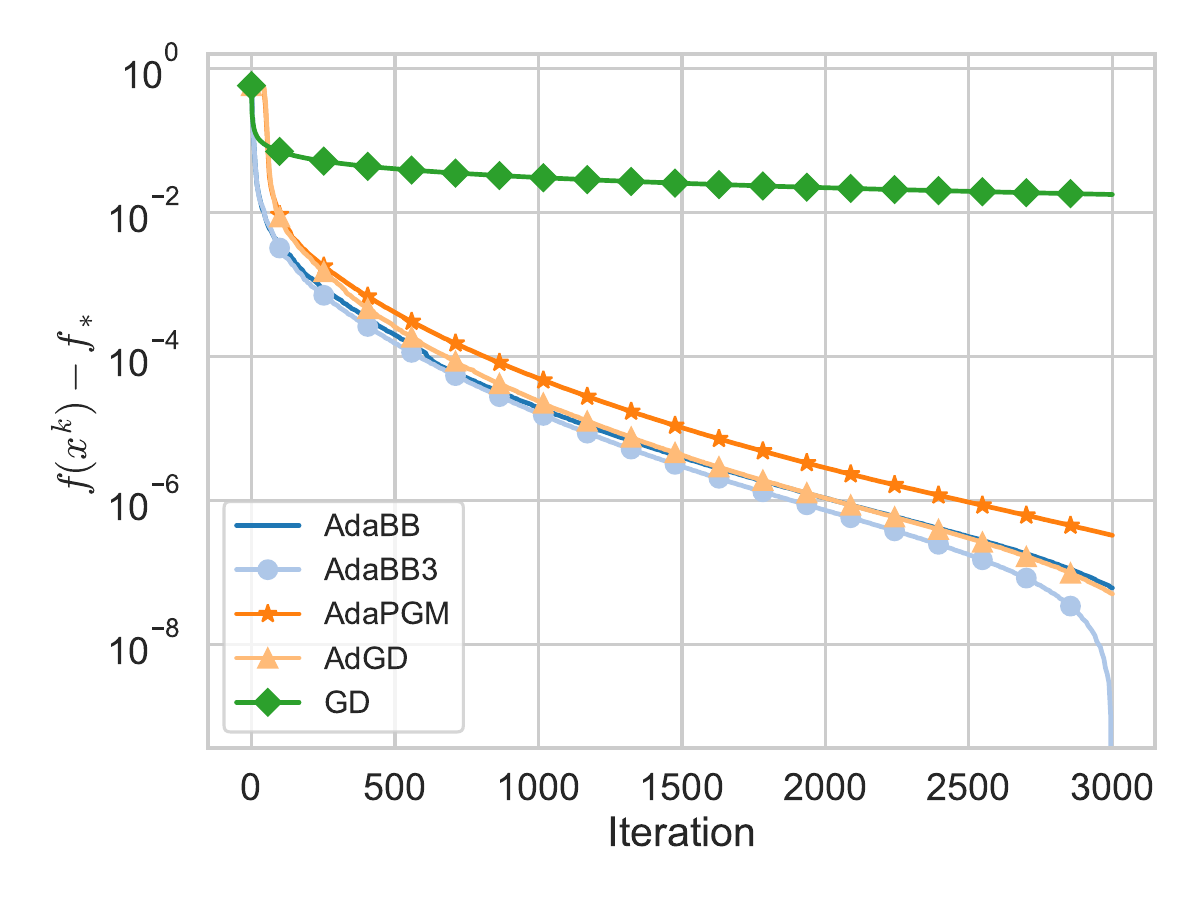}}
\subfloat[covtype, objective]{
\includegraphics[scale = 0.25]{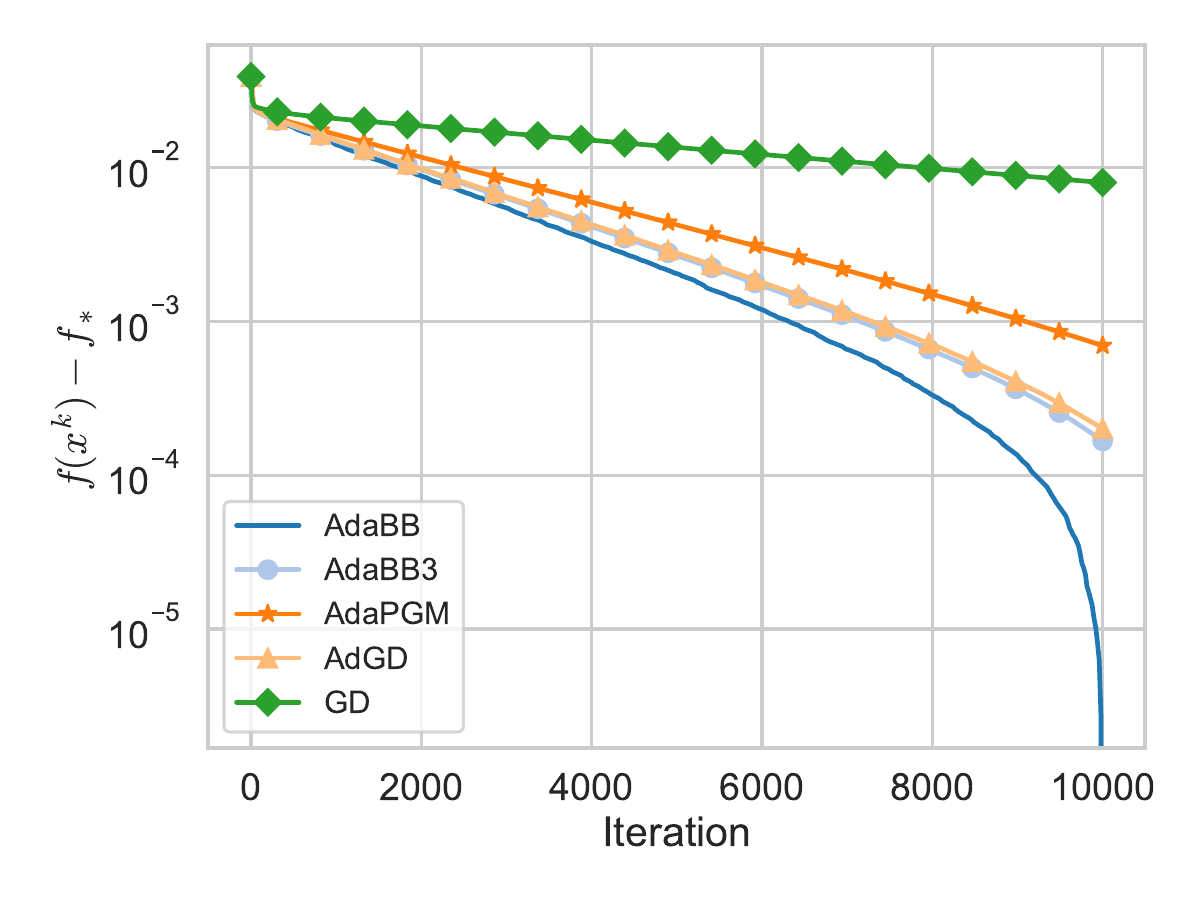}}
\\
\subfloat[mushrooms, gradient norm]{
\includegraphics[scale=0.25]{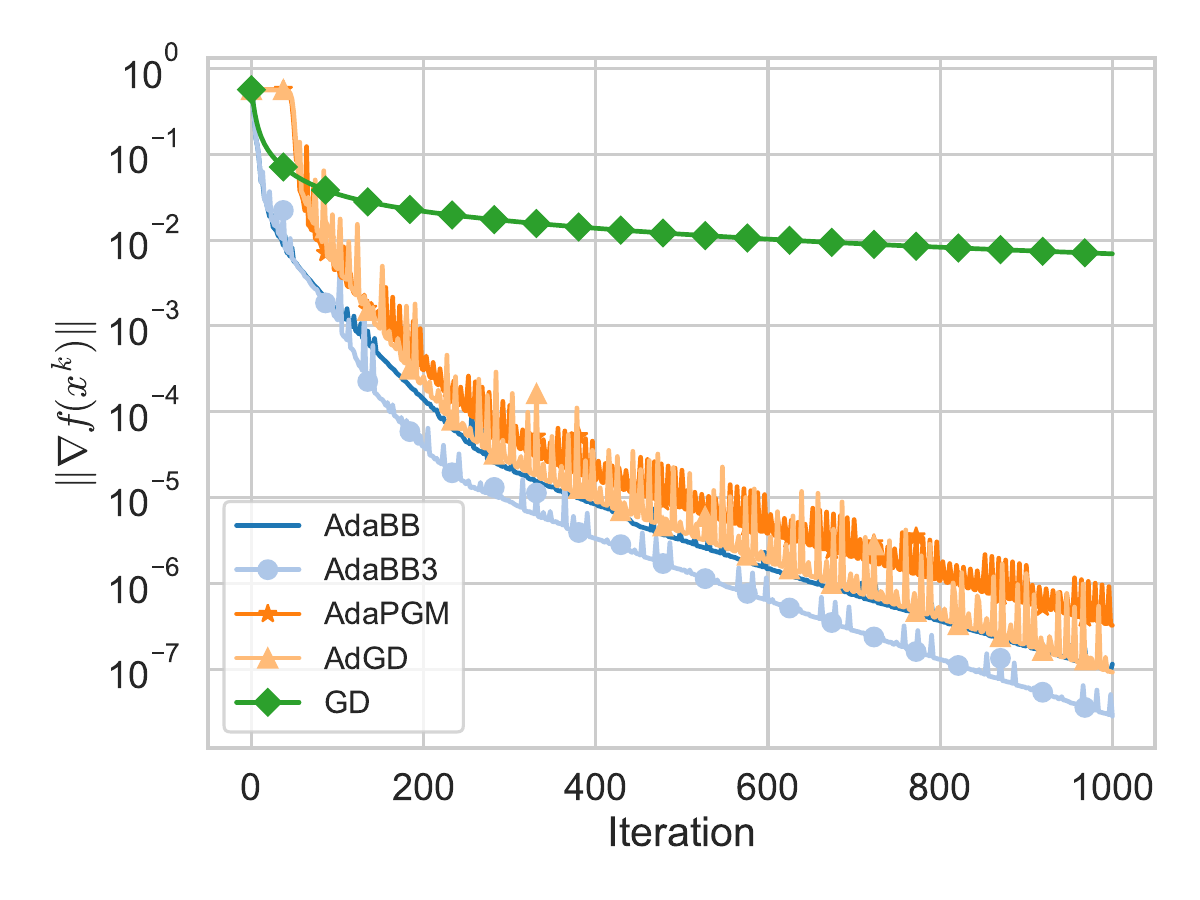}}
\subfloat[w8a, gradient norm]{
\includegraphics[scale=0.25]{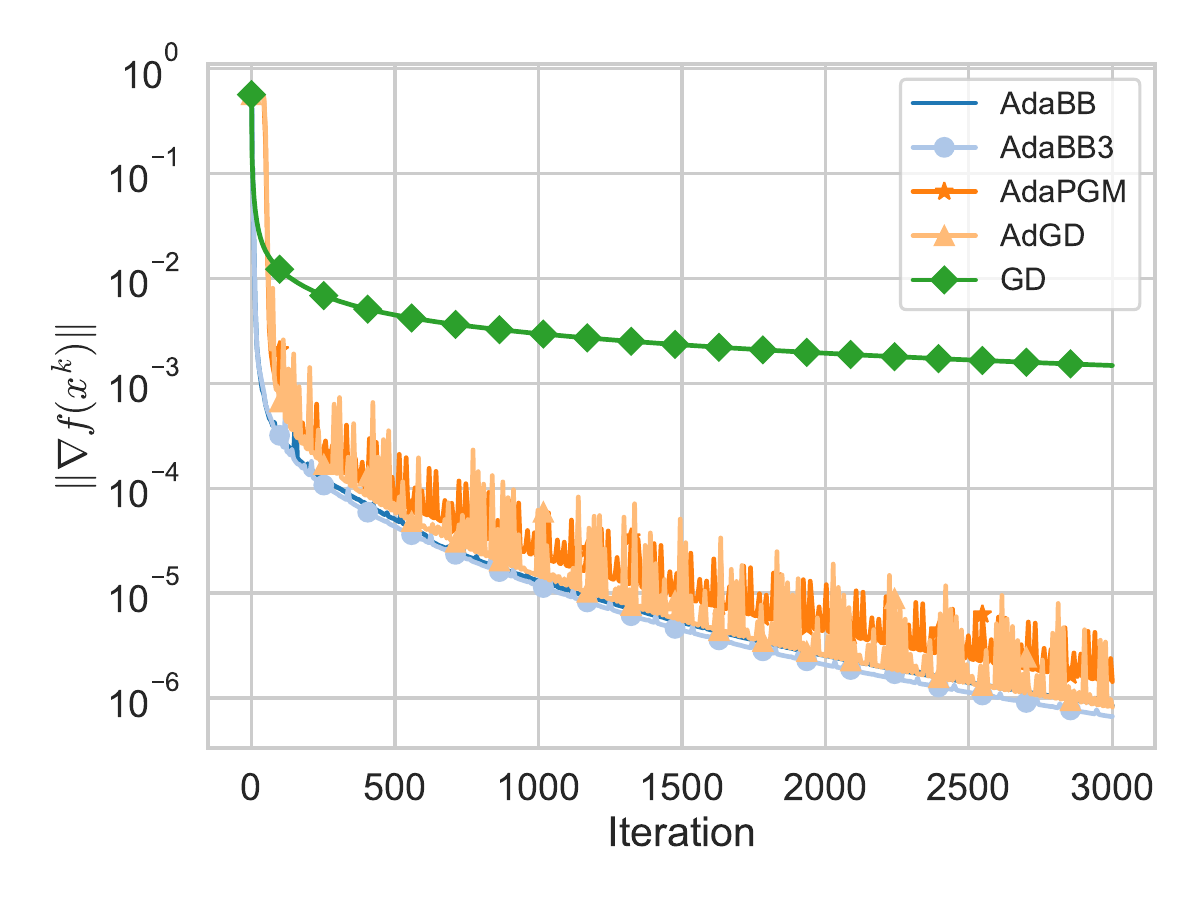}}
\subfloat[covtype, gradient norm]{
\includegraphics[scale=0.25]{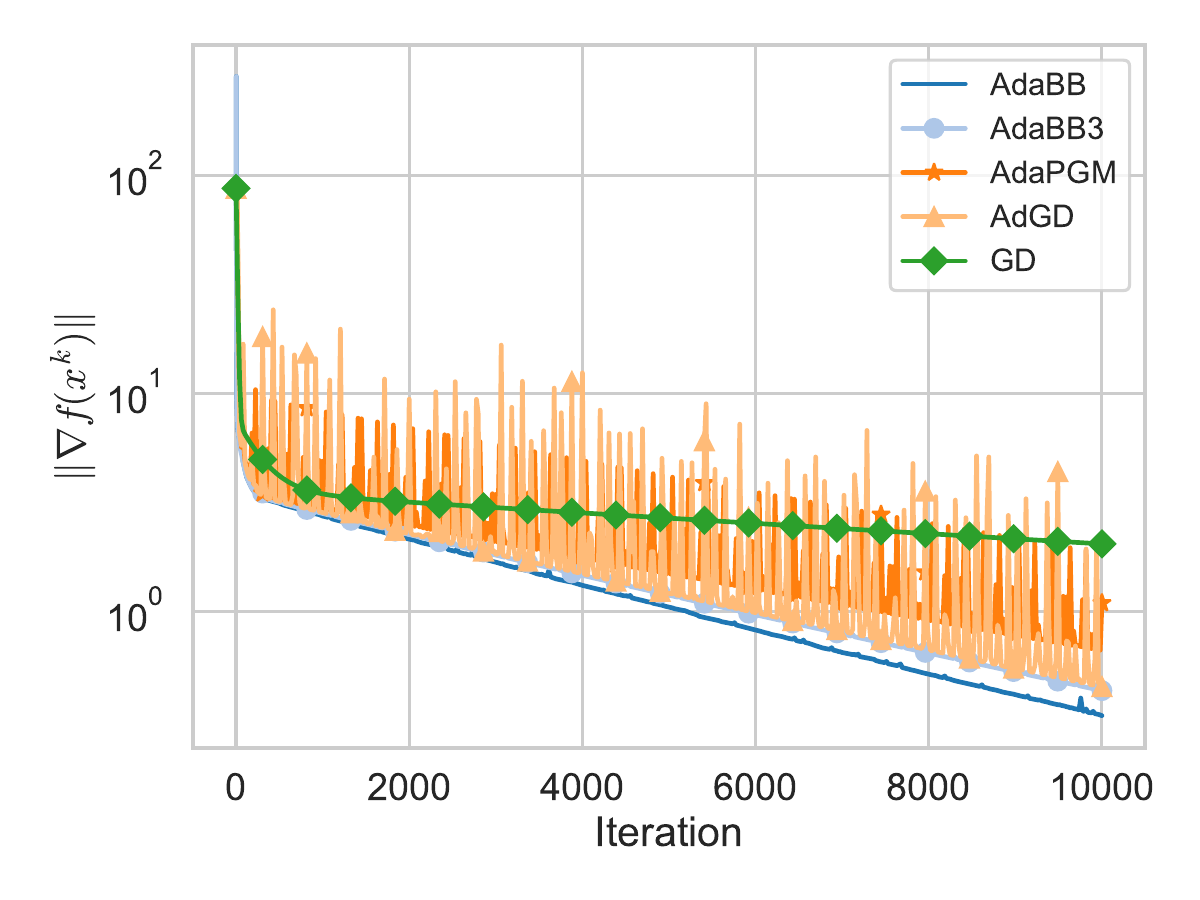}}
\caption{Results for the logistic regression problem via GD, AdGD, AdaPGM, and AdaBB concerning the function value residual and gradient norm.}
\label{fig:rg}
\end{figure}

We also compare AdaBB and AdaBB3 with line-search methods that do not require prior knowledge of $L$, including line search for GD (with Armijo) \cite{A66}, and BB stepsize with line search \cite{R97}. The results are shown in Figure \ref{fig:ls}. This time the $x$-axis denotes the number of matrix-vector multiplications.
The results indicate that both AdaBB and AdaBB3 usually perform better than the two line search methods. 

\begin{figure}[!htbp]
\centering
\subfloat[mushrooms, objective]{
\includegraphics[scale=0.25]{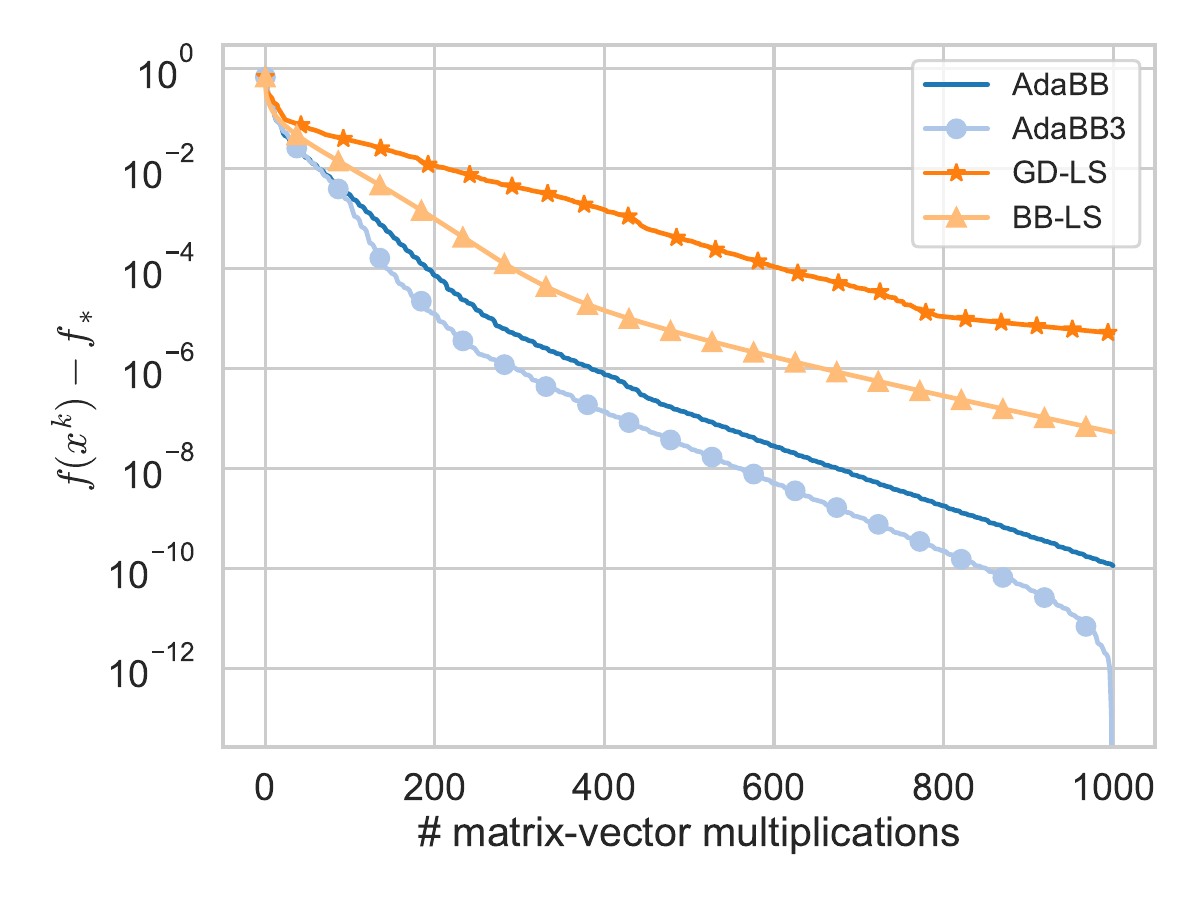}}
\subfloat[w8a, objective]{
\includegraphics[scale = 0.25]{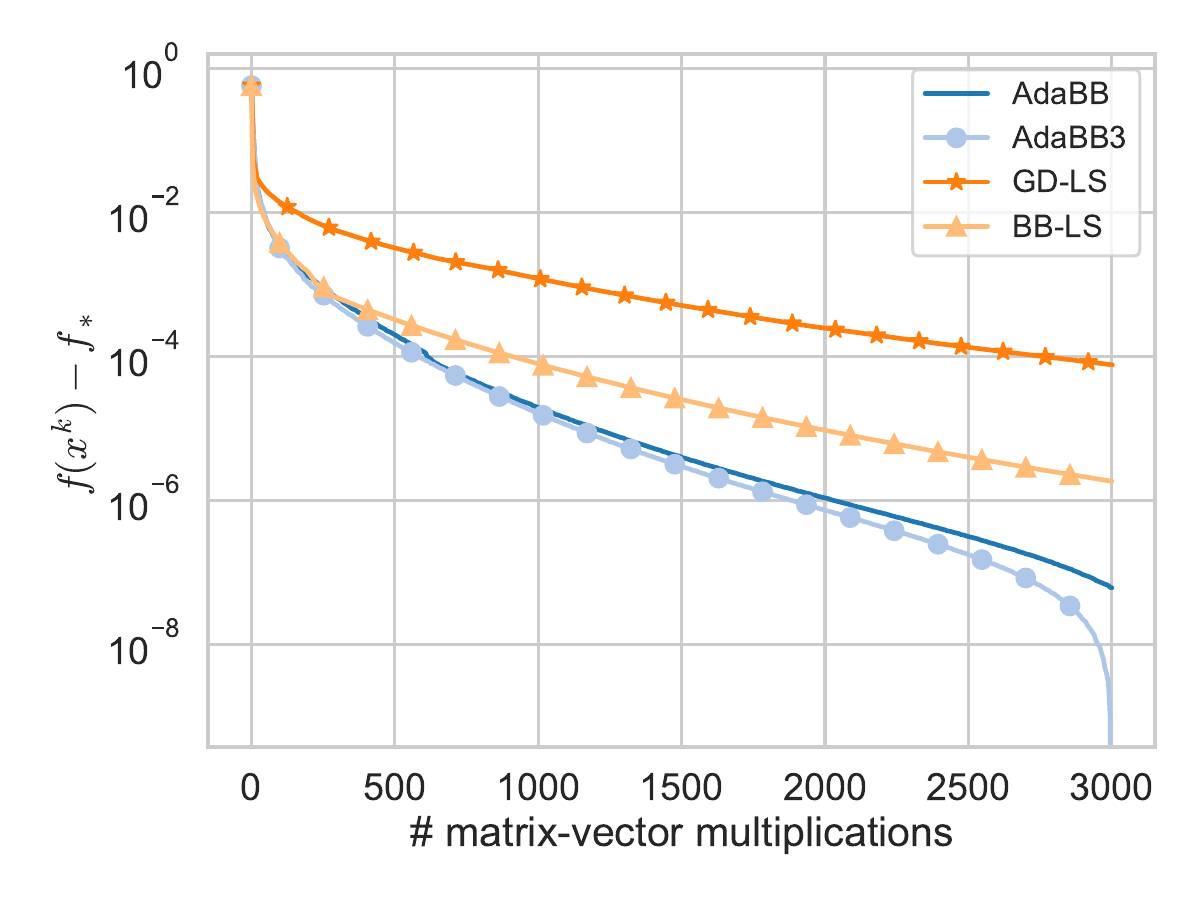}}
\subfloat[covtype, objective]
{\includegraphics[scale = 0.25]{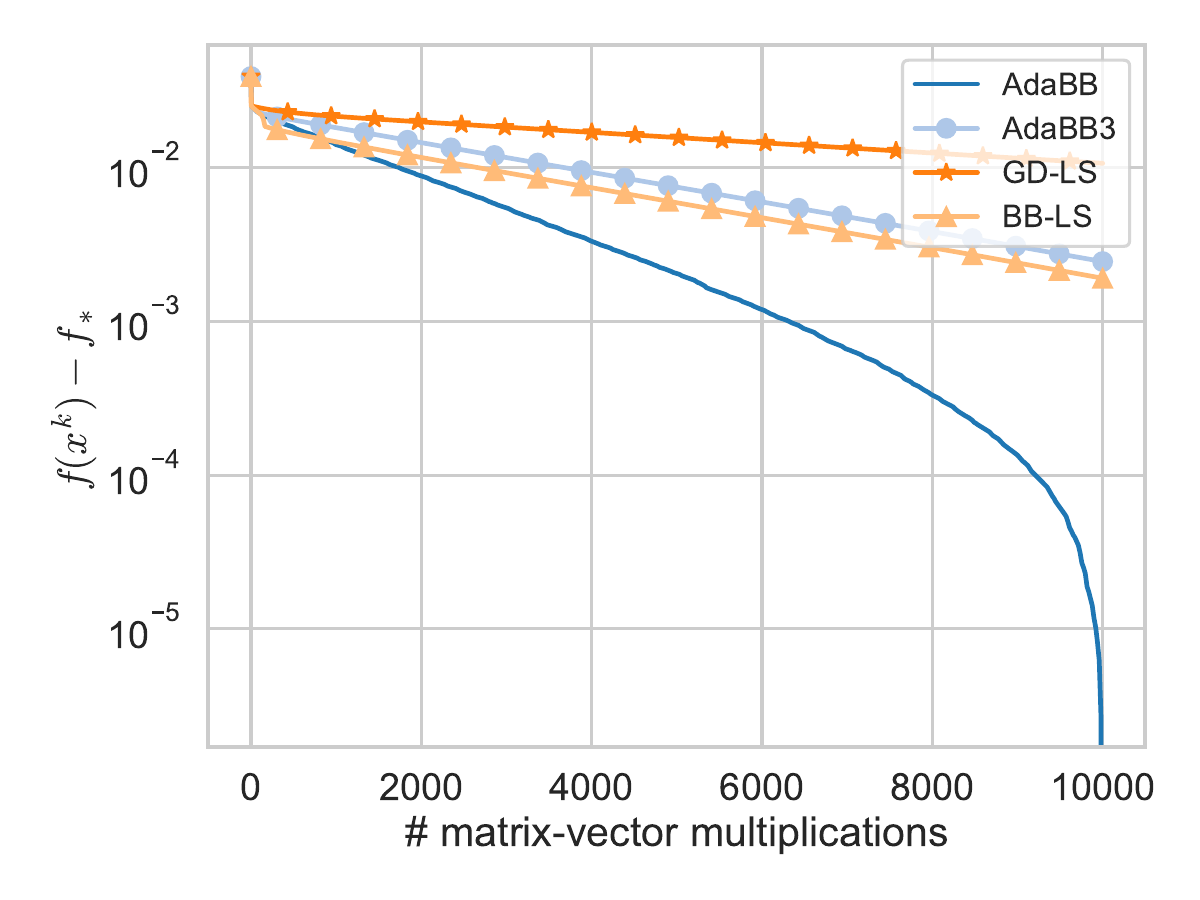}}
\caption{Results for the logistic regression problem via line search methods, and AdaBB concerning the function value residual.}
\label{fig:ls}
\end{figure}

At the end of this subsection, we show the stepsizes generated in the first $100$ iterations of AdGD and AdaBB. The results are shown in Figure \ref{fig:step}. From Figure \ref{fig:step} (a), (b) and (c) we see that the stepsizes produced by both AdGD and AdaBB have a fractal-like nature, and AdaBB usually produces larger stepsizes comparing with AdGD. Figure \ref{fig:step} (d), (e) and (f) illustrate the pattern of the stepsizes generated by AdaBB. We see with excessively large $\alpha_k$, it is more likely that the next stepsize will be very small, i.e., $(k+1) \in I_3$. Conversely, when $\alpha_k$ is too small, AdaBB automatically opts for $(k+1) \in I_1$ to enlarge the stepsize, and opting for the BB stepsize $\lambda_k$ is rational when the stepsize is moderate.

\begin{figure}[!htbp]
\centering
\subfloat[mushrooms, stepsize]{
\includegraphics[scale=0.25]{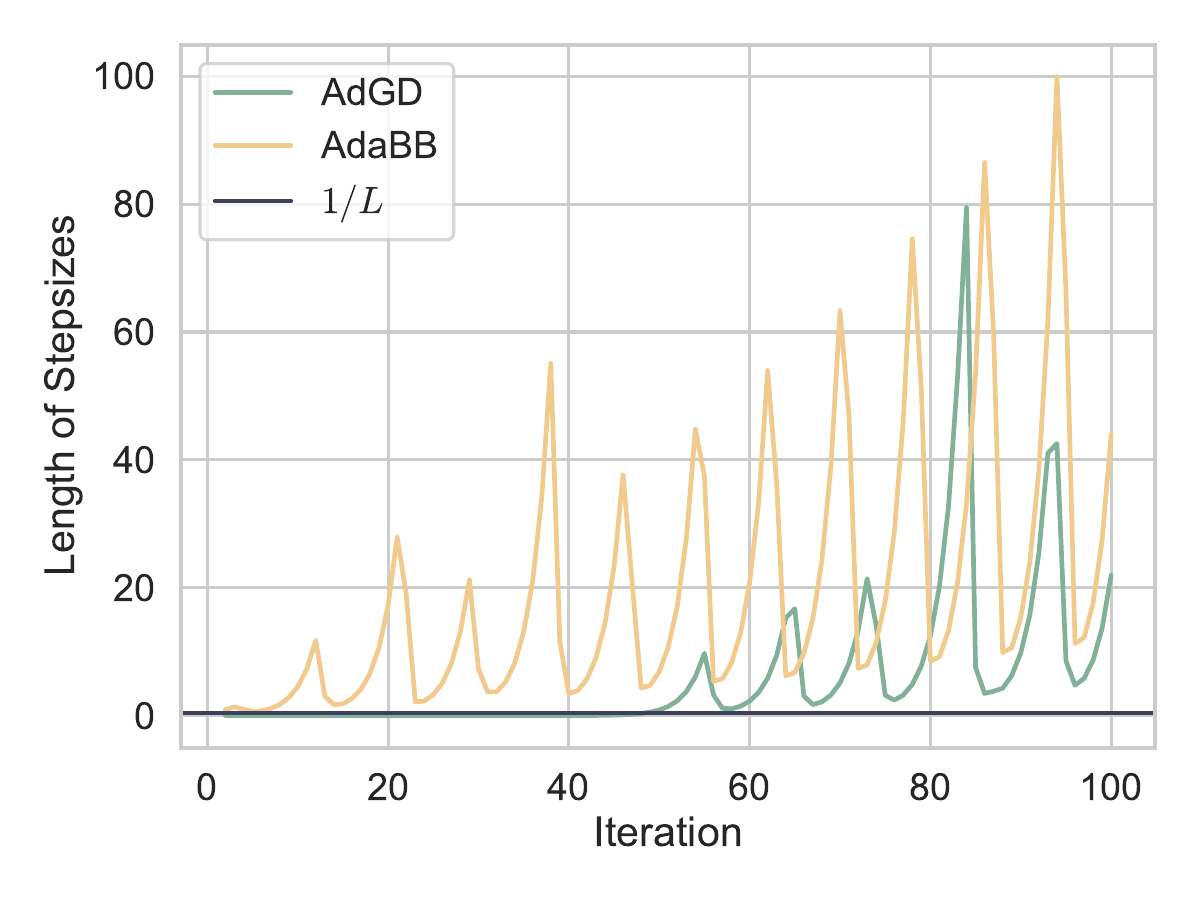}}
\subfloat[w8a, stepsize]{
\includegraphics[scale=0.25]{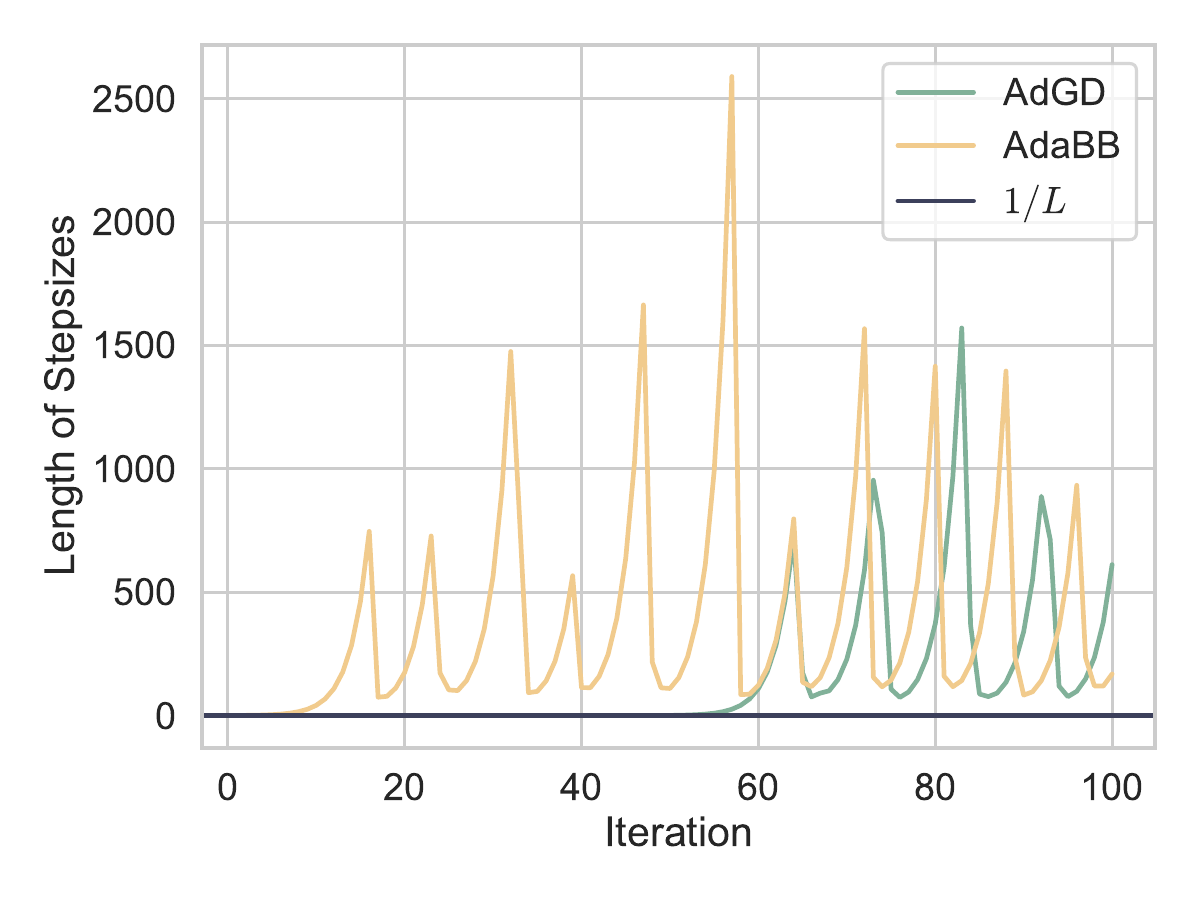}}
\subfloat[covtype, stepsize]{
\includegraphics[scale=0.25]{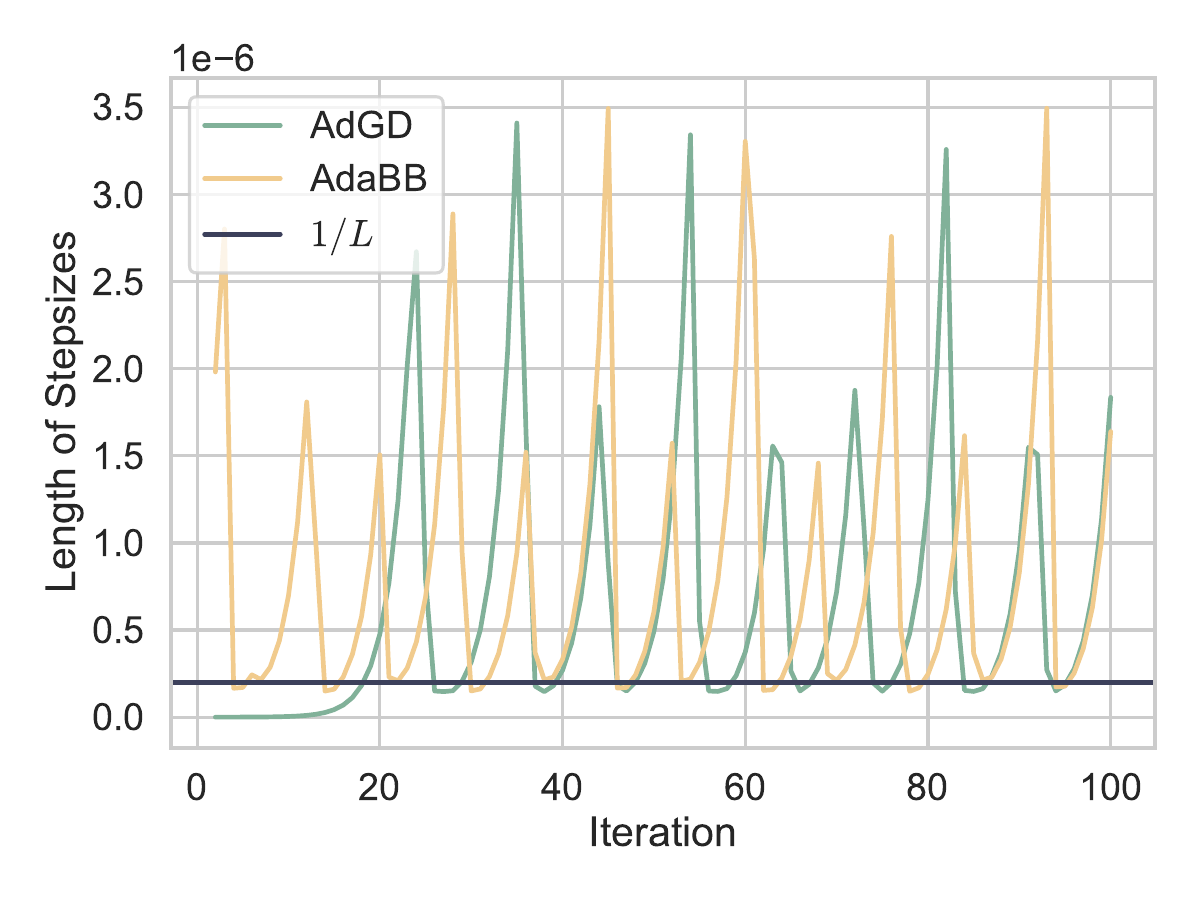}}\\
\subfloat[mushrooms, pattern]{
\includegraphics[scale=0.25]{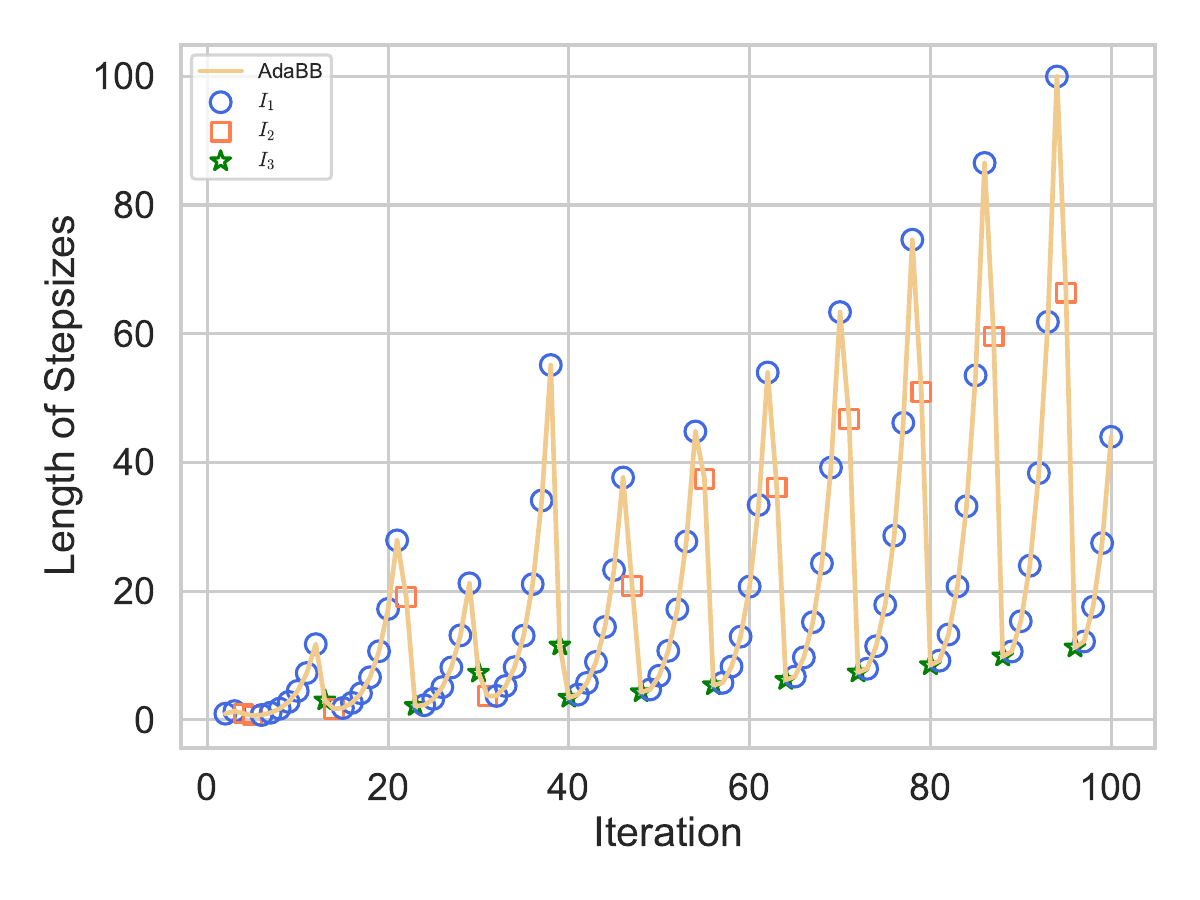}}
\subfloat[w8a, pattern]{
\includegraphics[scale=0.25]{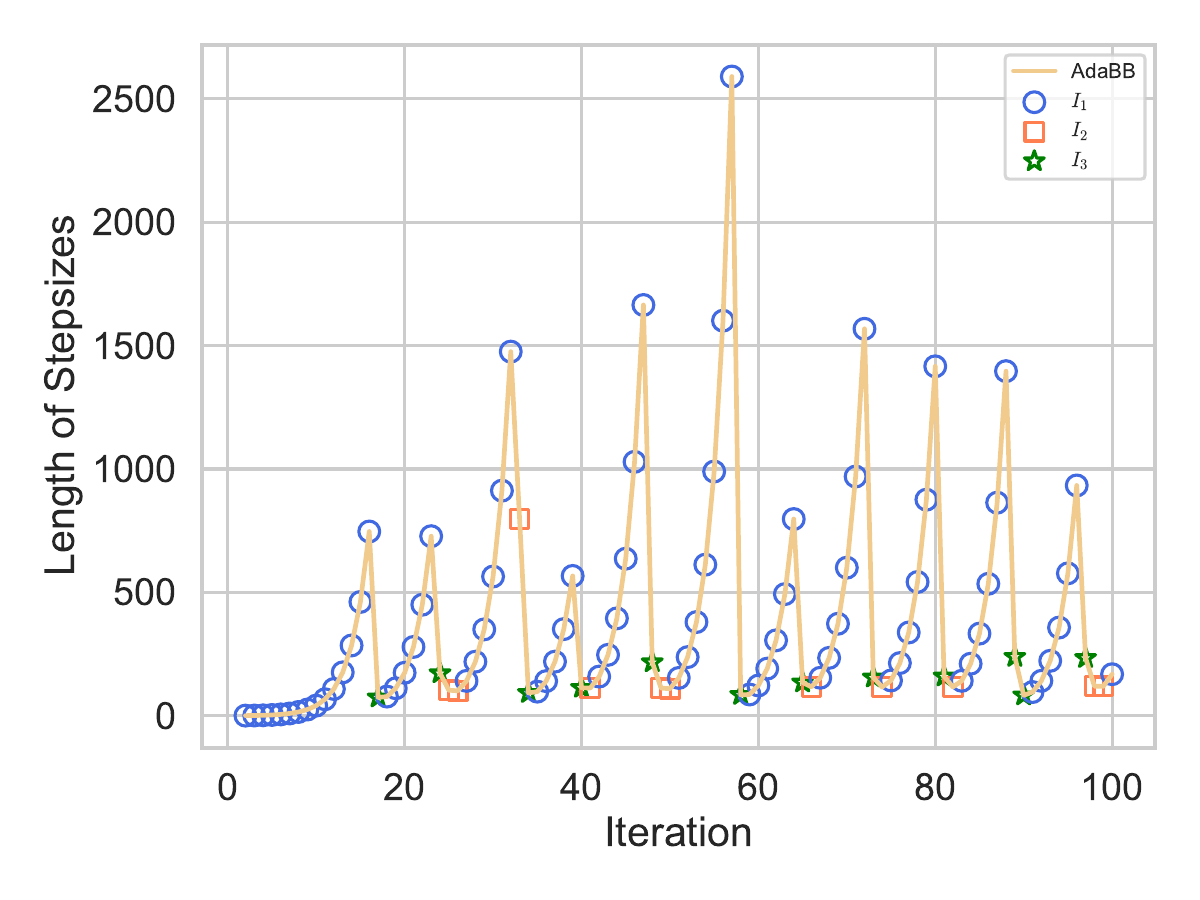}}
\subfloat[covtype, pattern]{
\includegraphics[scale=0.25]{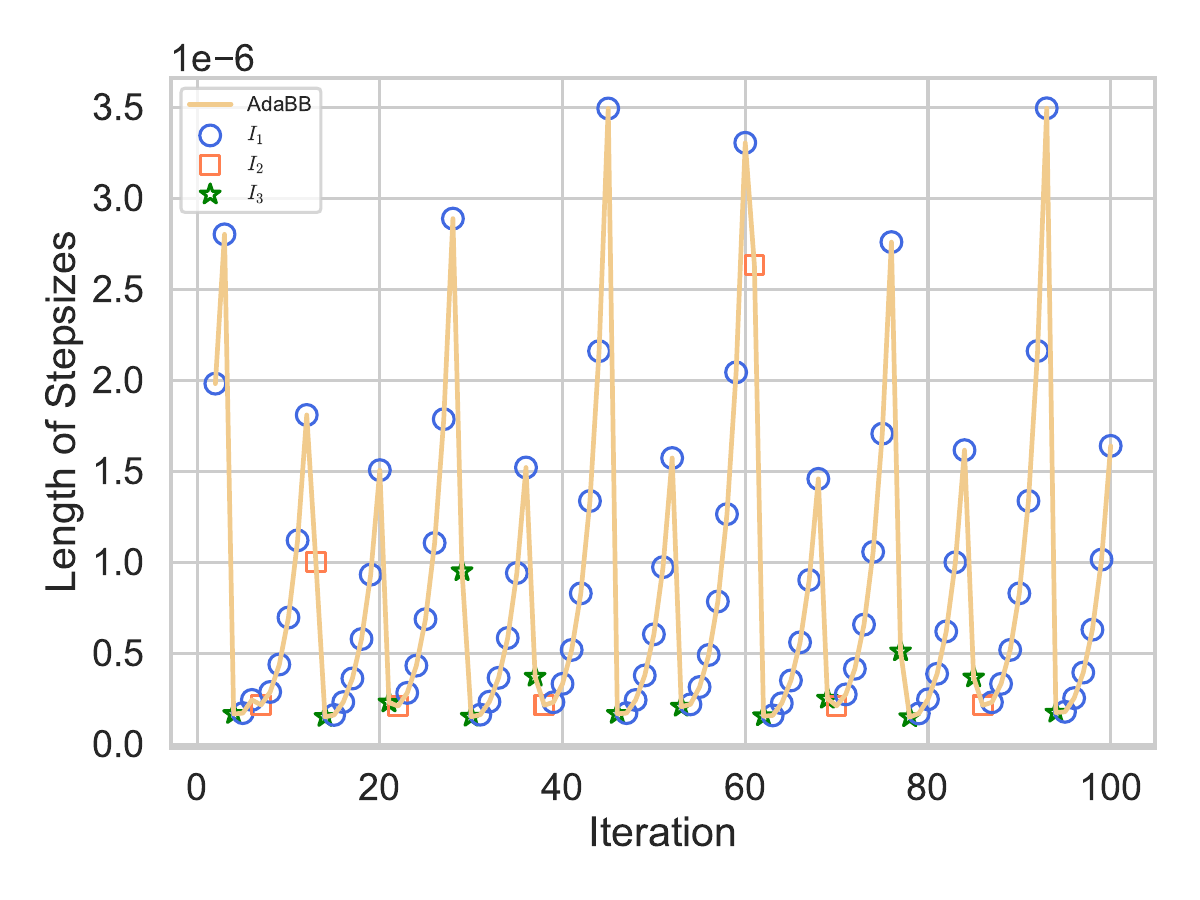}}
\caption{Stepsizes generated by AdGD and AdaBB. }
\label{fig:step}
\end{figure}

\subsection{Subproblem of cubic regularized Newton method} \label{sec:cr}
The cubic regularized Newton method \cite{NP06} requires solving the following subproblem: 
\begin{equation}\label{cubic-sub}
    \min_{x\in {\mathbb R}^{n}} f(x)=g^\top x + \frac{1}{2} x^\top H x + \frac{M}{6}\|x\|^3,
\end{equation}
in each iteration, where $g\in {\mathbb R}^{n}$, $H\in {\mathbb R}^{n\times n}$, and $M>0$ is a given regularization parameter. For this problem, the gradient of $f$ is given by $\nabla f(x) = g + H x + \frac{M \|x\|}{2}x$. Note that there is no value of $L$ that can guarantee $\|\nabla f(x)-\nabla f(y)\|\leq L\|x-y\|$ for all $x$ and $y$ in $\mathbb{R}^{n}$. This implies that $f$ is only smooth locally. Therefore, it becomes challenging to determine the stepsizes for GD. To solve this issue, we adopt a trial-and-error approach \cite{MM20} to fine-tune the stepsize for these two methods. Specifically, we designate the values from an array of $10$ numbers evenly spaced on a logarithmic scale between $10^{-1}$ and $10$, as potential stepsizes. By computing $f(x^k)$ with $k=\text{MaxIter}/2$ for each of these stepsizes, we select the largest number for which $f(x^k)$ is evaluated, i.e., not NaN, to be our fine-tuned stepsize. 

In this experiment, we assume that the problem \eqref{cubic-sub} is the subproblem of the cubic regularized Newton method for solving the logistic regression problem \eqref{log-regression}. We thus generate the gradient $g$ and the Hessian $H$ for logistic regression problem evaluated at $0$ for different values of $M$. Specifically, we consider $M=\{10,15\}$ for all three datasets. 
The dimension $n$ stays the same as in Table \ref{tab:settings}. Moreover, the stopping criterion is similar to the one outlined in Subsection \ref{sec:lr}. We run all algorithms for a given number of iterations. Specifically, MaxIter is set to $30$ for mushrooms and w8a datasets, and to $1000$ for the covtype dataset, because the covtype dataset is around 90 times larger than the mushrooms dataset and 30 times larger than the w8a dataset and it requires more iterations to solve. Note that we here set a limited number of iterations. To avoid AdGD and AdaPGM requiring several initial iterations to allow for the small $\alpha_0=10^{-10}$ to grow to an appropriate step, we first calculate $x^1$ and $L_1$ defined in \eqref{AdGD} for the given $\alpha_0$ and then reset $\alpha_0$ to $1/(\sqrt{2}L_1)$.

\begin{figure}[!htbp]
\centering
\subfloat[mushrooms, $M=10$, objective]{
\includegraphics[scale = 0.25]{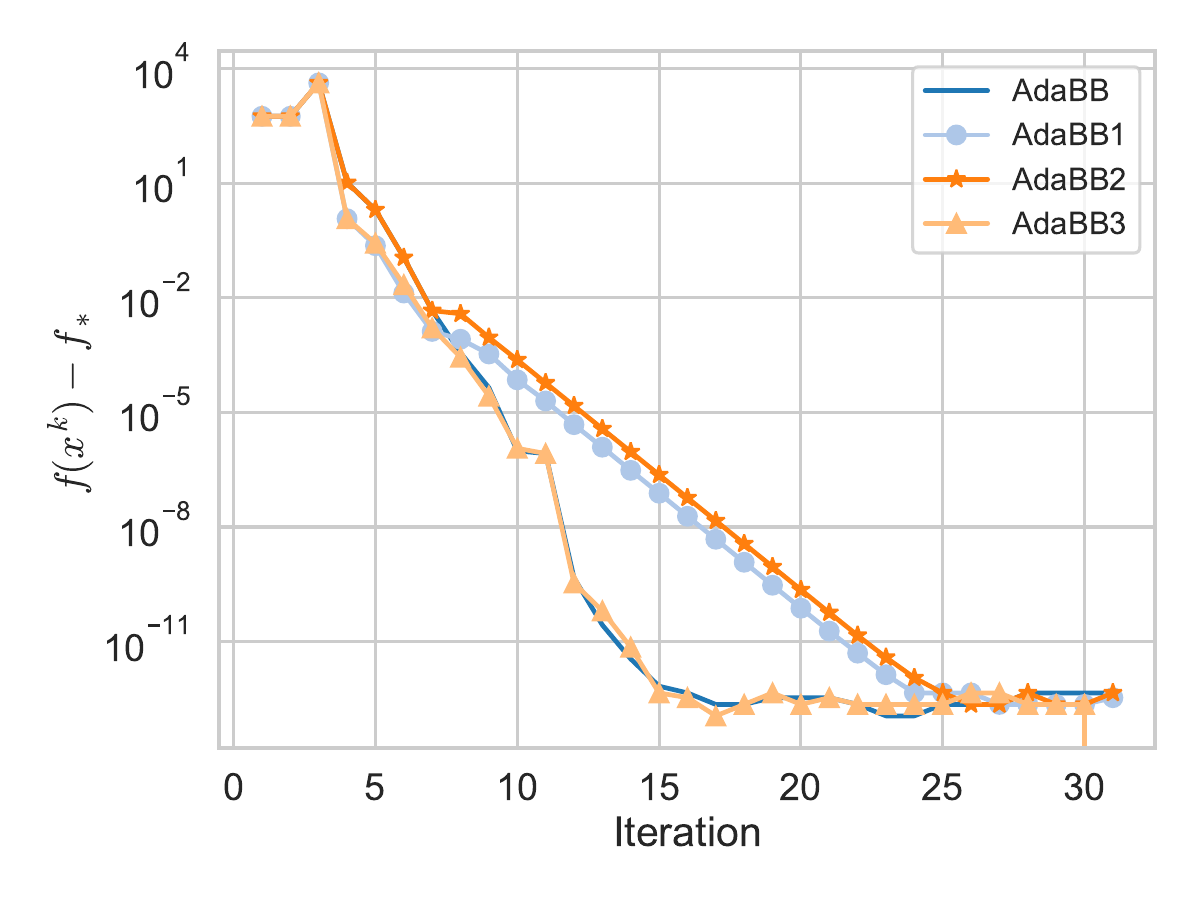}}
\subfloat[w8a, $M=10$, objective]{
\includegraphics[scale = 0.25]{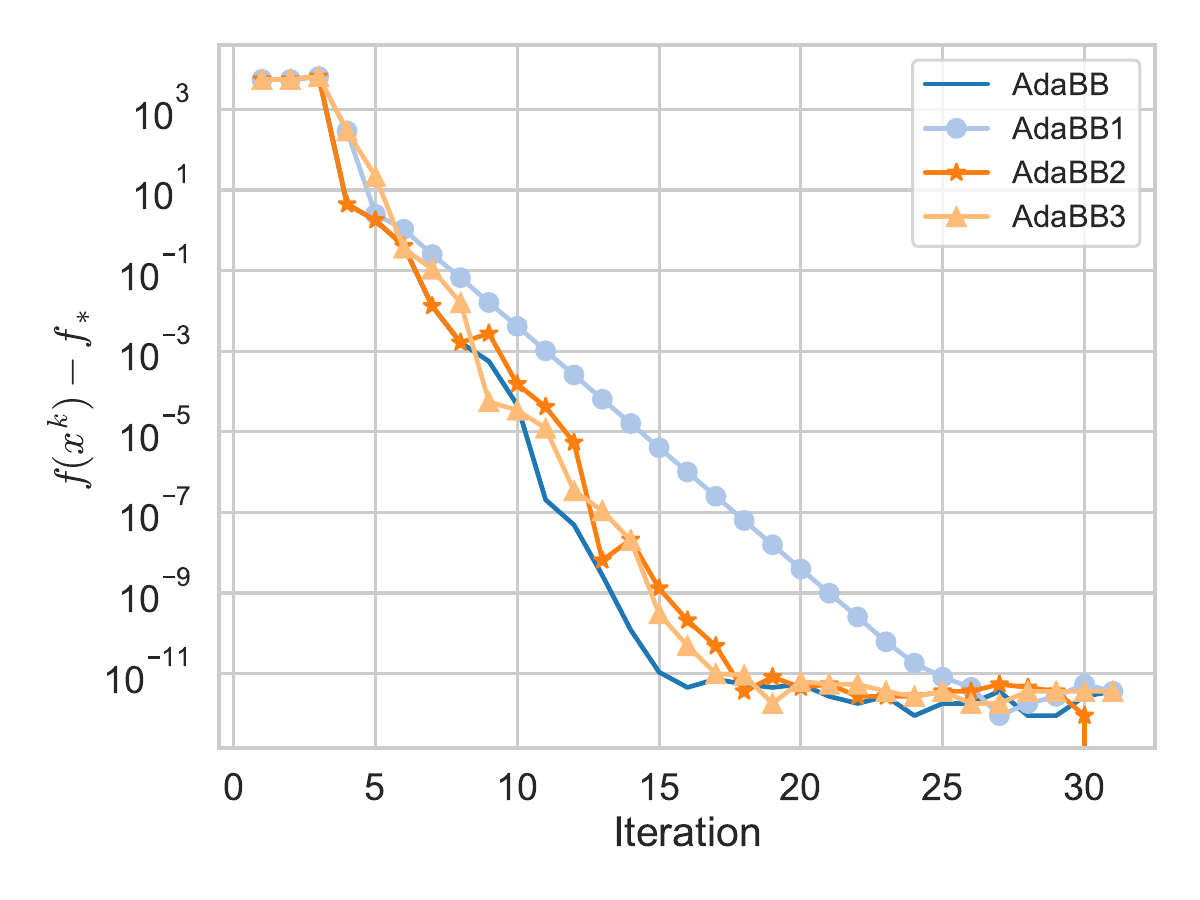}}
\subfloat[covtype, $M=10$, objective]{
\includegraphics[scale = 0.25]{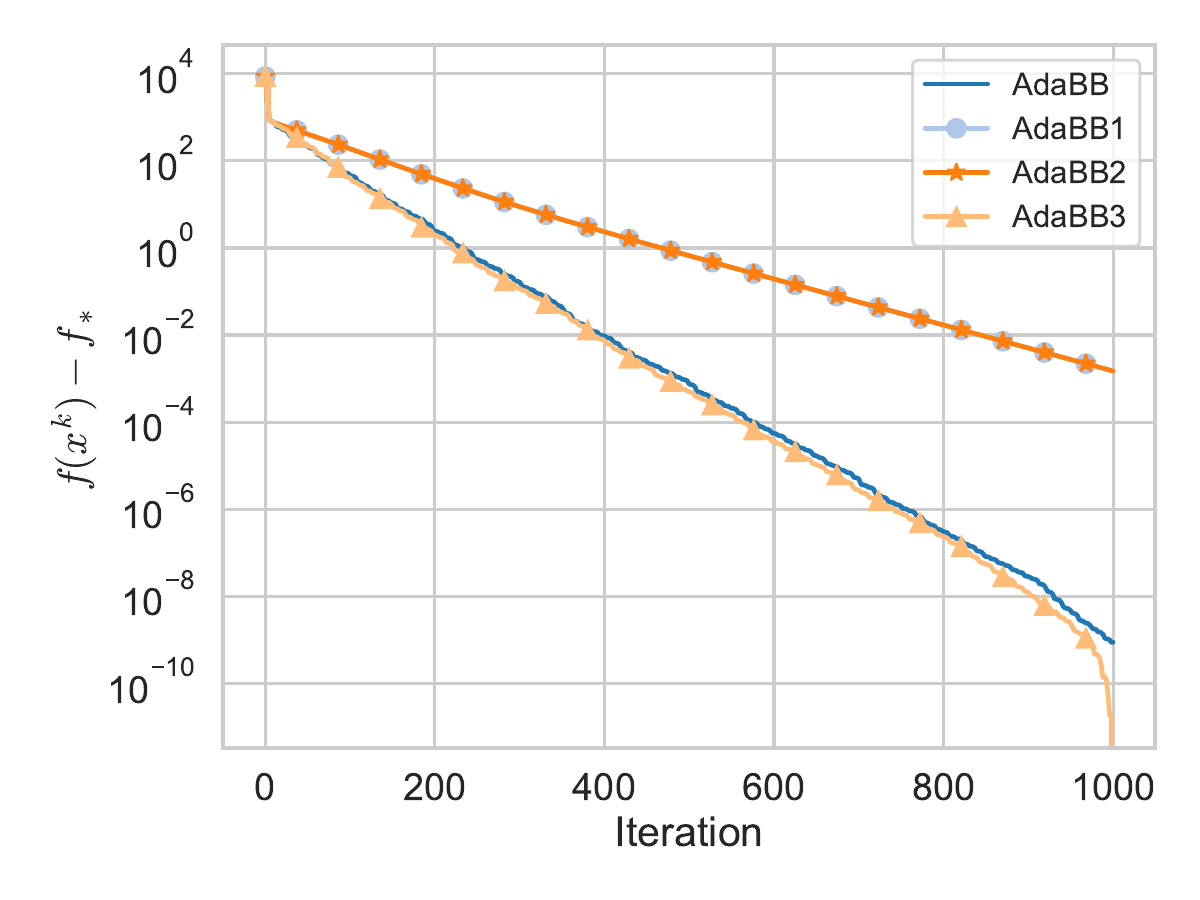}}\\
\subfloat[mushrooms, $M=15$, objective]
{\includegraphics[scale = 0.25]{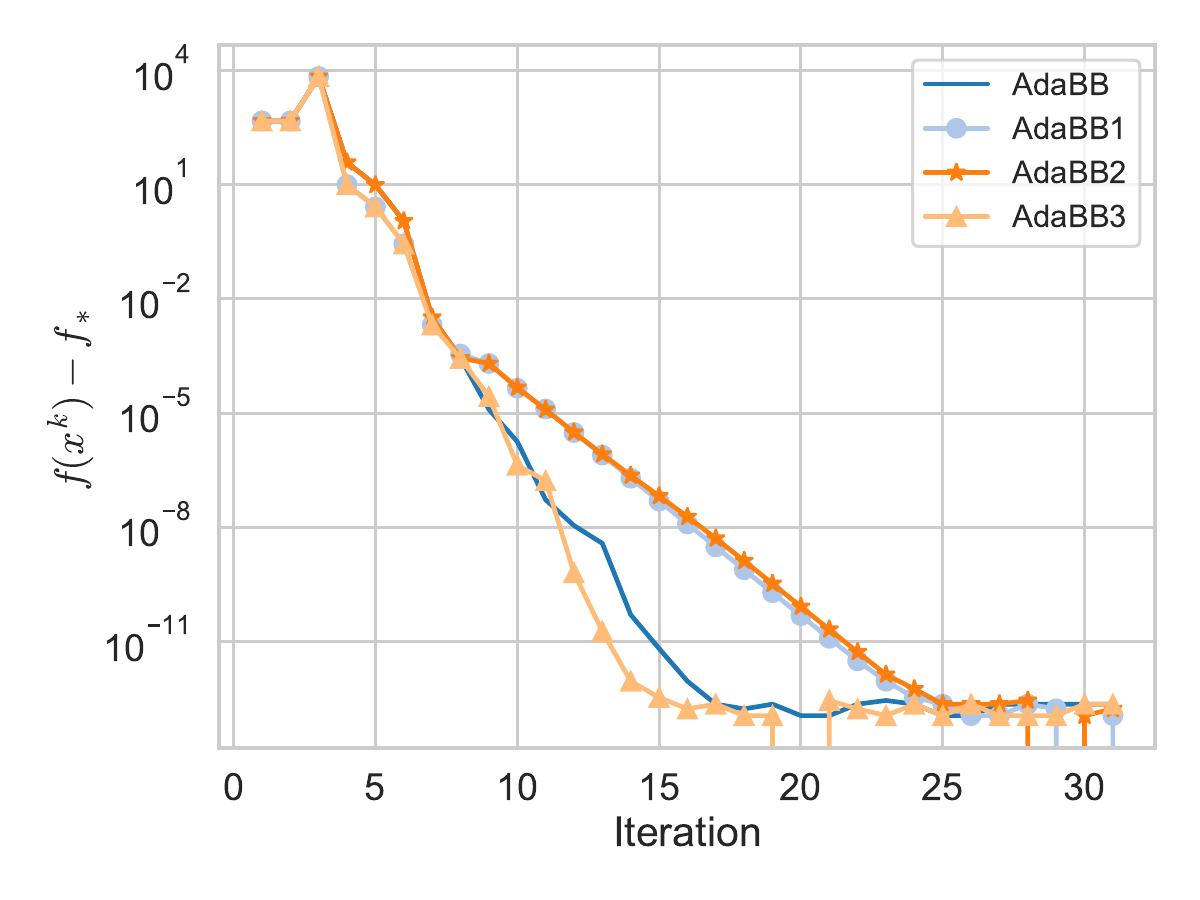}}
\subfloat[w8a, $M=15$, objective]
{\includegraphics[scale = 0.25]{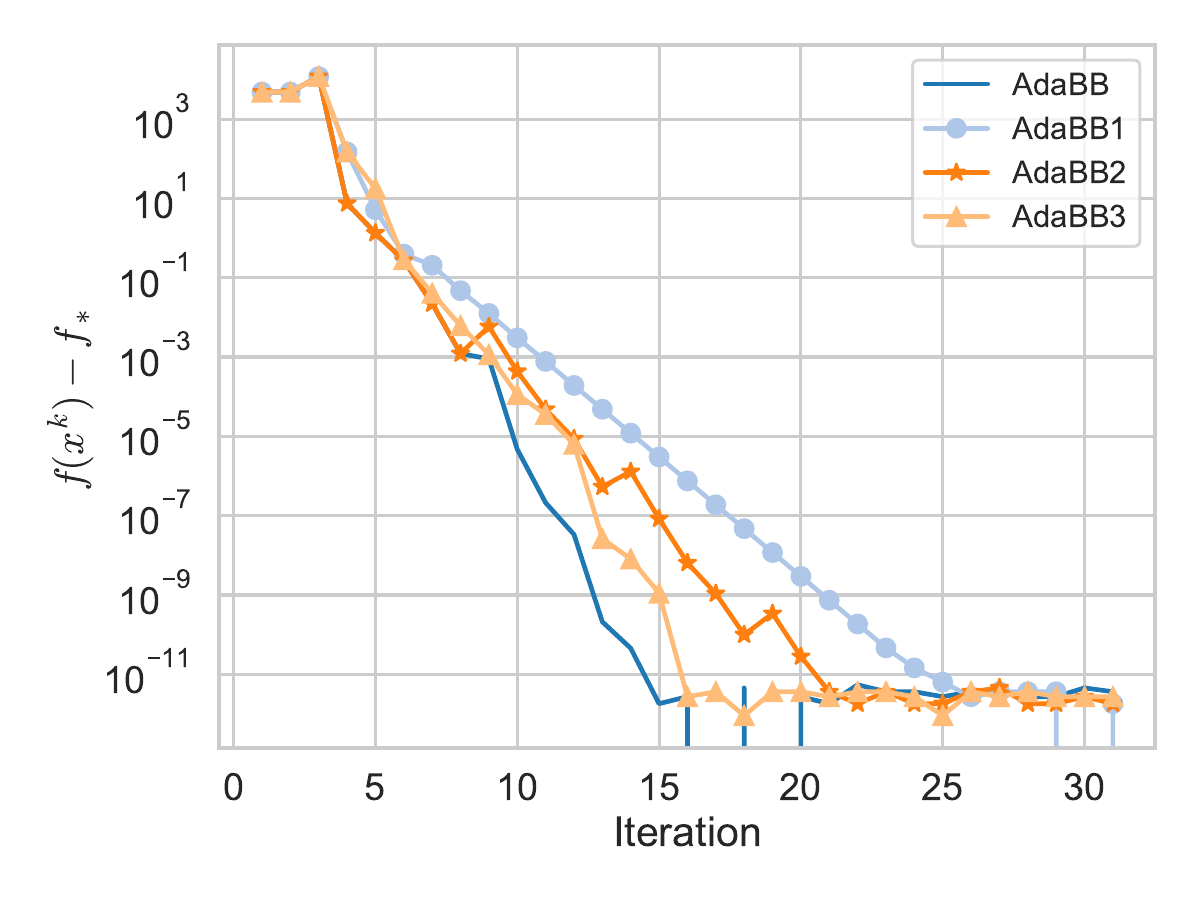}}
\subfloat[covtype, $M=15$, objective]
{\includegraphics[scale = 0.25]{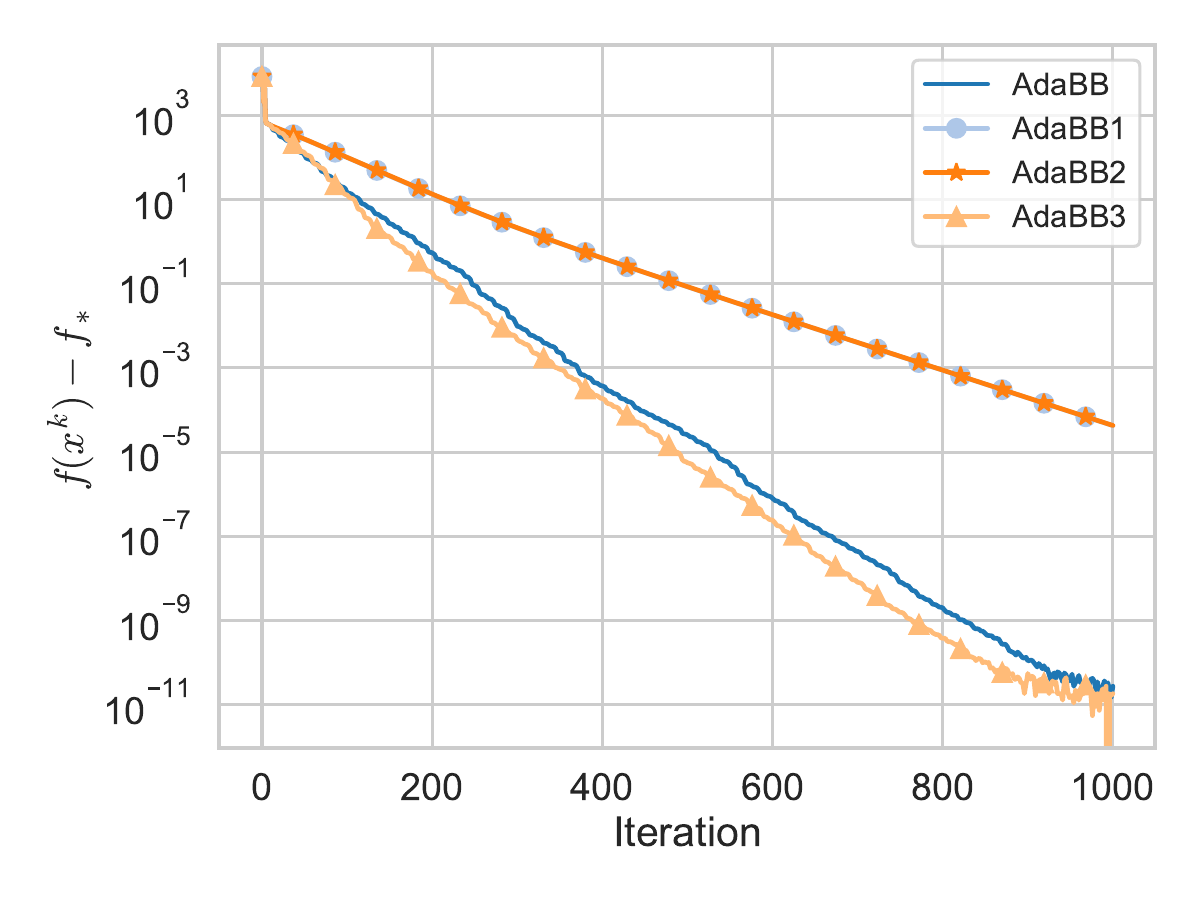}}
\caption{Results for the cubic regulation problem for AdaBB, AdaBB1, AdaBB2, AdaBB3 concerning the function value residual.}
\label{fig:self-2}
\end{figure}

We first compare the four AdaBB variants in Table \ref{tab:4AdaBB}. The results are shown in Figure \ref{fig:self-2}, which again confirms that AdaBB and AdaBB3 are usually better than the other two variants, and thus indicates the effectiveness of opting for Option II in (Case ii). 

\begin{figure}[!htbp]
\centering
\subfloat[mushrooms, $M=10$]{
\includegraphics[scale = 0.2]{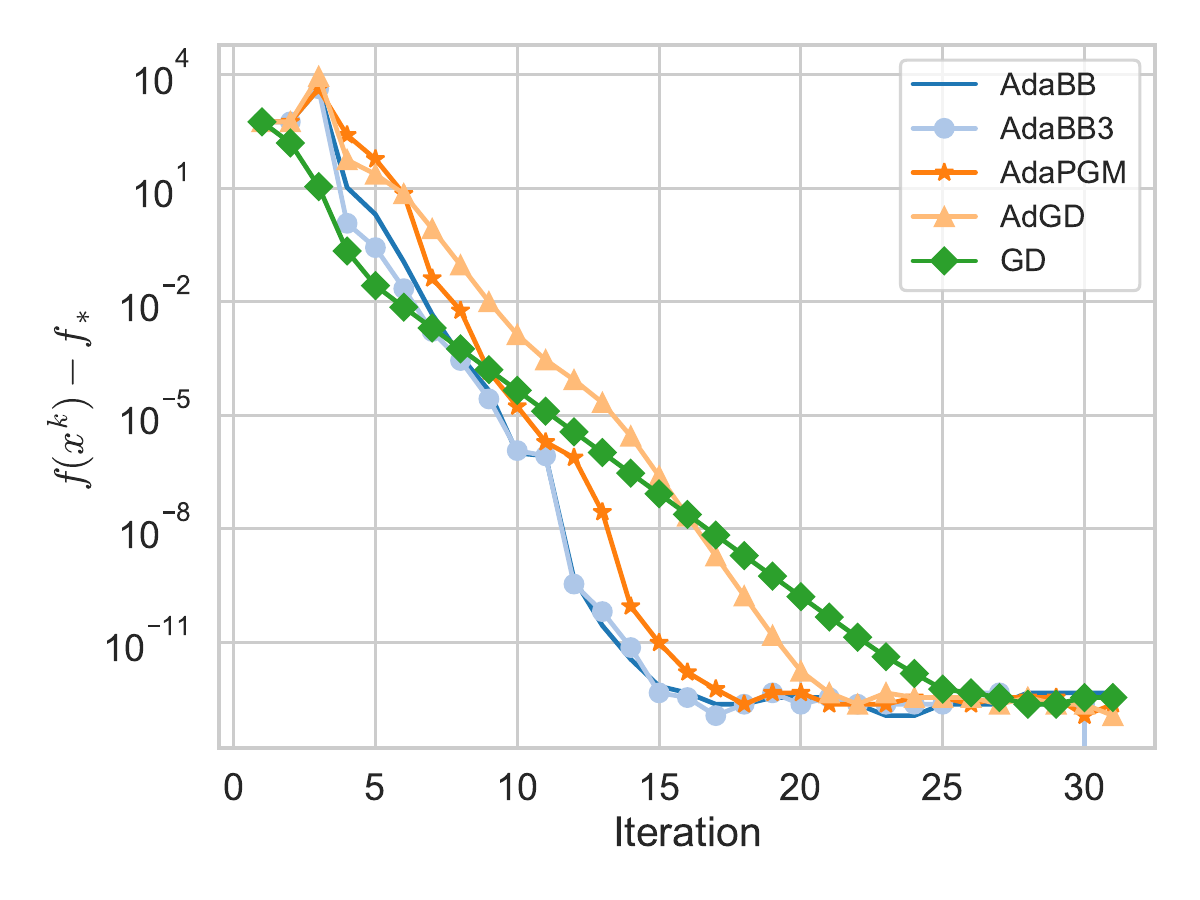}}
\subfloat[mushrooms, $M=15$]
{\includegraphics[scale = 0.2]{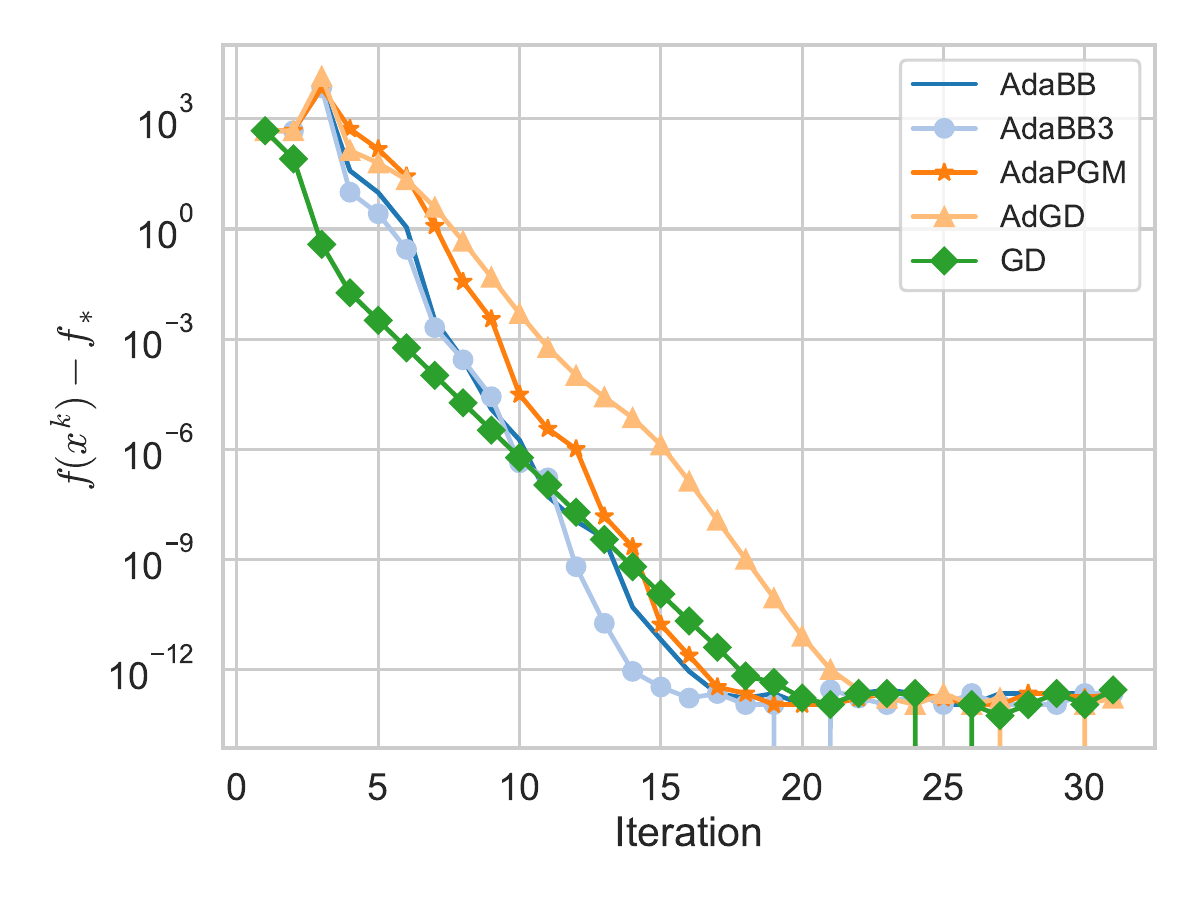}}
\subfloat[mushrooms, $M=10$]
{\includegraphics[scale = 0.2]{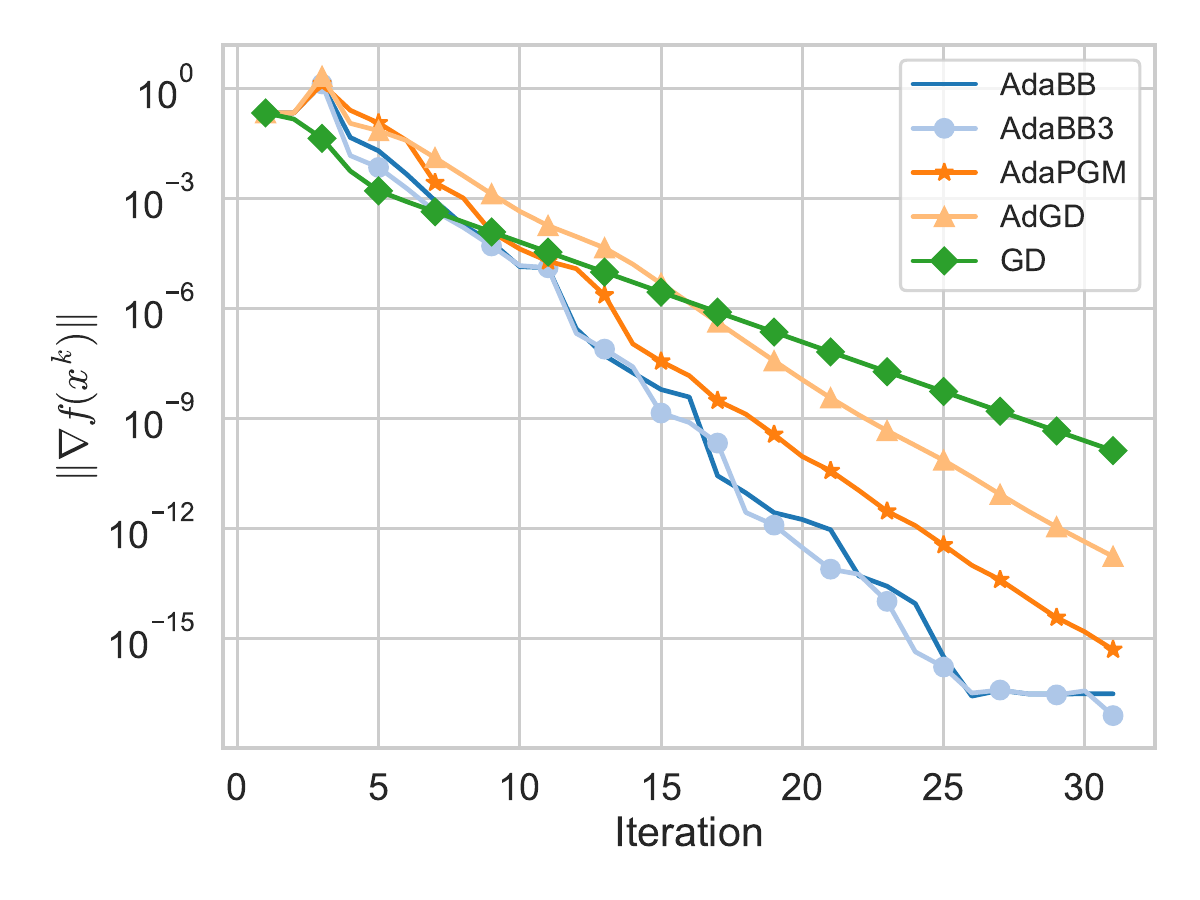}}
\subfloat[mushrooms, $M=15$]
{\includegraphics[scale = 0.2]{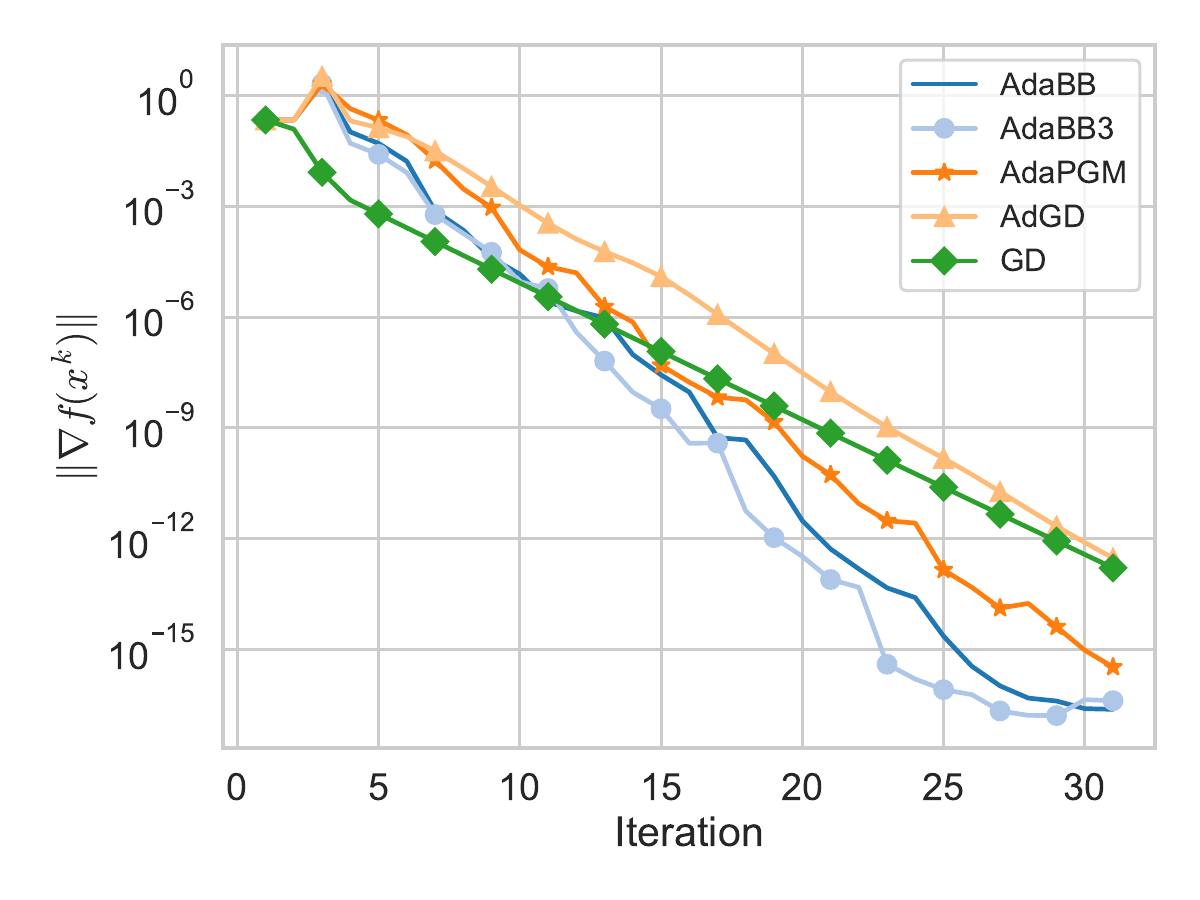}}
\\
\subfloat[w8a, $M=10$]{
\includegraphics[scale = 0.2]{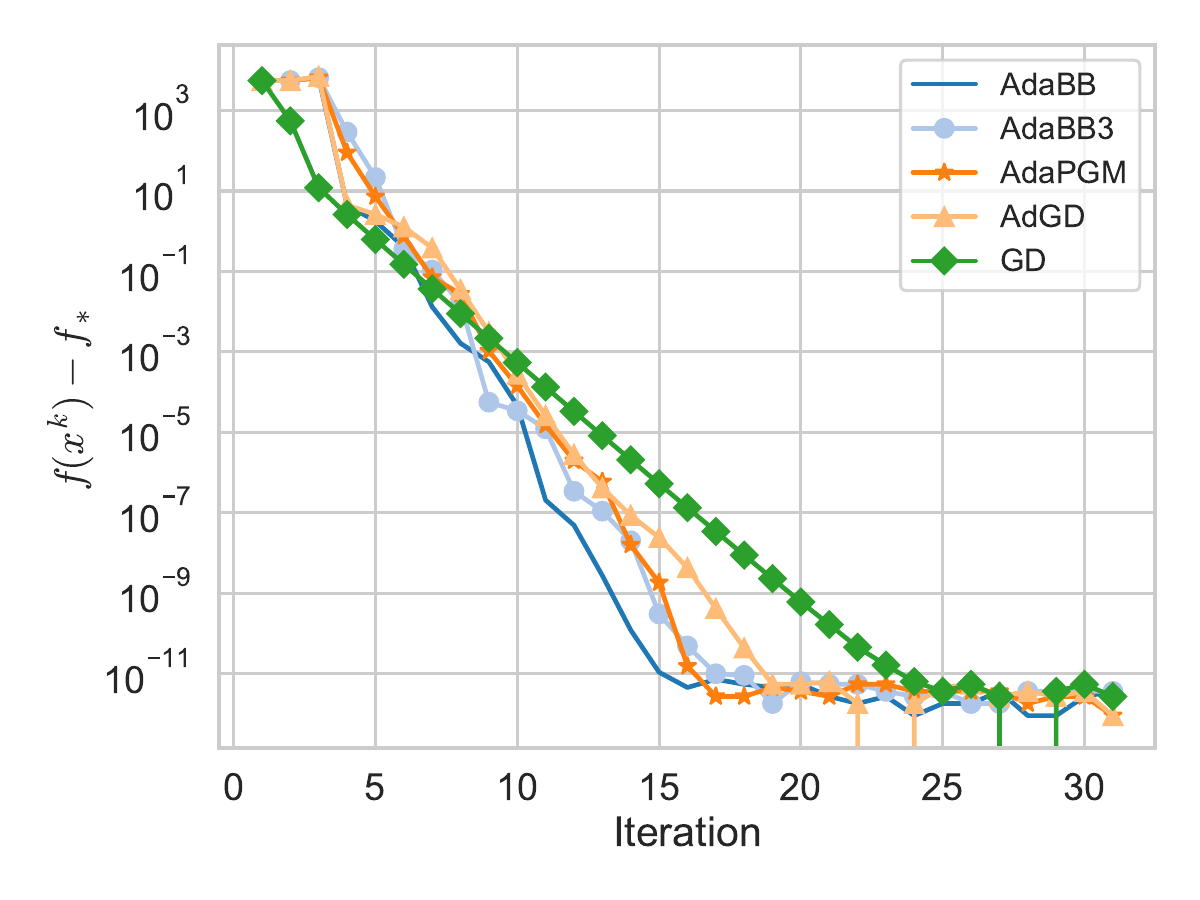}}
\subfloat[w8a, $M=15$]
{\includegraphics[scale = 0.2]{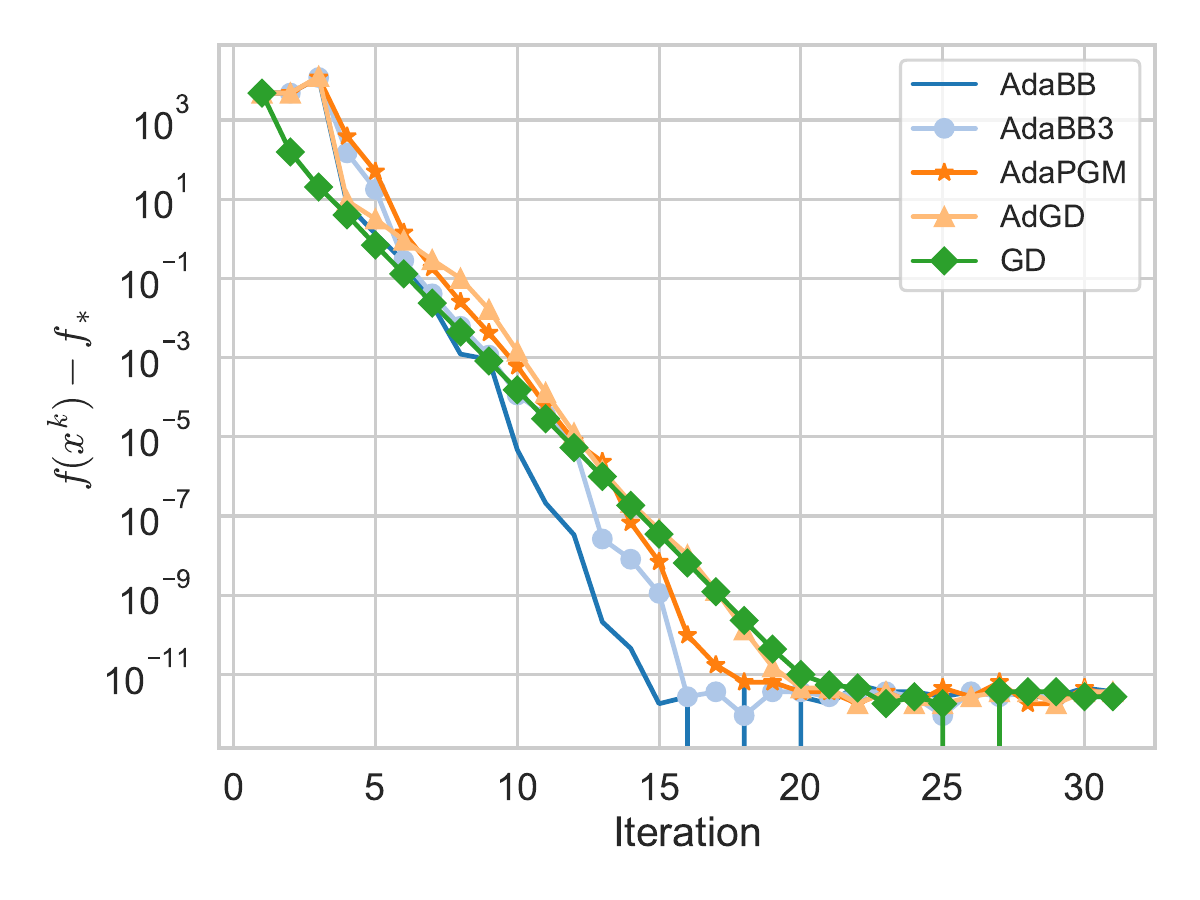}}
\subfloat[w8a, $M=10$]
{\includegraphics[scale = 0.2]{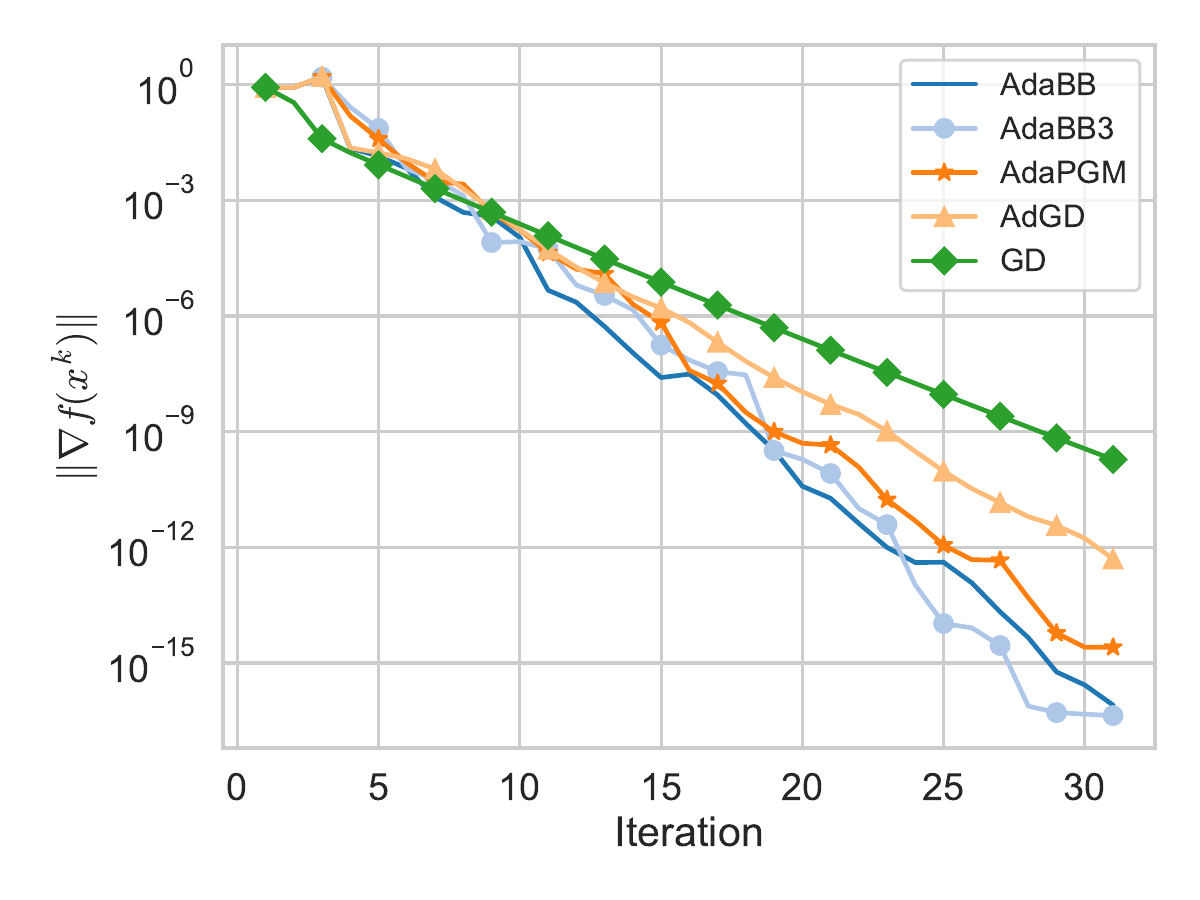}}
\subfloat[w8a, $M=15$]
{\includegraphics[scale = 0.2]{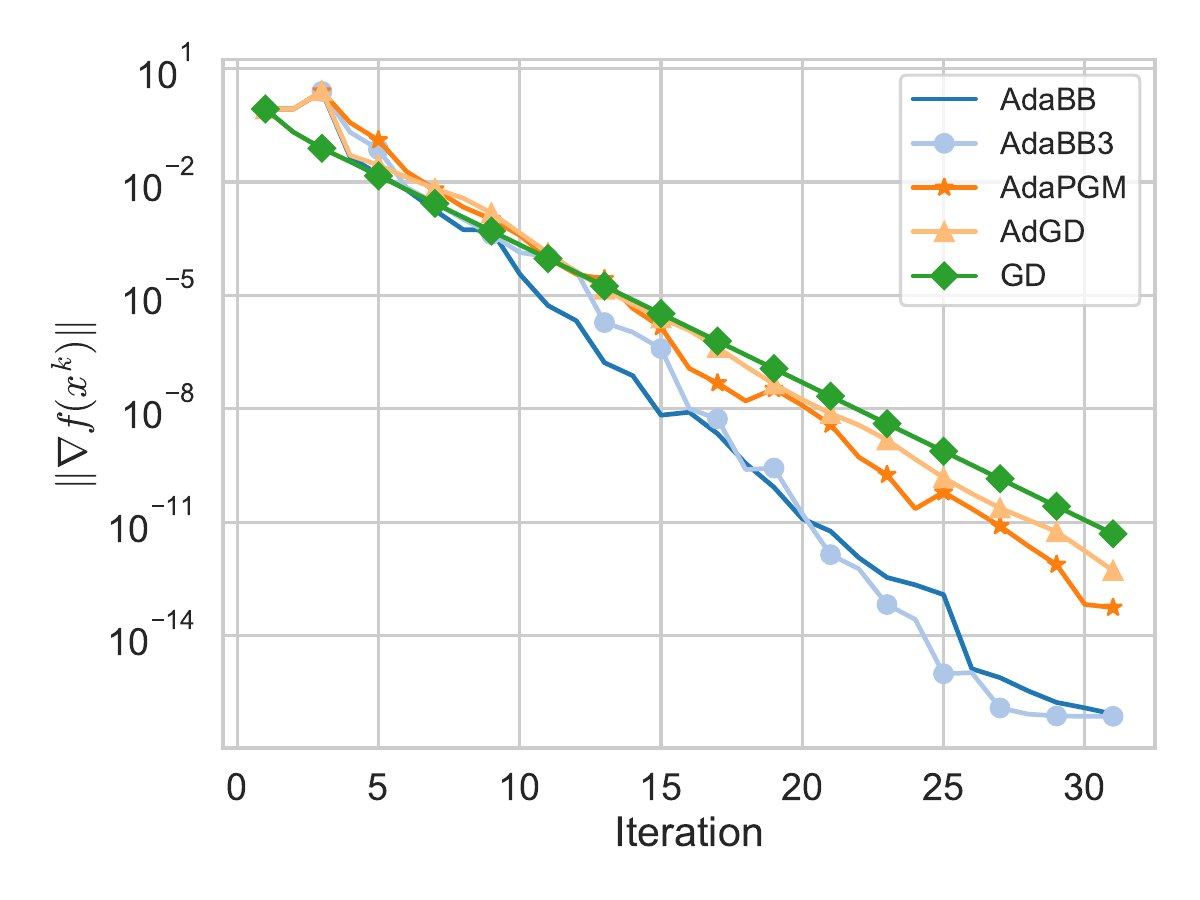}}\\
\subfloat[covtype, $M=10$]{
\includegraphics[scale = 0.2]{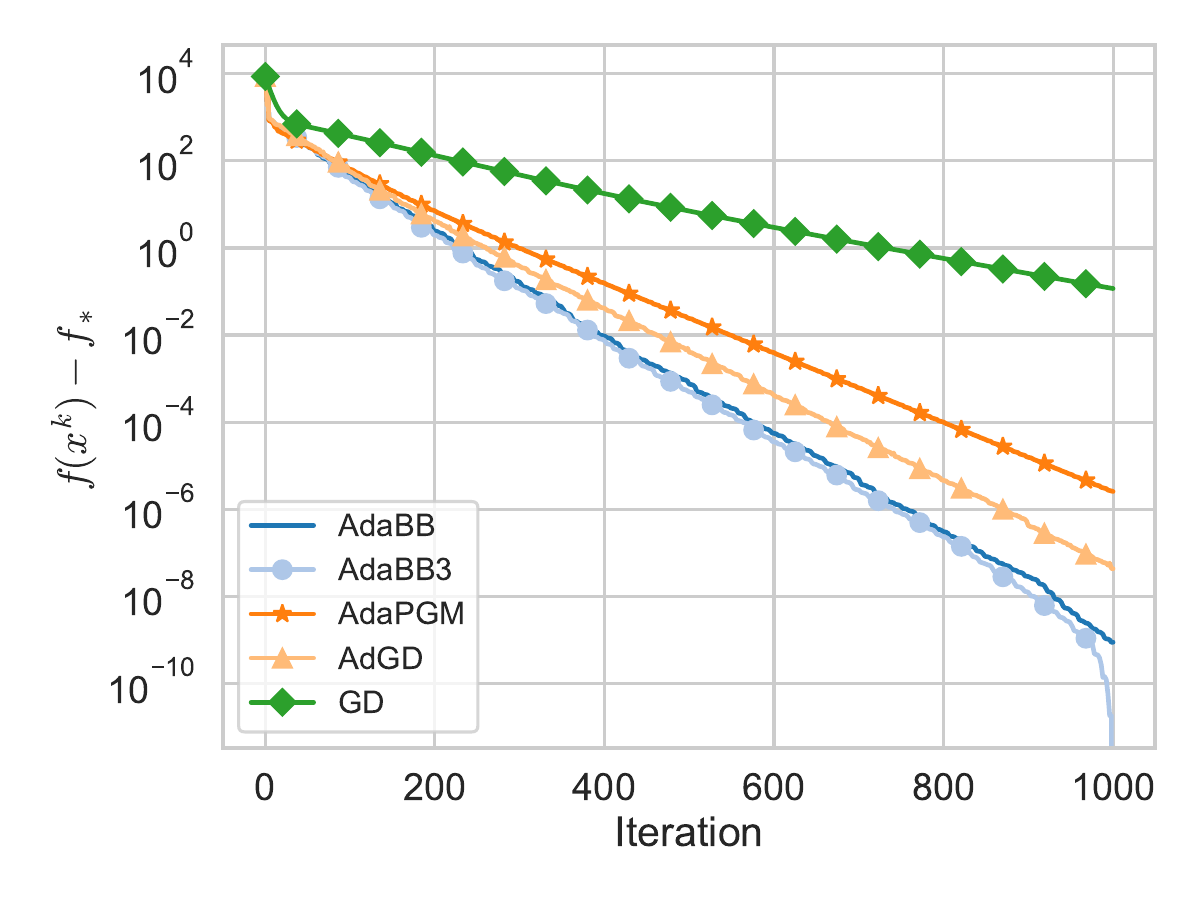}}
\subfloat[covtype, $M=15$]
{\includegraphics[scale = 0.2]{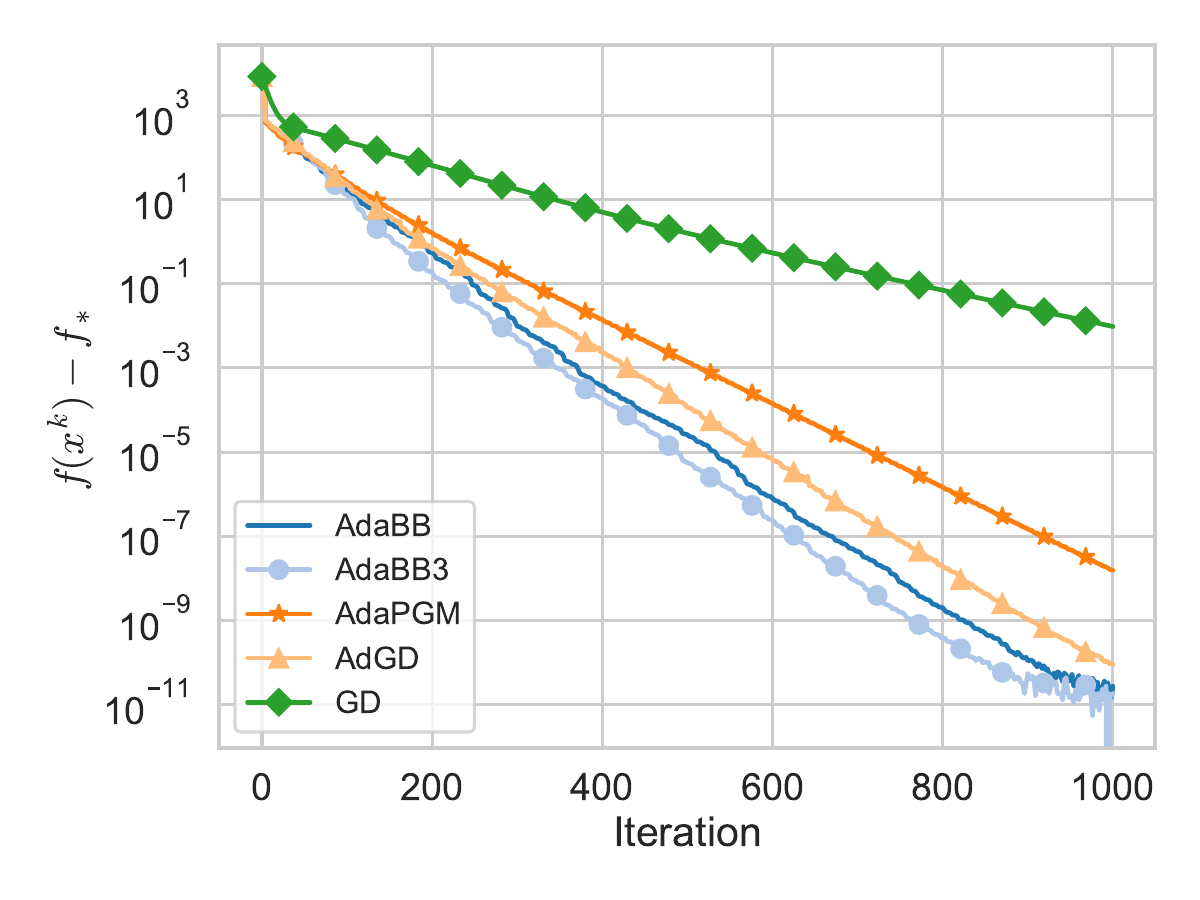}}
\subfloat[covtype, $M=10$]
{\includegraphics[scale = 0.2]{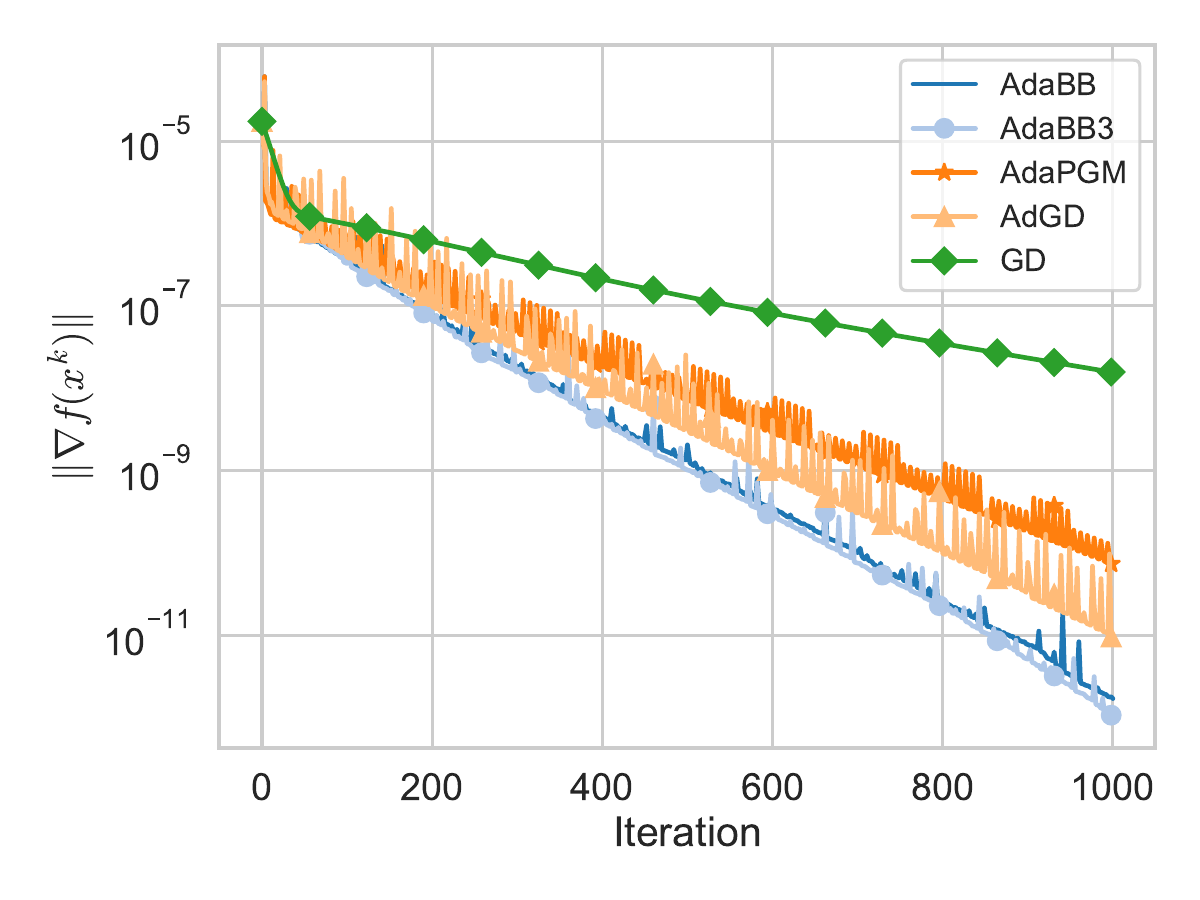}}
\subfloat[covtype, $M=15$]
{\includegraphics[scale = 0.2]{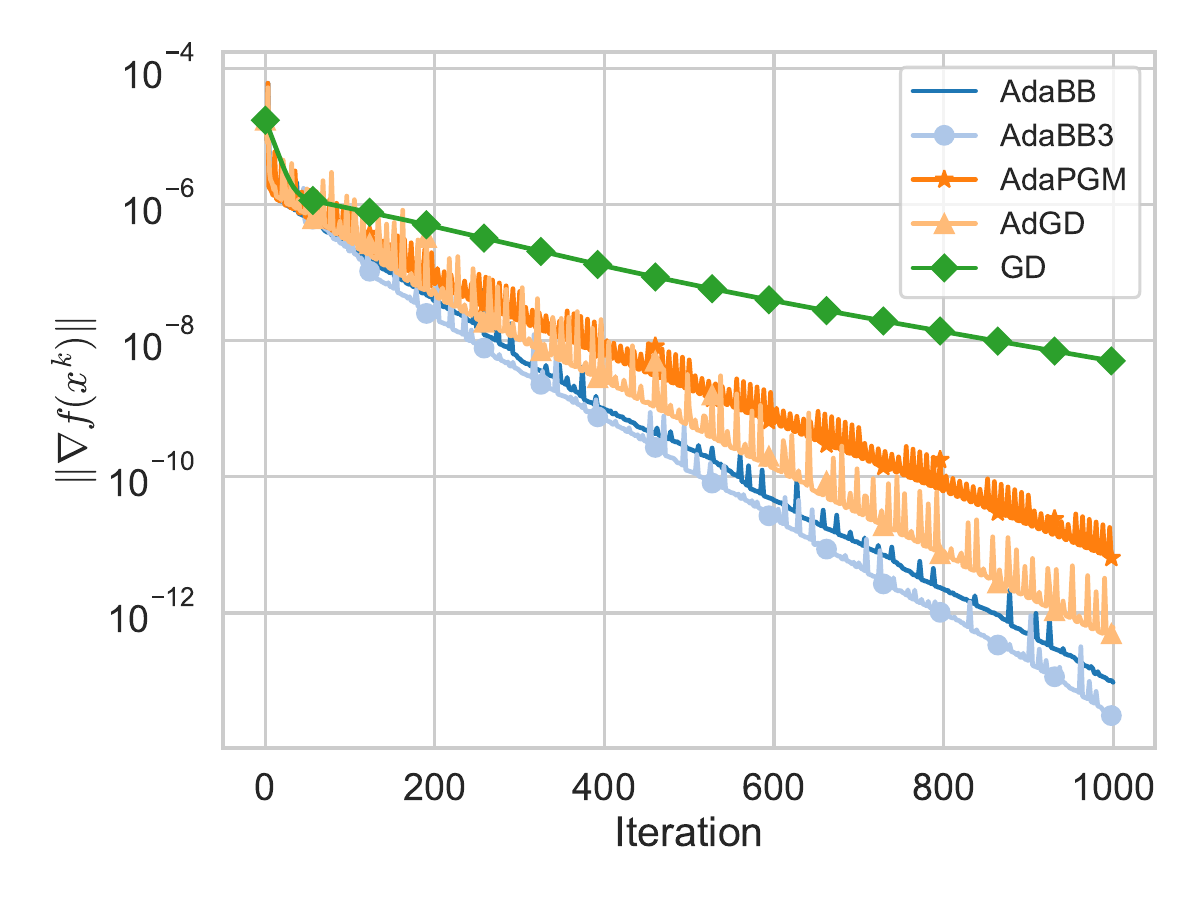}}
\caption{Results for the cubic regulation problem for GD, AdGD, AdaPGM, and AdaBB concerning the function value residual and gradient norm.}
\label{fig:cr}
\end{figure}

In Figure \ref{fig:cr}, we show the comparison of AdaBB and AdaBB3 with GD, AdGD and AdaPGM, from which we see again that AdaBB and AdaBB3 are usually better than the other three algorithms. 
Moreover, we also draw the stepsizes generated by AdGD and AdaBB in Figure \ref{fig:crst}. 

\begin{figure}[!htbp]
\centering
\subfloat[mushrooms, $M=10$, stepsize]{
\includegraphics[scale = 0.25]{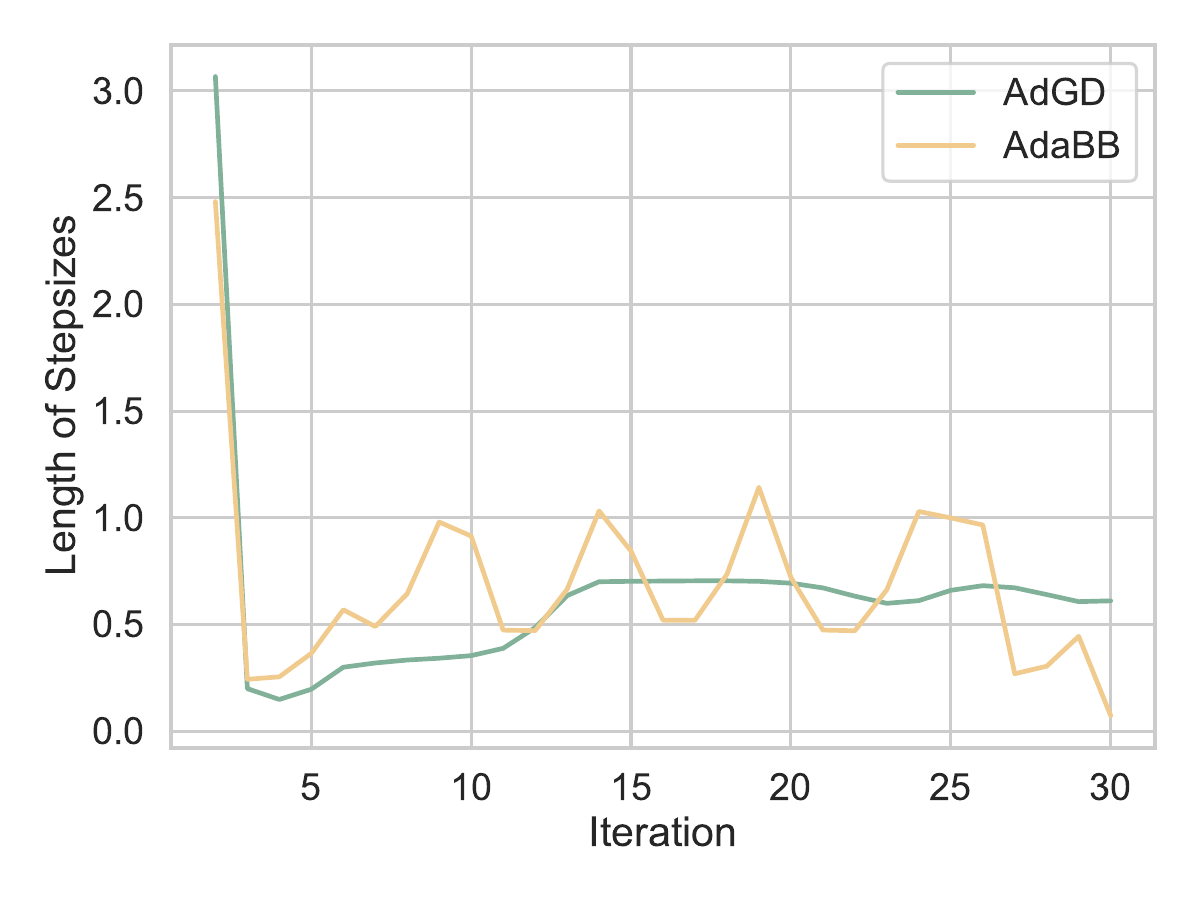}}
\subfloat[w8a, $M=10$, stepsize]{
\includegraphics[scale = 0.25]{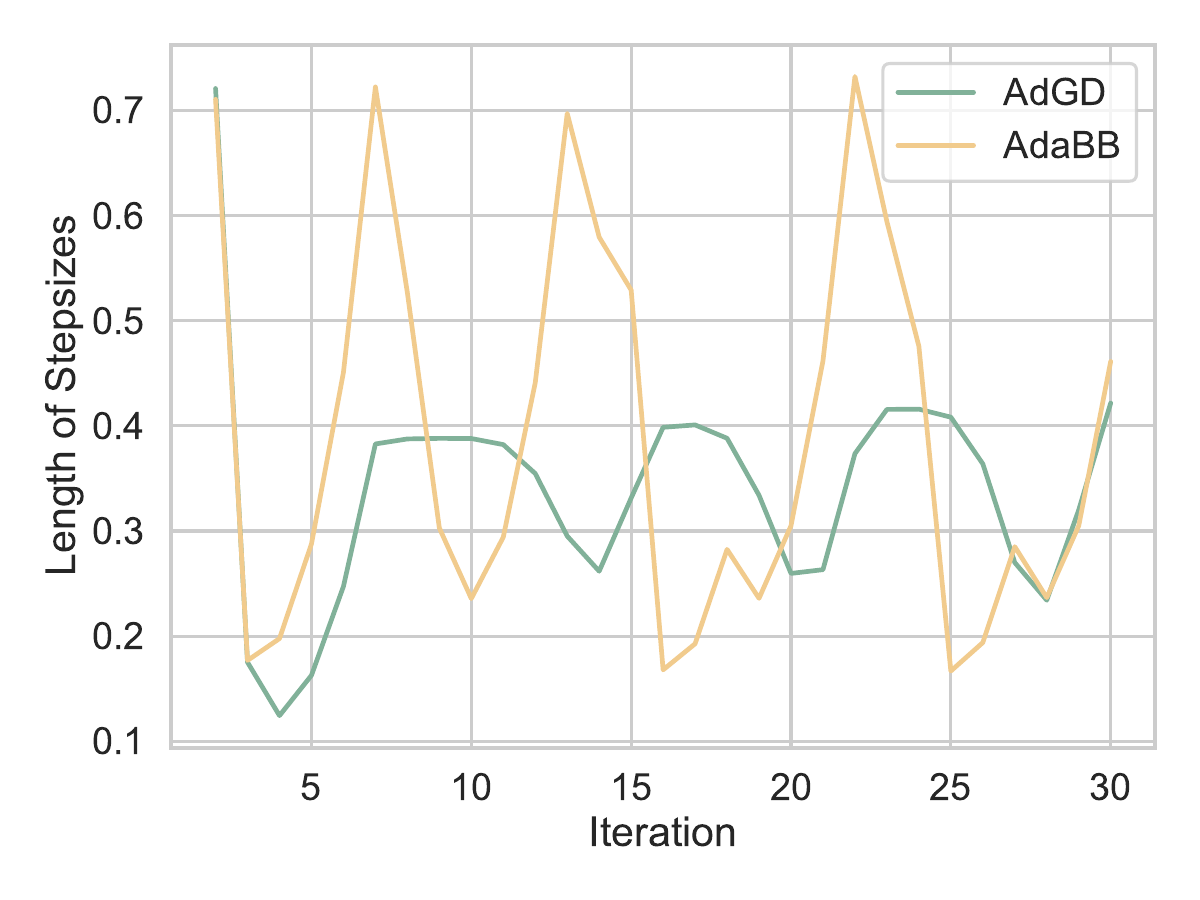}}
\subfloat[covtype, $M=10$, stepsize]
{\includegraphics[scale = 0.25]{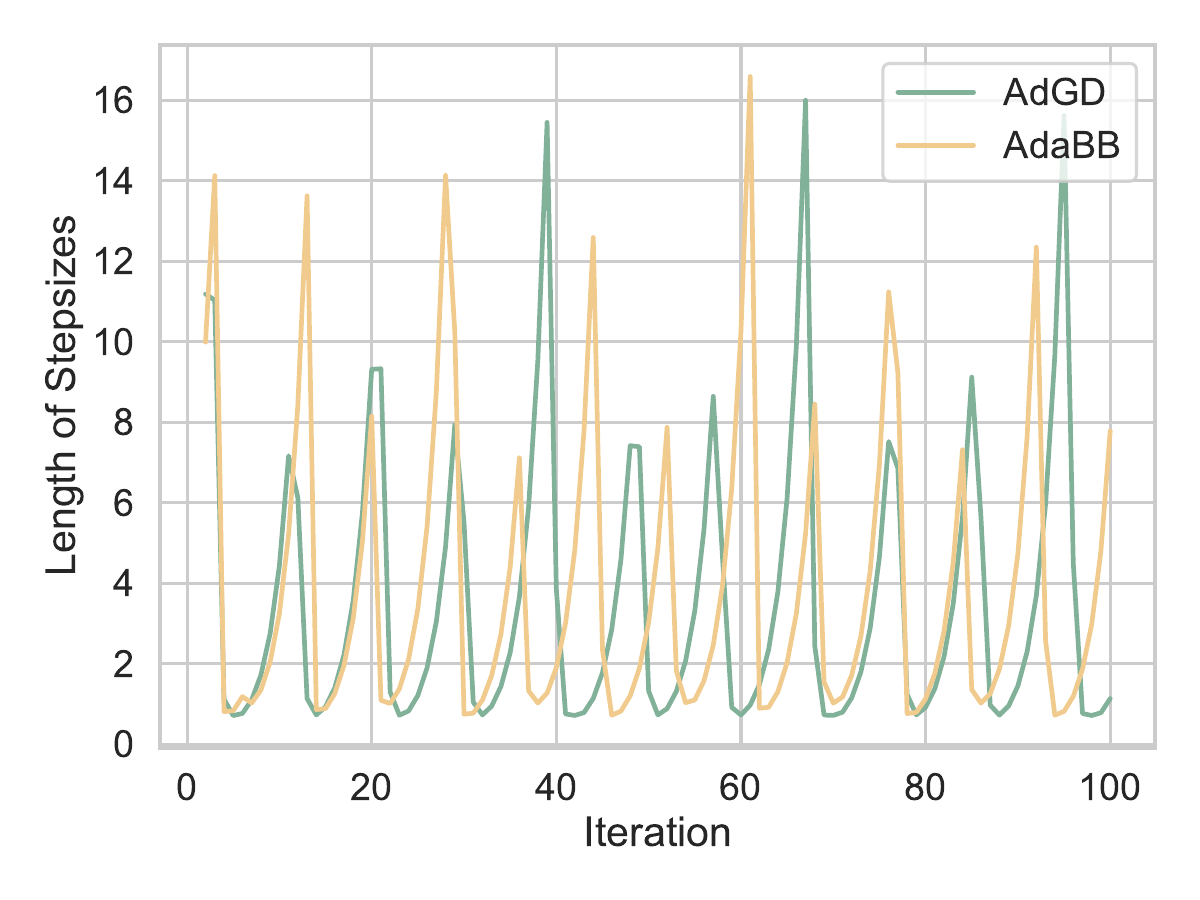}}\\
\subfloat[mushrooms, $M=10$, pattern]{
\includegraphics[scale = 0.25]{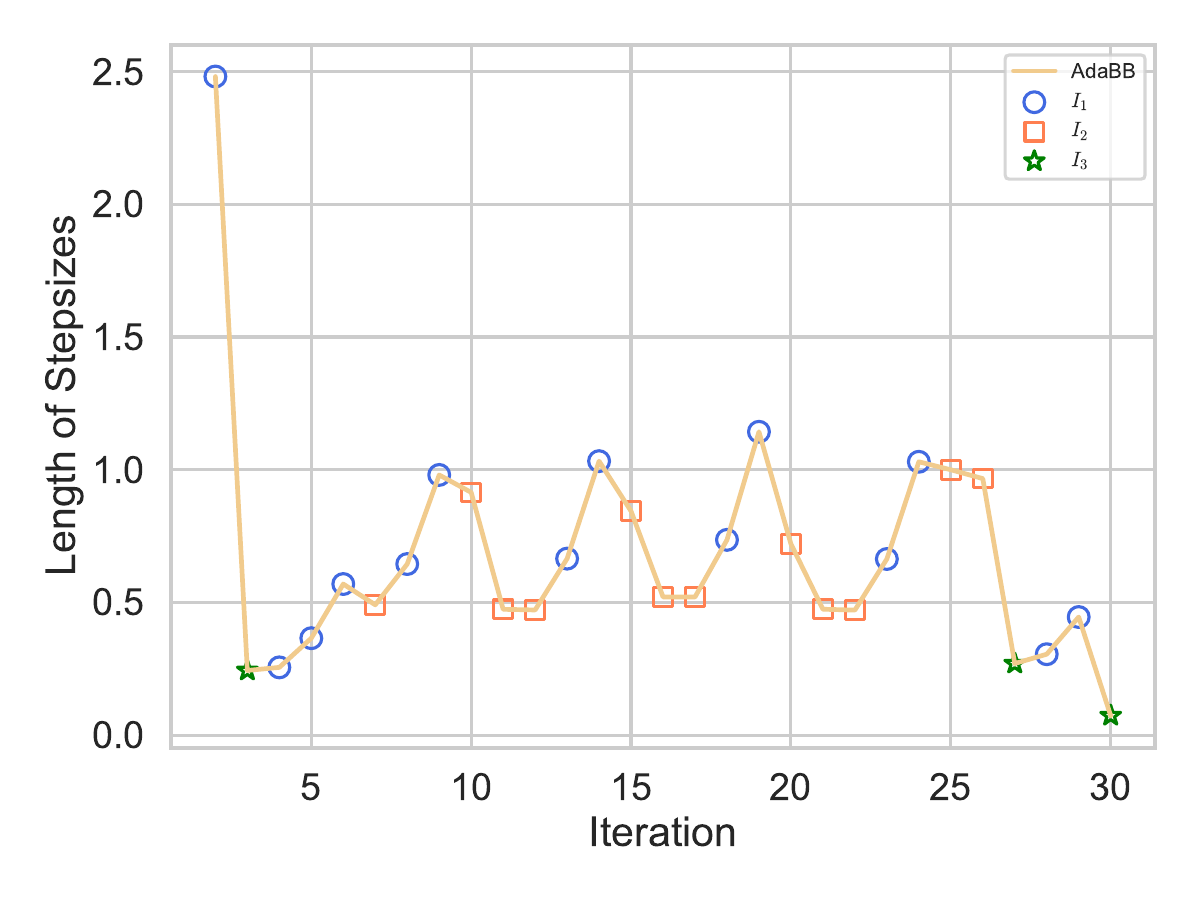}}
\subfloat[w8a, $M=10$, pattern]{
\includegraphics[scale = 0.25]{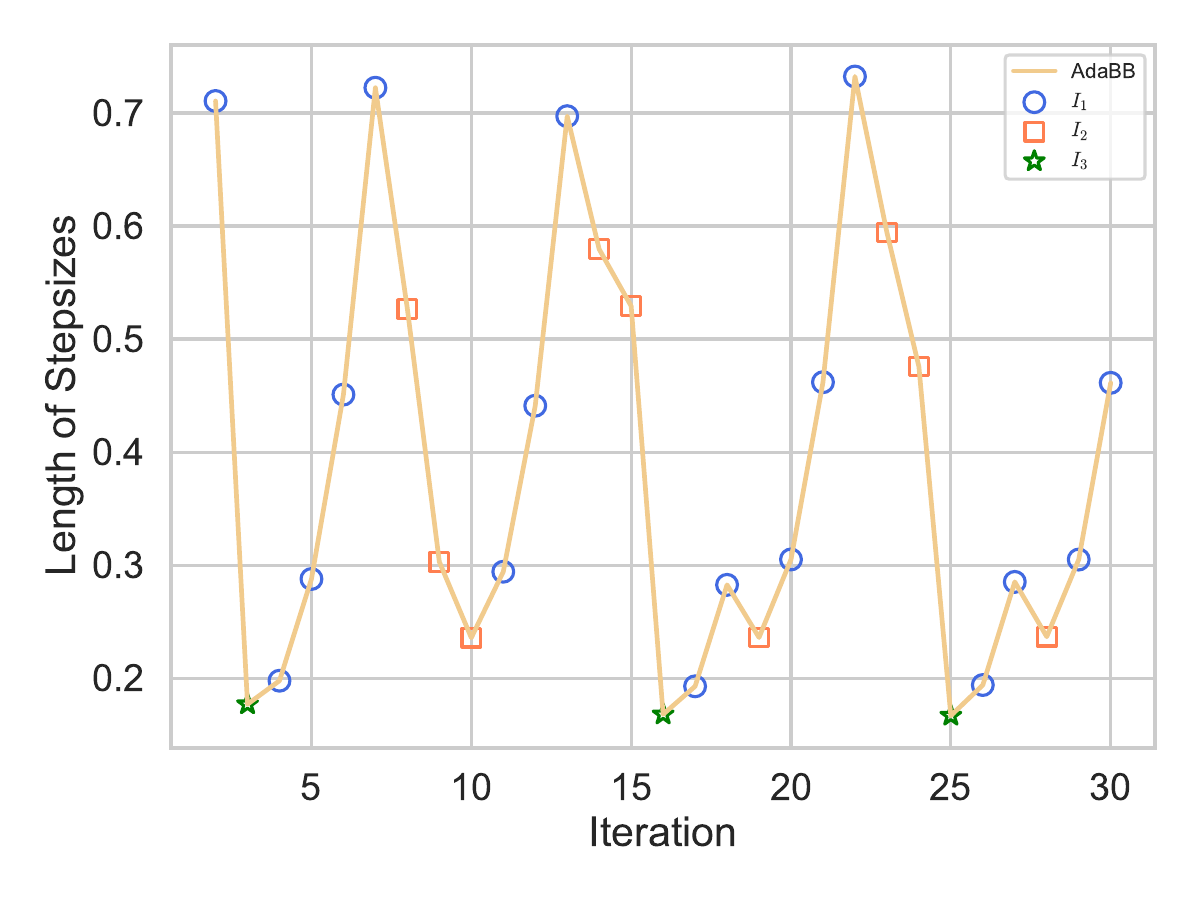}}
\subfloat[covtype, $M=10$, pattern]
{\includegraphics[scale = 0.25]{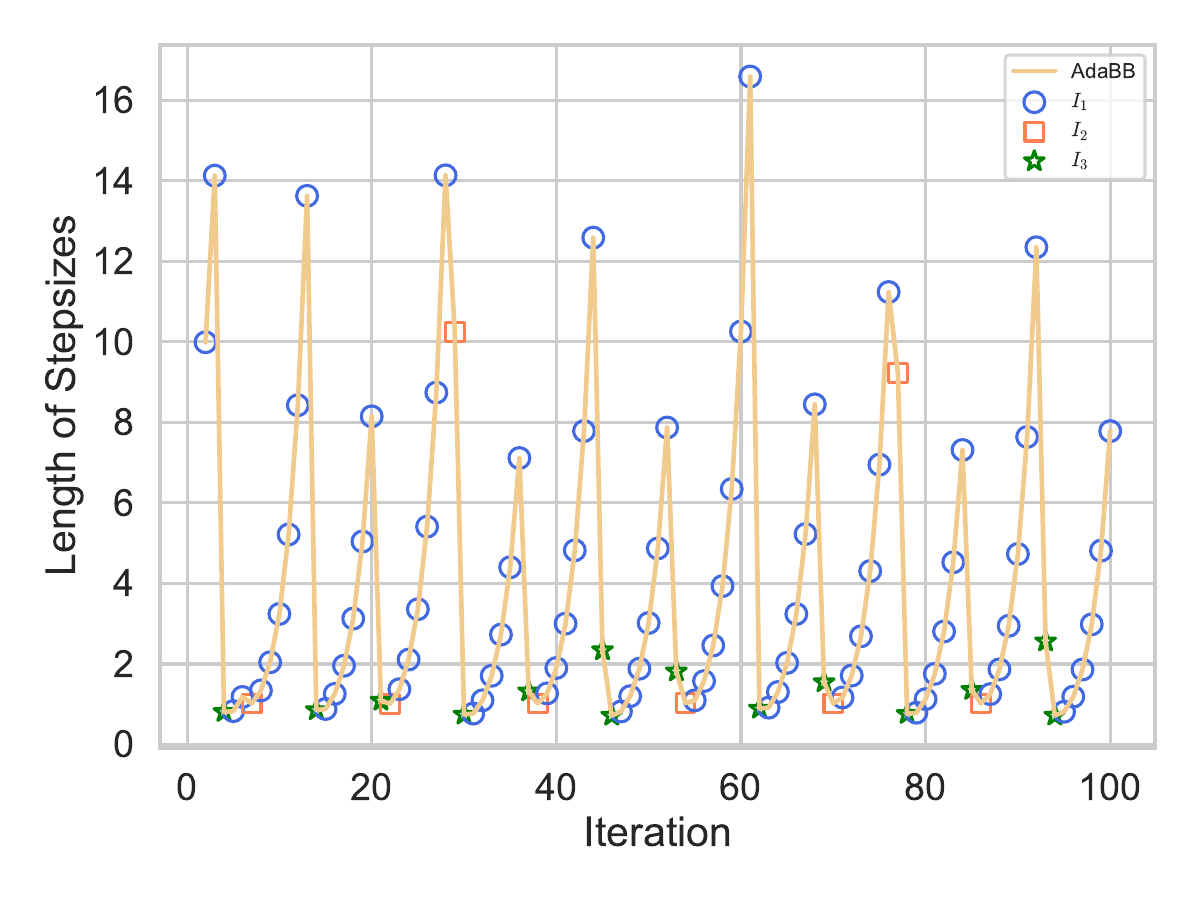}}
\caption{Stepsizes generated by AdGD and AdaBB.}
\label{fig:crst}
\end{figure}
From Figure \ref{fig:crst} we obtain similar observations as the ones in the previous subsection. 

\section{Concluding Remarks} \label{conclusions}

In this paper, we proposed an adaptive BB method for solving unconstrained smooth convex optimization. The proposed AdaBB algorithm is line-search-free and parameter-free. It essentially provides a convergent variant of the BB method for general unconstrained convex optimization. Comparing other adaptive algorithms along the same line of research, our AdaBB achieves the best lower bounds on the stepsize and the average of the stepsizes, which also affirmatively answers an open questions posed by Malitsky and Mishchenko \cite{MM23}. Our numerical results show the superior performance of two versions of AdaBB which takes the BB stepsize directly when it is not too large and not too small. This confirms the great potential of using BB stepsize in practice, under the condition that some safeguard procedure ensuring the convergence is taken, just like our AdaBB algorithm.

\printbibliography

@inproceedings{MM20,
author = {Malitsky, Yura and Mishchenko, Konstantin},
title = {Adaptive Gradient Descent without Descent},
year = {2020},
booktitle = {ICML}
}

@inproceedings{MS10,
author = {H. Brendan McMahan and
Matthew J. Streeter},
title = {Adaptive bound optimization for online convex optimization},
year = {2010},
booktitle = {COLT}
}

@inproceedings{Tan-SVRGBB,
author = {Conghui Tan and Shiqian Ma and Yu-Hong Dai and Yuqiu Qian},
title = {Barzilai-{B}orwein Step Size for Stochastic Gradient Descent},
year = {2016},
booktitle = {NeurIPS}
}

@misc{MM23,
      title={Adaptive Proximal Gradient Method for Convex Optimization}, 
      author={Yura Malitsky and Konstantin Mishchenko},
      year={2023},
      eprint={2308.02261},
      archivePrefix={arXiv},
      primaryClass={math.OC}
}

@article {DT14,
    AUTHOR = {Drori, Yoel and Teboulle, Marc},
     TITLE = {Performance of first-order methods for smooth convex
              minimization: a novel approach},
   JOURNAL = {Math. Program.},
  FJOURNAL = {Mathematical Programming},
    VOLUME = {145},
      YEAR = {2014},
    NUMBER = {1-2},
     PAGES = {451--482},
      ISSN = {0025-5610,1436-4646},
   MRCLASS = {90C60 (68Q25 90C22 90C25)},
  MRNUMBER = {3207695},
MRREVIEWER = {Klaus\ Meer},
       DOI = {10.1007/s10107-013-0653-0},
       URL = {https://doi.org/10.1007/s10107-013-0653-0},
}

@article {TV23,
    AUTHOR = {Teboulle, Marc and Vaisbourd, Yakov},
     TITLE = {An elementary approach to tight worst case complexity analysis
              of gradient based methods},
   JOURNAL = {Math. Program.},
  FJOURNAL = {Mathematical Programming},
    VOLUME = {201},
      YEAR = {2023},
    NUMBER = {1-2},
     PAGES = {63--96},
      ISSN = {0025-5610,1436-4646},
   MRCLASS = {90C25 (90C60)},
  MRNUMBER = {4620224},
       DOI = {10.1007/s10107-022-01899-0},
       URL = {https://doi.org/10.1007/s10107-022-01899-0},
}

@article {BB88,
    AUTHOR = {Barzilai, Jonathan and Borwein, Jonathan M.},
     TITLE = {Two-point step size gradient methods},
   JOURNAL = {IMA J. Numer. Anal.},
  FJOURNAL = {IMA Journal of Numerical Analysis},
    VOLUME = {8},
      YEAR = {1988},
    NUMBER = {1},
     PAGES = {141--148},
      ISSN = {0272-4979},
   MRCLASS = {65K05},
  MRNUMBER = {967848},
       DOI = {10.1093/imanum/8.1.141},
       URL = {https://doi.org/10.1093/imanum/8.1.141},
}

@article {BDH19,
    AUTHOR = {Burdakov, Oleg and Dai, Yu-hong and Huang, Na},
     TITLE = {Stabilized {B}arzilai-{B}orwein method},
   JOURNAL = {J. Comput. Math.},
  FJOURNAL = {Journal of Computational Mathematics},
    VOLUME = {37},
      YEAR = {2019},
    NUMBER = {6},
     PAGES = {916--936},
      ISSN = {0254-9409,1991-7139},
   MRCLASS = {65K05 (90C26 90C53)},
  MRNUMBER = {4038517},
MRREVIEWER = {Douglas\ S.\ Gon\c{c}alves},
       DOI = {10.4208/jcm.1911-m2019-0171},
       URL = {https://doi.org/10.4208/jcm.1911-m2019-0171},
}

@article {R97,
AUTHOR = {Raydan, Marcos},
     TITLE = {The {B}arzilai and {B}orwein gradient method for the large
              scale unconstrained minimization problem},
   JOURNAL = {SIAM J. Optim.},
  FJOURNAL = {SIAM Journal on Optimization},
    VOLUME = {7},
      YEAR = {1997},
    NUMBER = {1},
     PAGES = {26--33},
      ISSN = {1052-6234},
   MRCLASS = {90C30 @article {MR1430555,
    AUTH(65K05)}},
  MRNUMBER = {1430555},
MRREVIEWER = {Nada\ I.\ Djuranovi\'{c}-Mili\v{c}i\v{c}},
       DOI = {10.1137/S1052623494266365},
       URL = {https://doi.org/10.1137/S1052623494266365},
}

@article {GLL86,
    AUTHOR = {Grippo, L. and Lampariello, F. and Lucidi, S.},
     TITLE = {A nonmonotone line search technique for {N}ewton's method},
   JOURNAL = {SIAM J. Numer. Anal.},
  FJOURNAL = {SIAM Journal on Numerical Analysis},
    VOLUME = {23},
      YEAR = {1986},
    NUMBER = {4},
     PAGES = {707--716},
      ISSN = {0036-1429},
   MRCLASS = {90C30 (49D15 65K10)},
  MRNUMBER = {849278},
       DOI = {10.1137/0723046},
       URL = {https://doi.org/10.1137/0723046},
}

@article {DL02,
    AUTHOR = {Dai, Yu-Hong and Liao, Li-Zhi},
     TITLE = {{\bf {R}}-linear convergence of the {B}arzilai and {B}orwein
              gradient method},
   JOURNAL = {IMA J. Numer. Anal.},
  FJOURNAL = {IMA Journal of Numerical Analysis},
    VOLUME = {22},
      YEAR = {2002},
    NUMBER = {1},
     PAGES = {1--10},
      ISSN = {0272-4979,1464-3642},
   MRCLASS = {90C30 (65K05 90C06)},
  MRNUMBER = {1880051},
MRREVIEWER = {Marcos\ Raydan},
       DOI = {10.1093/imanum/22.1.1},
       URL = {https://doi.org/10.1093/imanum/22.1.1},
}

@article {R93,
    AUTHOR = {Raydan, Marcos},
     TITLE = {On the {B}arzilai and {B}orwein choice of steplength for the
              gradient method},
   JOURNAL = {IMA J. Numer. Anal.},
  FJOURNAL = {IMA Journal of Numerical Analysis},
    VOLUME = {13},
      YEAR = {1993},
    NUMBER = {3},
     PAGES = {321--326},
      ISSN = {0272-4979,1464-3642},
   MRCLASS = {90C30 (65K05)},
  MRNUMBER = {1225468},
       DOI = {10.1093/imanum/13.3.321},
       URL = {https://doi.org/10.1093/imanum/13.3.321},
}

@article {A66,
    AUTHOR = {Armijo, Larry},
     TITLE = {Minimization of functions having {L}ipschitz continuous first
              partial derivatives},
   JOURNAL = {Pacific J. Math.},
  FJOURNAL = {Pacific Journal of Mathematics},
    VOLUME = {16},
      YEAR = {1966},
     PAGES = {1--3},
      ISSN = {0030-8730,1945-5844},
   MRCLASS = {65.10 (65.30)},
  MRNUMBER = {191071},
MRREVIEWER = {M.\ Lotkin},
       URL = {http://projecteuclid.org/euclid.pjm/1102995080},
}

@article {NP06,
    AUTHOR = {Nesterov, Yurii and Polyak, B. T.},
     TITLE = {Cubic regularization of {N}ewton method and its global
              performance},
   JOURNAL = {Math. Program.},
  FJOURNAL = {Mathematical Programming. A Publication of the Mathematical
              Programming Society},
    VOLUME = {108},
      YEAR = {2006},
    NUMBER = {1},
     PAGES = {177--205},
      ISSN = {0025-5610,1436-4646},
   MRCLASS = {90C53 (90C30)},
  MRNUMBER = {2229459},
MRREVIEWER = {Michel\ H.\ Geoffroy},
       DOI = {10.1007/s10107-006-0706-8},
       URL = {https://doi.org/10.1007/s10107-006-0706-8},
}

@article {DHS11,
    AUTHOR = {Duchi, John and Hazan, Elad and Singer, Yoram},
     TITLE = {Adaptive subgradient methods for online learning and
              stochastic optimization},
   JOURNAL = {J. Mach. Learn. Res.},
  FJOURNAL = {Journal of Machine Learning Research (JMLR)},
    VOLUME = {12},
      YEAR = {2011},
     PAGES = {2121--2159},
      ISSN = {1532-4435,1533-7928},
   MRCLASS = {68T05 (62G08 90C15 90C25)},
  MRNUMBER = {2825422},
}

@misc{LTSP23,
      title={Adaptive proximal algorithms for convex optimization under local Lipschitz continuity of the gradient}, 
      author={Puya Latafat and Andreas Themelis and Lorenzo Stella and Panagiotis Patrinos},
      year={2023},
      eprint={2301.04431},
      archivePrefix={arXiv},
      primaryClass={math.OC},
}

@misc{lTP23,
      title={On the convergence of adaptive first order methods: proximal gradient and alternating minimization algorithms}, 
      author={Puya Latafat and Andreas Themelis and Panagiotis Patrinos},
      year={2023},
      eprint={2311.18431},
      archivePrefix={arXiv},
      primaryClass={math.OC}
}

@book {N87,
    AUTHOR = {Nesterov, Yurii},
     TITLE = {Introductory lectures on convex optimization},
    SERIES = {Applied Optimization},
    VOLUME = {87},
      NOTE = {A basic course},
 PUBLISHER = {Kluwer Academic Publishers, Boston, MA},
      YEAR = {2004},
     PAGES = {xviii+236},
      ISBN = {1-4020-7553-7},
   MRCLASS = {90-02 (90-01 90C25)},
  MRNUMBER = {2142598},
       DOI = {10.1007/978-1-4419-8853-9},
       URL = {https://doi.org/10.1007/978-1-4419-8853-9},
}

@misc{D18,
      title={On the Properties of Convex Functions over Open Sets}, 
      author={Yoel Drori},
      year={2018},
      eprint={1812.02419},
      archivePrefix={arXiv},
      primaryClass={math.OC}
}

@misc{GSW23,
      title={Accelerated Gradient Descent via Long Steps}, 
      author={Benjamin Grimmer and Kevin Shu and Alex L. Wang},
      year={2023},
      eprint={2309.09961},
      archivePrefix={arXiv},
      primaryClass={math.OC}
}

@misc{G23,
      title={Provably Faster Gradient Descent via Long Steps}, 
      author={Benjamin Grimmer},
      year={2023},
      eprint={2307.06324},
      archivePrefix={arXiv},
      primaryClass={math.OC}
}

@misc{LOZ23,
      title={Optimal and parameter-free gradient minimization methods for convex and nonconvex optimization}, 
      author={Guanghui Lan and Yuyuan Ouyang and Zhe Zhang},
      year={2023},
      eprint={2310.12139},
      archivePrefix={arXiv},
      primaryClass={math.OC}
}

@misc{AP23,
      title={Acceleration by Stepsize Hedging II: Silver Stepsize Schedule for Smooth Convex Optimization}, 
      author={Jason M. Altschuler and Pablo A. Parrilo},
      year={2023},
      eprint={2309.16530},
      archivePrefix={arXiv},
      primaryClass={math.OC}
}

@misc{AP232,
      title={Acceleration by Stepsize Hedging I: Multi-Step Descent and the Silver Stepsize Schedule}, 
      author={Jason M. Altschuler and Pablo A. Parrilo},
      year={2023},
      eprint={2309.07879},
      archivePrefix={arXiv},
      primaryClass={math.OC}
}

@misc{LL23,
      title={A simple uniformly optimal method without line search for convex optimization}, 
      author={Tianjiao Li and Guanghui Lan},
      year={2023},
      eprint={2310.10082},
      archivePrefix={arXiv},
      primaryClass={math.OC}
}

@article{DVR23,
  title={Branch-and-bound performance estimation programming: a unified methodology for constructing optimal optimization methods},
  author={Das Gupta, Shuvomoy and Van Parys, Bart PG and Ryu, Ernest K},
  JOURNAL = {Math. Program.},
  FJOURNAL = {Mathematical Programming},
  pages={1--73},
  year={2023},
  publisher={Springer}
}

@article{CL11,
author = {Chang, Chih-Chung and Lin, Chih-Jen},
title = {LIBSVM: A Library for Support Vector Machines},
year = {2011},
issue_date = {April 2011},
publisher = {Association for Computing Machinery},
address = {New York, NY, USA},
volume = {2},
number = {3},
issn = {2157-6904},
url = {https://doi.org/10.1145/1961189.1961199},
doi = {10.1145/1961189.1961199},
journal = {ACM Trans. Intell. Syst. Technol.},
articleno = {27},
pages = {27}
}

\end{document}